\documentclass[11pt,a4paper]{amsart}
\usepackage{graphicx}
\usepackage{amsmath,amsfonts}
\usepackage{amsthm,amssymb,latexsym}
\usepackage[active]{srcltx}
\usepackage[usenames]{color}
\usepackage[utf8]{inputenc}

\newcommand{\bfs}{\boldsymbol}

\newtheorem{theorem}{Theorem}[section]
\newtheorem{corollary}[theorem]{Corollary}
\newtheorem{lemma}[theorem]{Lemma}

\newtheorem*{fact*}{Fact}
\newtheorem{proposition}[theorem]{Proposition}
\theoremstyle{definition}

\newtheorem{remark}[theorem]{Remark}
\numberwithin{equation}{section}

\newcommand{\N}{\mathbb N}
\newcommand{\Z}{\mathbb Z}
\newcommand{\G}{\mathbb G}

\newcommand{\A}{\mathbb A}
\newcommand{\F}{\mathbb F}
\newcommand{\K}{\mathbb K}
\newcommand{\R}{\mathbb R}
\newcommand{\Pp}{\mathbb P}

\newcommand{\fq}{\F_{\hskip-0.7mm q}}

\newcommand{\fqtwo}{\F_{\hskip-0.7mm q^{2}}}
\newcommand{\fqi}{\F_{\hskip-0.7mm q^i}}
\newcommand{\cfq}{\overline{\F}_{\hskip-0.7mm q}}
\def\ifm#1#2{\relax \ifmmode#1\else#2\fi}

\newcommand{\klk}    {\ifm {,\ldots,} {$,\ldots,$}}

\newcommand{\plp}    {\ifm {+\cdots+} {$+\ldots+$}}

\newcommand{\wt}    {\ifm {{\sf wt}} {{$\sf wt$}}}

\begin{document}
\title[Smooth symmetric systems]{Smooth symmetric systems over a finite field and
applications}
\author[N. Gim\'enez et al.]{
Nardo Gim\'enez${}^{1,2}$
Guillermo Matera${}^{2,3}$,
Mariana P\'erez${}^{1,2}$,
and Melina Privitelli${}^{1,2}$}

\address{${}^{1}$Universidad Nacional de Hurlingham, Instituto de Tecnolog\'ia e
Ingenier\'ia\\ Av. Gdor. Vergara 2222 (B1688GEZ), Villa Tesei,
Buenos Aires, Argentina}
\email{\{nardo.gimenez,mariana.perez,melina.privitelli\}@unahur.edu.ar}
\address{${}^{2}$ National Council of Science and Technology (CONICET),
Ar\-gentina}
\address{${}^{3}$Instituto del Desarrollo Humano,
Universidad Nacional de Gene\-ral Sarmiento, J.M. Guti\'errez 1150
(B1613GSX) Los Polvorines, Buenos Aires, Argentina}
\email{gmatera@ungs.edu.ar}

\thanks{The authors were partially supported by the grants
PIP 11220210100361CO CONICET, PICTO UNAHUR 2019-00012 and UNGS
30/1146.}


\date{\today}%

\begin{abstract}
We study the set of common $\fq$--rational solutions of ``smooth''
systems of multivariate symmetric polynomials with coefficients in a
finite field $\fq$. We show that, under certain conditions, the set
of common solutions of such polynomial systems over the algebraic
closure of $\fq$ has a ``good'' geometric behavior. This allows us
to obtain precise estimates on the corresponding number of common
$\fq$--rational solutions. In the case of hypersurfaces we are able
to improve the results. We illustrate the interest of these
estimates through their application to certain classical
combinatorial problems over finite fields.
\end{abstract}

\maketitle
%
%
\section{Introduction}
Several problems of coding theory, cryptography or combinatorics
require the study of the set of rational solutions of polynomial
systems with coefficients in a finite field $\fq$ on which the
symmetric group of permutations of the coordinates acts. In coding
theory, deep holes in standard Reed--Solomon codes over $\fq$ can be
expressed by the set of zeros with coefficients in $\fq$ of certain
symmetric polynomial (see, e.g., \cite{ChMu07} or \cite{CaMaPr12}).
On the other hand, the study of the set of $\fq$--rational zeros of
a certain class of symmetric polynomials is fundamental for the
decoding algorithm for the standard Reed--Solomon code over $\fq$ of
\cite{Sidelnikov94}. In cryptography, the characterization of
monomials defining an almost perfect nonlinear polynomial or a
differentially uniform mapping can be reduced to estimate the number
of $\fq$--rational zeros of certain symmetric polynomials (see,
e.g., \cite{Rodier09} or \cite{AuRo09}). In \cite{GoMa15}, an
optimal representation for the set of $\fq$--rational points of the
trace--zero variety of an elliptic curve defined over $\fq$ has been
obtained using symmetric polynomials. Finally, there are also
important questions in the setting of combinatorics over finite
fields that have been expressed in terms of the number of common
$\fq$--rational zeros of symmetric polynomials defined over $\fq$.
For example, the determination of the average cardinality of the
value set of families of univariate polynomials with coefficients in
$\fq$ (see \cite{CeMaPePr14}, \cite{MaPePr14} and \cite{MaPePr16b})
and the distribution of factorization patterns of families of
univariate polynomials with coefficients in $\fq$ (see
\cite{CeMaPe17} and \cite{MaPePr20}) have been expressed in these
terms.

An outcome of the approach of \cite{CaMaPr12}, \cite{CeMaPePr14},
\cite{CeMaPe17} and \cite{MaPePr20} is a methodology to deal with
some of the problems mentioned above. This methodology relies on the
study of the geometry of the set of common zeros of the symmetric
polynomials under consideration over the algebraic closure $\cfq$ of
$\fq$. In all the cases above we have been able to prove that the
set of common zeros in $\cfq$ of the corresponding symmetric
polynomials is a complete intersection, whose singular locus has a
nontrivial codimension. This has allowed us to apply certain
explicit estimates on the number of $\fq$--rational points of
projective complete intersections defined over $\fq$ (see, e.g.,
\cite{GhLa02a}, \cite{CaMa07}, \cite{CaMaPr15} or \cite{MaPePr16})
to reach to a conclusion for the problem under consideration.

The analysis of \cite{CaMaPr12}, \cite{CeMaPePr14}, \cite{CeMaPe17}
and \cite{MaPePr20} has several points in common, which may be put
on a common basis that might be useful for other problems. For this
reason, in this paper we present a framework where a systematic
study of the geometry of varieties defined by symmetric polynomials
with coefficients in $\fq$ can be carried out. More precisely, for a
positive integer $m$, let $X_1 \klk X_m$ be indeterminates over
$\fq$. For $s<m$, we shall be interested in systems
$F_1=0,\ldots,F_s=0$ defined by polynomials
$F_1,\ldots,F_s\in\fq[X_1,\ldots,X_m]$ which are symmetric in
$X_1,\ldots,X_m$. We shall express these polynomials in terms of the
elementary symmetric polynomials $\Pi_1 \klk \Pi_m$ of $\fq[X_1 \klk
X_m]$, namely
$ F_i=G_i(\Pi_1,\ldots,\Pi_m)$
for $1\le i\le s$, where $G_1\ldots,G_s\in\fq[Y_1,\ldots,Y_m]$ and
$Y_1,\ldots,Y_m$ are new indeterminates over $\fq$. In particular,
we shall study systems $F_1=0,\ldots,F_s=0$ as above where
$F_1,\ldots,F_s$ do not depend on the last $k$ elementary symmetric
polynomials for a given $k\ge 2$, that is, where each $F_i$ can be
expressed as
\begin{equation}\label{eq: def F_s intro}
F_i=G_i(\Pi_1,\ldots,\Pi_{m-k})\quad(1\le i\le s).
\end{equation}
We shall further require that $G_1,\ldots,G_s$ define a ``smooth''
variety, namely
\begin{itemize}
\item[$({\sf A}_1)$]  $(\partial\bfs{G}/
\partial \bfs{Y})$ has full rank on every point of the affine variety
$V(G_1\klk G_s)\subset\A^m$,
\end{itemize}
where $(\partial\bfs{G}/
\partial \bfs{Y})$ is the Jacobian matrix of
$G_1\klk G_s$ with respect to $Y_1\klk Y_m$ and $\A^m$ denotes the
$m$-dimensional affine space $\cfq{\!}^m$.
Finally, we have also to consider the behavior of solutions ``at
infinity''. For this purpose, we introduce the weight
$\wt:\fq[Y_1,\ldots,Y_{m-k}]\to\N$ defined by setting $\wt(Y_j):=j$
for $1\leq j \leq m-k$ and denote by $G_1^{\wt}\klk G_s^{\wt}$ the
components of highest weight of $G_1\klk G_s$. Let
$(\partial\bfs{G}^{\wt}/
\partial \bfs{Y})$ be the Jacobian matrix of
$G_1^{\wt}\klk G_s^{\wt}$ with respect to $Y_1\klk Y_m$. Our last
assumption is the following:
\begin{itemize}
\item[$({\sf A}_2)$]  $(\partial\bfs{G}^{\wt}/
\partial \bfs{Y})$ has full rank on every point of the affine variety $V(G_1^{\wt}\klk
G_s^{\wt})\subset\A^m$.
\end{itemize}
We shall establish a number of facts concerning the geometry of the
affine variety $V:=V(F_1\klk F_s)\subset\A^m$ which will allow us to
obtain precise estimate on the number of $\fq$--rational points of
$V$. More precisely, we have the following result.
\begin{theorem}\label{th: main result - intro}
Suppose that $\mathrm{char}(\fq)$ does not divide
$m(m-1)\cdots(m-k+1)$. Let $F_1,\ldots,F_s$ be polynomials as in
\eqref{eq: def F_s intro}, satisfying assumptions $({\sf A_1})$ and
$({\sf A_2})$, and let $V:=V(F_1\klk F_s)\subset\A^m$. Let
$d_i:=\deg F_i$ for $1\le i\le s$, $\delta:=\prod_{i=1}^sd_i$,
$D:=\sum_{i=1}^s(d_i-1)$ and $d:=\max\{d_1,\ldots,d_s\}$. If
$|V(\fq)|$ is the cardinality of the set $V(\fq)$ of $\fq$--rational
points of $V$, then the following estimates hold:
\begin{align*}
\big||V(\fq)|-q^{m-s}\big|\le
\left\{\begin{array}{l}q^{m-s-\frac{1}{2}}(1+q^{-1})\big((\delta(D-2)+2)+14D^2\delta^2q^{-\frac{1}{2}}\big)\
\textrm{
for }k=2,\\[1ex]
q^{m-s-1}(1+q^{-1})14D^3\delta^2\ \textrm{ for
}k=3.\end{array}\right.
\end{align*}
Further, for $2\le k<m-s$ we have
\begin{align*}
\big||V(\fq)|-q^{m-s}\big|&\le
q^{m-s-\frac{k-1}{2}}(1+q^{-1})\bigg({m+1\choose s+1}(d+1)^m+9\cdot
2^s(sd+3)^{m+1} q^{-\frac{1}{2}}\bigg).
\end{align*}
\end{theorem}

In the case of hypersurfaces, namely for $s=1$, we are able to
improve Theorem \ref{th: main result - intro}, relaxing somewhat the
conditions under which we have precise estimates on the number of
$\fq$-rational zeros of the polynomial under consideration. More
precisely, we have the following result.
\begin{theorem}\label{th: estimate |V_F| intro}
Suppose that $\mathrm{char}(\fq)$ does not divide
$m(m-1)\cdots(m-k)$. Let $F$ be a polynomial of the form
$F=G(\Pi_1,\ldots,\Pi_{m-k+1})$, satisfying assumptions $({\sf
A_1})$ and $({\sf A_2})$, and let $V_F:=V(F)\subset\A^m$. Let
$d:=\deg F$. We have
$$
\big||V_F(\fq)|-q^{m-1}\big|\leq
\left\{\begin{array}{c}
q^{m-\frac{3}{2}}(1+q^{-1})\big((d-1)(d-2)+14(d-1)^2d^2q^{-\frac{1}{2}}\big)
\textrm{ for }k=2,\\[1ex]
q^{m-2}(1+q^{-1})14(d-1)^3d^2\textrm{ for }k=3,
\\[1ex]
q^{m-\frac{k+1}{2}}(1+q^{-1})\big((d-1)^k+
6(d+2)^{m+1}q^{-\frac{1}{2}}\big)\textrm{ for }2\le k<m-1.
\end{array}\right.
$$
\end{theorem}

There is a natural correspondence between orbits (under the action
of the group of permutations of coordinates) of points of varieties
defined by systems of symmetric polynomials with coefficients in
$\cfq$  and univariate polynomials with coefficients in $\cfq$. More
precisely, let $\bfs x\in\A^m$ be a common zero of polynomials
$F_1,\ldots,F_s$ as in \eqref{eq: def F_s intro}. Then the map $\bfs
x\mapsto T^m+\sum_{j=1}^m(-1)^j\Pi_j(\bfs x)T^{m-j}$ sends the orbit
of $\bfs x$ to the monic polynomial of $\cfq[T]$ of degree $m$
having the coordinates of $\bfs x$ as its roots over $\cfq$
(counting multiplicities). In this way, questions on the geometry of
varieties defined by symmetric polynomials can be expressed as
problems concerning the structure of families of univariate
polynomials in $\cfq[T]$, and viceversa. A critical point in our
approach is the use of this correspondence to establish properties
on the geometry of the variety $V\subset\A^m$ defined by polynomials
$F_1,\ldots,F_s$ as in \eqref{eq: def F_s intro}. In particular, we
study the variety defined by the vanishing of successive generalized
subdiscriminants. This variety has been the subject of, e.g.,
\cite{Batkhin16} and \cite{Batkhin18}. In Appendix \ref{section:
appendix} we show that fixing the last, say, $t$ coefficients of a
univariate polynomial vanishing on the first $t$ subdiscriminants
yields a finite-to-one determination. This is critical to
characterize the dimension of the singular locus of the varieties
$V$ and $V_F$ underlying Theorems \ref{th: main result - intro} and
\ref{th: estimate |V_F| intro}.

We illustrate the interest of Theorems \ref{th: main result - intro}
and \ref{th: estimate |V_F| intro} through their application to a
classical combinatorial problem over finite fields and a problem of
coding theory. The combinatorial problem is concerned with the
distribution of factorization patterns on families of monic
polynomials of given degree of $\fq[T]$. If the family under
consideration consists of the monic polynomials of degree $n$ for
which the first $s<n-2$ coefficients are fixed, then we have a
function--field analogous to the classical conjecture on the number
of primes in short intervals, which has been the subject of several
articles (see, e.g., \cite{Pollack13}, \cite{BaBaRo15} or \cite[\S
3.5]{MuPa13}). The study of general linear families of polynomial
goes back at least to \cite{Cohen72}. Here, following the approach
of \cite{CeMaPe17} and \cite{MaPePr20}, we obtain an explicit
estimate on the number of elements on nonlinear families of monic
polynomials of fixed degree of $\fq[T]$ having a given factorization
pattern. For this purpose, we rely on Theorem \ref{th: main result -
intro}, which allows us to improve the results of \cite{CeMaPe17}
and \cite{MaPePr20}. We remark that Theorem \ref{th: main result -
intro} can also be applied to estimate the average cardinality of
the value sets of families of polynomials of $\fq[T]$ with certain
consecutive coefficients prescribed (see \cite{CeMaPePr14},
\cite{MaPePr14} and \cite{MaPePr16b}).

The problem of coding theory we consider is concerned with the
existence of deep holes in standard Reed-Solomon codes over $\fq$.
In \cite{ChMu07} it is conjectured that a given word is a deep hole
of the standard Reed-Solomon code $C$ of dimension $k$ over $\fq$ if
and only if it is generated by a polynomial $f\in\fq[T]$ of degree
$k$ (see, e.g., \cite{LiWa08}, \cite{WuHo12} and \cite{LiZh16} for
further work on this). Furthermore, the existence of deep holes of
$C$ is related to the non-existence of $\fq$--rational points of a
certain hypersurface defined over $\fq$. In \cite{CaMaPr12} we show
that this hypersurface is defined by a symmetric polynomial. Taking
into account the characterization of this hypersurface of
\cite{CaMaPr12}, using Theorem \ref{th: estimate |V_F| intro} we
shall be able to improve the results of \cite{CaMaPr12}.

The paper is organized as follows. In Section \ref{section:
notation, notations} we briefly recall the notions and notations of
algebraic geometry we use. Section \ref{section: smooth symmetric
systems} is devoted to present our unified framework and to
establish the first results on the geometry of varieties defined by
general smooth symmetric systems. Section \ref{section: singular
locus F_s not depending on Pi_(m-1),Pi_m} is devoted to the study of
varieties defined by smooth symmetric systems as in \eqref{eq: def
F_s intro} and the proof of Theorem \ref{th: main result - intro}.
Section \ref{section: hypersurfaces not depending on Pi_m} is
concerned with hypersurfaces and the proof of Theorem \ref{th:
estimate |V_F| intro}. In Sections \ref{section: distribution of
fact patterns} and \ref{section: deep holes} we apply Theorems
\ref{th: main result - intro} and \ref{th: estimate |V_F| intro} to
determine the distribution of factorization patterns of families of
univariate polynomials and to establish criteria for the
nonexistence of deep holes in standard Reed-Solomon codes. Finally,
in Appendix \ref{section: appendix} we consider the behavior of the
variety defined by the vanishing of successive generalized
subdiscriminants.
%
%
\section{Basic notions of algebraic geometry}
\label{section: notation, notations}
In this section we collect the basic definitions and facts of
algebraic geometry that we need in the sequel. We use standard
notions and notations which can be found in, e.g., \cite{Kunz85},
\cite{Shafarevich94}.

Let $p$ be a prime number, $q$ a power of $p$, $\fq$ the finite
field of $q$ elements and $\cfq$ its algebraic closure. Let $\K$ be
any of the fields $\fq$ or $\cfq$. We denote by $\A^r$ the affine
$r$--dimensional space $\cfq{\!}^{r}$ and by $\Pp^r$ the projective
$r$--dimensional space over $\cfq{\!}^{r+1}$. Both spaces are
endowed with their respective Zariski topologies over $\K$, for
which a closed set is the zero locus of a set of polynomials of
$\K[X_1,\ldots, X_r]$, or of a set of homogeneous polynomials of
$\K[X_0,\ldots, X_r]$.

A subset $V\subset \Pp^r$ is a {\em projective variety defined over}
$\K$ (or a projective $\K$--variety for short) if it is the set of
common zeros in $\Pp^r$ of homogeneous polynomials $F_1,\ldots, F_m
\in\K[X_0 ,\ldots, X_r]$. Correspondingly, an {\em affine variety of
    $\A^r$ defined over} $\K$ (or an affine $\K$--variety) is the set of
common zeros in $\A^r$ of polynomials $F_1,\ldots, F_{m} \in
\K[X_1,\ldots, X_r]$. We think a projective or affine $\K$--variety
to be equipped with the induced Zariski topology. We shall denote by
$\{F_1=0,\ldots, F_m=0\}$ or $V(F_1,\ldots,F_m)$ the affine or
projective $\K$--variety consisting of the common zeros of
$F_1,\ldots, F_m$.

In the remaining part of this section, unless otherwise stated, all
results referring to varieties in general should be understood as
valid for both projective and affine varieties.

A $\K$--variety $V$ is {\em irreducible} if it cannot be expressed
as a finite union of proper $\K$--subvarieties of $V$. Further, $V$
is {\em absolutely irreducible} if it is $\cfq$--irreducible as a
$\cfq$--variety. Any $\K$--variety $V$ can be expressed as an
irredundant union $V=\mathcal{C}_1\cup \cdots\cup\mathcal{C}_s$ of
irreducible (absolutely irreducible) $\K$--varieties, unique up to
reordering, called the {\em irreducible} ({\em absolutely
irreducible}) $\K$--{\em components} of $V$.

For a $\K$--variety $V$ contained in $\Pp^r$ or $\A^r$, its {\em
defining ideal} $I(V)$ is the set of polynomials of $\K[X_0,\ldots,
X_r]$, or of $\K[X_1,\ldots, X_r]$, vanishing on $V$. The {\em
coordinate ring} $\K[V]$ of $V$ is the quotient ring
$\K[X_0,\ldots,X_r]/I(V)$ or $\K[X_1,\ldots,X_r]/I(V)$. The {\em
dimension} $\dim V$ of $V$ is the length $n$ of a longest chain
$V_0\varsubsetneq V_1 \varsubsetneq\cdots \varsubsetneq V_n$ of
nonempty irreducible $\K$--varieties contained in $V$. We say that
$V$ has {\em pure dimension} $n$ if every irreducible
$\K$--component of $V$ has dimension $n$. A $\K$--variety of $\Pp^r$
or $\A^r$ of pure dimension $r-1$ is called a $\K$--{\em
hypersurface}. A $\K$--hypersurface of $\Pp^r$ (or $\A^r$) can also
be described as the set of zeros of a single nonzero polynomial of
$\K[X_0,\ldots, X_r]$ (or of $\K[X_1,\ldots, X_r]$).

The {\em degree} $\deg V$ of an irreducible $\K$--variety $V$ is the
maximum of the cardinaliy $|V\cap L|$ of $V\cap L$, considering all
the linear spaces $L$ of codimension $\dim V$ such that $|V\cap
L|<\infty$. More generally, following \cite{Heintz83} (see also
\cite{Fulton84}), if $V=\mathcal{C}_1\cup\cdots\cup \mathcal{C}_s$
is the decomposition of $V$ into irreducible $\K$--components, we
define the degree of $V$ as
$$\deg V:=\sum_{i=1}^s\deg \mathcal{C}_i.$$
The degree of a $\K$--hypersurface $V$ is the degree of a polynomial
of minimal degree defining $V$. 
%
We shall use the following {\em B\'ezout inequality} (see
\cite{Heintz83}, \cite{Fulton84}, \cite{Vogel84}): if $V$ and $W$
are $\K$--varieties of the same ambient space, then
\begin{equation}\label{eq: Bezout}
\deg (V\cap W)\le \deg V \cdot \deg W.
\end{equation}

Let $V\subset\A^r$ be a $\K$--variety, $I(V)\subset \K[X_1,\ldots,
X_r]$ its defining ideal and $\bfs x$ a point of $V$. The {\em
dimension} $\dim_{\bfs x}V$ {\em of} $V$ {\em at} $\bfs x$ is the
maximum of the dimensions of the irreducible $\K$--components of $V$
containing $\bfs x$. If $I(V)=(F_1,\ldots, F_m)$, the {\em tangent
space} $\mathcal{T}_{\bfs x}V$ to $V$ at $\bfs x$ is the kernel of
the Jacobian matrix $(\partial F_i/\partial X_j)_{1\le i\le m,1\le
j\le r}(\bfs x)$ of $F_1,\ldots, F_m$ with respect to $X_1,\ldots,
X_r$ at $\bfs x$. We have $\dim\mathcal{T}_{\bfs x}V\ge \dim_{\bfs
x}V$ (see, e.g., \cite[page 94]{Shafarevich94}). The point $\bfs x$
is {\em regular} if $\dim\mathcal{T}_{\bfs x}V=\dim_{\bfs x}V$;
otherwise, $\bfs x$ is called {\em singular}. The set of singular
points of $V$ is the {\em singular locus} of $V$; it is a closed
$\K$--subvariety of $V$. A variety is called {\em nonsingular} if
its singular locus is empty. For projective varieties, the concepts
of tangent space, regular and singular point can be defined by
considering an affine neighborhood of the point under consideration.

Let $V$ and $W$ be irreducibles $\K$--varieties of the same
dimension and let $f:V\to W$ be a regular map. We say that $f$ is
{\em dominant} if $\overline{f(V)}=W$ holds, where $\overline{f(V)}$
denotes the closure of $f(V)$ with respect to the Zariski
$\K$--topology of $W$. Then $f$ induces a ring extension
$\K[W]\hookrightarrow \K[V]$ by composition with $f$. We say that
$f$ is a {\em finite morphism} if this extension is integral, namely
if each element $\eta\in\K[V]$ satisfies a monic equation with
coefficients in $\K[W]$. A basic fact is that a finite morphism is
necessarily closed. Another fact concerning finite morphisms we
shall use in the sequel is that the preimage $f^{-1}(S)$ of an
irreducible closed subset $S\subset W$ is of pure dimension $\dim
S$.
%
%
\subsection{Rational points}
Let $\Pp^r(\fq)$ be the $r$--dimensional projective space over $\fq$
and $\A^r(\fq)$ the $r$--dimensional $\fq$--vector space $\fq^r$.
For a projective variety $V\subset\Pp^r$ or an affine variety
$V\subset\A^r$, we denote by $V(\fq)$ the set of $\fq$--rational
points of $V$, namely $V(\fq):=V\cap \Pp^r(\fq)$ in the projective
case and $V(\fq):=V\cap \A^r(\fq)$ in the affine case. For an affine
variety $V$ of dimension $n$ and degree $\delta$, we have the
following bound (see, e.g., \cite[Lemma 2.1]{CaMa06}):
\begin{equation}\label{eq: upper bound -- affine gral}
|V(\fq)|\leq \delta\, q^n.
\end{equation}
On the other hand, if $V$ is a projective variety of dimension $n$
and degree $\delta$, then we have the following bound (see
\cite[Proposition 12.1]{GhLa02a} or \cite[Proposition 3.1]{CaMa07};
see \cite{LaRo15} for more precise upper bounds):
\begin{equation*}\label{eq: upper bound -- projective gral}
|V(\fq)|\leq \delta\, p_n,
\end{equation*}
where $p_n:=q^n+q^{n-1}+\cdots+q+1=|\Pp^n(\fq)|$.
%
%
\subsection{Complete intersections}\label{subsec: complete intersections}
A \emph{set-theoretic complete intersection} is an affine
$\mathbb{K}$-variety $V(F_1 \klk F_m) \subseteq \A^r$ or a
projective $\mathbb{K}$-variety $V(F_1 \klk F_m)\subseteq
\mathbb{P}^r$ defined by $m\le r$ polynomials $F_1 \klk F_m\in
\K[X_1 \klk X_r]$, or homogeneous polynomials $F_1 \klk F_m$ in
$\K[X_0 \klk X_r]$, which is of pure dimension $r-m$. If in addition
$(F_1 \klk F_m)$ is a radical ideal of $\K[X_1 \klk X_r]$ or $\K[X_0
\klk X_r]$, then we say that $V(F_1 \klk F_s)$ is an
\emph{ideal-theoretic complete intersection}. Elements $F_1 \klk
F_m\in \K[X_1 \klk X_r]$ or $\K[X_0 \klk X_r]$ form a \emph{regular
sequence} if the ideal $(F_1 \klk F_m)$ they define in $\K[X_1 \klk
X_r]$ or $\K[X_0 \klk X_r]$ is a proper ideal, $F_1$ is nonzero and,
for $2\le i \le m$, $F_i$ is neither zero nor a zero divisor in
$\K[X_1 \klk X_r]/(F_1 \klk F_{i-1})$ or $\K[X_0 \klk X_r]/(F_1 \klk
F_{i-1})$. In such a case, the variety $V(F_1,\ldots,F_m)$ they
define is a set-theoretic complete intersection. In what follows we
shall use the following result (see, e.g., \cite[Chapitre 3,
Remarque 2.2]{Lejeune84}).
\begin{lemma}\label{lemma: reg seq and complete int}
If $F_1 \klk F_m$ are  $m\le r+1$ homogeneous polynomials in $\K[X_0
\klk X_r]\setminus \K$, the following conditions  are equivalent:
   \begin{itemize}
     \item $F_1 \klk F_m$ define a set-theoretic complete intersection of $\mathbb{P}^r$.
     \item $F_1 \klk F_m$  is a regular sequence in $\K[X_0 \klk X_r]$.
   \end{itemize}
\end{lemma}

If $V\subset\Pp^r$ is a complete intersection defined over $\K$ of
dimension $r-m$, and $F_1 ,\ldots, F_m$ is a system of homogeneous
generators of $I(V)$, the degrees $d_1,\ldots, d_m$ depend only on
$V$ and not on the system of generators. Arranging the $d_i$ in such
a way that $d_1\geq d_2 \geq \cdots \geq d_m$, we call $(d_1,\ldots,
d_m)$ the {\em multidegree} of
$V$. In this case, a stronger version of 
\eqref{eq: Bezout} holds, called the {\em B\'ezout theorem} (see,
e.g., \cite[Theorem 18.3]{Harris92}):
\begin{equation}\label{eq: Bezout theorem}
\deg V=d_1\cdots d_m.
\end{equation}

A complete intersection $V$ is {\em regular in codimension $k$} if
the singular locus $\mathrm{Sing}(V)$ of $V$ has codimension at
least $k+1$ in $V$, namely $\dim V-\dim \mathrm{Sing}(V)\ge k+1$. In
particular, a complete intersection $V$ is called {\em normal} if it
is regular in codimension 1 (actually, normality is a general notion
that agrees on complete intersections with the one we define here).
A fundamental result for projective complete intersections is the
Hartshorne connectedness theorem (see, e.g., \cite[Theorem
VI.4.2]{Kunz85}): if $V\subset\Pp^r$ is a complete intersection
defined over $\K$ and $W\subset V$ is any $\K$--subvariety of
codimension at least 2, then $V\setminus W$ is connected in the
Zariski topology of $\Pp^r$ over $\K$. Applying the Hartshorne
connectedness theorem with $W:=\mathrm{Sing}(V)$, one deduces the
following result.
\begin{theorem}\label{th: normal complete int implies irred}
    If $V\subset\Pp^r$ is a normal complete intersection, then $V$ is
    absolutely irreducible.
\end{theorem}
%
%
\section{Smooth symmetric systems}
\label{section: smooth symmetric systems}
For a positive integer $m$, let $X_1 \klk X_m$ be indeterminates
over $\fq$. For $s<m$, we are interested in systems
$F_1=0,\ldots,F_s=0$ defined by polynomials
$F_1,\ldots,F_s\in\fq[X_1,\ldots,X_m]$ which are symmetric in
$X_1,\ldots,X_m$. We express these polynomials in terms of the
elementary symmetric polynomials $\Pi_1 \klk \Pi_m$ of $\fq[X_1 \klk
X_m]$, namely
\begin{equation}\label{def: f}
F_i=G_i(\Pi_1,\ldots,\Pi_m)\quad(1\le i\le s),
\end{equation}
where $G_1\ldots,G_s\in\fq[Y_1,\ldots,Y_m]$ and $Y_1,\ldots,Y_m$ are
new indeterminates over $\fq$.

Let $d_i$ be the degree of $F_i$ for $1\le i\le s$. Write each $F_i$
in the form $F_i=F_i^{(d_i)}+F_i^*$, where $F_i^{(d_i)}$ denotes the
homogeneous component of $F_i$ of degree $d_i$ and $F_i^*$ is the
sum of the homogeneous components of $F_i$ of degree less than
$d_i$. We denote $\bfs d:=(d_1,\ldots,d_s)$. As each $F_i^{(d_i)}$
is symmetric in $X_1,\ldots,X_m$, it can be expressed in the
following way:
\begin{equation}\label{def: f_i^(d_i)}
F_i^{(d_i)}=G_i^{(d_i)}(\Pi_1,\ldots,\Pi_m)\quad(1\le i\le s),
\end{equation}
where $G_1^{(d_1)}\ldots,G_s^{(d_s)}\in\fq[Y_1,\ldots,Y_m]$. We
remark that $G_i^{(d_i)}$ does not necessarily agree with the
homogeneous component of highest degree of $G_i$.

We aim to estimate the number of $\fq$-rational solutions of the
system $F_1=0,\ldots,F_s=0$ and obtain conditions on
$G_1,\ldots,G_s$ which imply that the system $F_1=0,\ldots,F_s=0$
has at least an $\fq$-rational solution. For this purpose, we
consider the affine $\fq$--variety $V:=V(\bfs F_s) \subset \A^m$
defined by $\bfs F_s:=(F_1,\ldots,F_s)$. We shall establish some
facts concerning the geometry of $V$, starting with an analysis of
the singular locus $\Sigma$ of $V$.
%
%
\subsection{On the singular locus of $V$}
\label{subsec: singular locus of V}
Let $\bfs Y:=(Y_1,\ldots,Y_m)$, $\bfs G_s:=(G_1\ldots,G_s)\in
\fq[\bfs Y]^s$, $\bfs F_s^{(\bfs
d)}:=(F_1^{(d_1)},\ldots,F_s^{(d_s)})$ and $\bfs G_s^{(\bfs
d)}:=(G_1^{(d_1)},\ldots,G_s^{(d_s)})$. Therefore, we have $\bfs
F_s=\bfs G_s(\bfs \Pi)$ and $\bfs F_s^{(\bfs d)}=\bfs G_s^{(\bfs
d)}(\bfs \Pi)$, where $\bfs\Pi:=(\Pi_1\klk \Pi_m)$.
In the sequel, we shall denote the singular locus of $V$ by
$\Sigma$.

Given ${\bfs x}\in \A^m$, we denote by $J\,\bfs\Pi({\bfs x})$ the
Jacobian matrix of $\bfs \Pi$ with respect to $\bfs X:=(X_1 \klk
X_m)$ at $\bfs x$, namely
$$
J\,\bfs\Pi({\bfs x}):=
\left(
\begin{array}{ccc}
\dfrac{\partial \Pi_{1}}{\partial X_1}({\bfs x}) & \cdots &
\dfrac{\partial \Pi_{1}}{\partial X_{m}}({\bfs x})
\\
\vdots & & \vdots
\\
\dfrac{\partial \Pi_{m}}{\partial X_1}({\bfs x})& \cdots &
\dfrac{\partial \Pi_{m}}{\partial X_{m}}({\bfs x})
\end{array}
\right).
$$

We have the following result.
\begin{lemma}\label{lemma: determinant Jacobian Pi}
$\det(J\,\bfs\Pi)=(-1)^{\frac{(m-1)m}{2}} \prod_{1\le i < j\le
m}(X_j-X_i)$.
\end{lemma}
\begin{proof}
We have the following well-known equalities (see, e.g.,
\cite{LaPr02}):
\begin{equation}\label{eq: derivative Pi_i}
\frac{\partial \Pi_i}{\partial X_{j}}= \Pi_{i-1}-X_{j} \Pi_{i-2} +
X_{j}^2 \Pi_{i-3} +\cdots+ (-1)^{i-1} X_{j}^{i-1}\quad (1\leq i,j
\leq m).
\end{equation}
Denote by $A_m$ the $(m\times m)$--Vandermonde matrix
$$
A_m:=(X_j^{i-1})_{1\leq i,j\leq m}.
$$
From \eqref{eq: derivative Pi_i} we easily deduce that the Jacobian
matrix $J\,\bfs\Pi$ can be factored as follows:
\begin{equation*} \label{eq: factorization Jacobian elem sim pols}
J\,\bfs\Pi:=B_m\cdot A_m := \left(
\begin{array}{ccccc}
1 & \ 0 & 0 &  \dots & 0
\\
\ \ \Pi_1 &  -1 & 0 &  &
\\
\Pi_2 & \ \ -\Pi_1 & 1 & \ddots & \vdots
\\
\vdots &\vdots  & \vdots & \ddots & 0
\\
\Pi_{m-1} & -\Pi_{m-2} &\Pi_{m-3} & \cdots &\!\! (-1)^{m-1}
\end{array}
\!\!\right) \cdot A_m.
\end{equation*}
Since $\det B_m=(-1)^{\frac{(m-1)m}{2}}$, the conclusion of the
lemma readily follow.
\end{proof}

By the chain rule, the partial derivatives of $F_i(\bf \Pi)$ satisfy
the following equality:
$$\dfrac{\partial G_i(\bf \Pi)}{\partial X_j}  =
\bigg(\dfrac{\partial G_i}{\partial
    Y_1}\circ\bfs{\Pi}\bigg)\cdot\dfrac{\partial \Pi_{1}}{\partial
    X_j}+\bigg(\dfrac{\partial G_i}{\partial
    Y_2}\circ\bfs{\Pi}\bigg)\cdot\dfrac{\partial \Pi_{2}}{\partial
    X_j}+\cdots+\bigg(\dfrac{\partial G_i}{\partial
    Y_m}\circ\bfs{\Pi}\bigg)\cdot \dfrac{\partial \Pi_{m}}{\partial X_j}.
$$
In the sequel we shall assume that $\bfs G_s$ satisfies the
following conditions:
\medskip

\noindent (${\sf A}_1$) $W:=V(\bfs G_s)$ is nonempty, and the
Jacobian matrix $J\bfs G_s(\bfs y)$ of $\bfs G_s$ with respect to
$\bfs Y$ at $\bfs y$ is of full rank for any $\bfs y\in W$.
\medskip

\noindent (${\sf A}_2$) $W^{(\bfs d)}:=V(\bfs G_s^{(\bfs d)})$ is
nonempty, and the Jacobian matrix $J\bfs G_s^{(\bfs d)}(\bfs y)$ of
$\bfs G_s^{(\bfs d)}$ with respect to $\bfs Y$ at $\bfs y$ is of
full rank for any $\bfs y\in W^{(\bfs d)}$.
\medskip

We have the following consequences of assumptions $({\sf A}_1)$ and
$({\sf A}_2)$.
\begin{lemma}\label{lemma: consequences hypothesis A_1}
$W$ and $W^{(\bfs d)}$ are of pure dimension $m-s$, the ideals
generated by $\bfs G_s$ and $\bfs G_s^{(\bfs d)}$ are radical, and
$V$ and $V^{(\bfs d)}:=V(\bfs F_s^{(\bfs d)})$ are of pure dimension
$m-s$.
\end{lemma}
\begin{proof}
We start with the assertions about $W$, $\bfs G_s$ and $V$. Let
$\bfs y\in W$. As $J\bfs G_s(\bfs y)$ is of full rank, by, e.g.,
\cite[Chapter VI, Proposition 1.5]{Kunz85}, we conclude that the
local dimension $\dim_{\bfs y}W$ of $W$ at $\bfs y$ satisfies the
inequality $\dim_{\bfs y}W\le m-s$. On the other hand, as $W$ is
defined by $s$ polynomials, any irreducible component of $W$ is of
dimension at least $m-s$. It follows that $W$ is of pure dimension
$m-s$.

In particular, by, e.g., \cite[Proposition 18.13]{Eisenbud95}, we
conclude that the quotient ring
%
$S:=\cfq[\bfs Y]/(\bfs G_s)$ is Cohen-Macaulay, where $(\bfs G_s)$
is the ideal generated by $\bfs G_s$. Then \cite[Theorem
18.15]{Eisenbud95} implies that $S$ is reduced, and therefore $(\bfs
G_s)$ is a radical ideal.

Finally, we observe that the (surjective) morphism of
$\fq$--varieties $\Pi: \A^m  \rightarrow \A^m$,
$\Pi(\bfs{x}):=(\Pi_1(\bfs{x}),\ldots,\Pi_m(\bfs{x}))$ is a dominant
finite morphism (see, e.g., \cite[\S 5.3, Example
1]{Shafarevich94}). In particular, the preimage $V=\Pi^{-1}(W)$ is
of pure dimension $m-s$.

The assertions on $W^{(\bfs d)}$, $\bfs G_s^{(\bfs d)}$ and
$V^{(\bfs d)}$ follow {\em mutatis mutandis}.
\end{proof}

For $\bfs x\in V$, we have that, if ${\bfs x}\in\Sigma$, then the
Jacobian matrix $J\bfs F_s(\bfs x)$ of $\bfs F_s$ with respect to
$\bfs X$ at $\bfs x$ is rank deficient. As
$$
J\bfs F_s({{\bfs x}})=J\bfs G_s (\boldsymbol{\Pi}({{\bfs x}}))\cdot
J\,\bfs\Pi({{\bfs x}}),
$$
assumption (${\sf A}_1$) implies that $\det(J\,\bfs\Pi({\bfs
x}))=0$. According to Lemma \ref{lemma: determinant Jacobian Pi},
$\det(J\,\bfs\Pi({\bfs x}))=0$ if and only if $\prod_{1\le i < j\le
m}(x_j-x_i)=0$. It follows that, if ${{\bfs x}} \in \Sigma$, then
$x_i=x_j$ for some $1 \leq i < j \leq m$. More precisely, we have
the following result.
\begin{lemma}
Any point $\bfs x\in V$ such that $ J\bfs F_s({{\bfs x}})$ is not of
full rank has at most $m-1$ pairwise distinct coordinates. In
particular, all the points of the singular locus $\Sigma$ of $V$
have at most $m-1$ pairwise distinct coordinates, and we have the
inclusion
$$\Sigma \subset \bigcup_{1 \leq i < j \leq m} \{X_i=X_j\}\cap V.$$
\end{lemma}

Now we discuss significant classes of systems $F_1=0,\ldots,F_s=0$
as in \eqref{def: f} for which the associated variety $V$ has a
singular locus of codimension at least $2$. We consider first the
set of systems $F_1=0,\ldots,F_s=0$ as in \eqref{def: f} not
depending on $\Pi_{m-k+1},\ldots,\Pi_m$ for $1\le i\le s$ and $2\le
k\le m-s$.
%
%
\section{Sparse polynomials systems}
\label{section: singular locus F_s not depending on Pi_(m-1),Pi_m}
For $2\le k\le m-s$ and $F_1,\ldots,F_s$ as in \eqref{def: f} not
depending on $\Pi_{m-k+1},\ldots,\Pi_m$, we can write $\bfs F_s=\bfs
G_s({\bfs\Pi}_{m-k})$, where ${\bfs\Pi}_{m-k}:=(\Pi_1\klk
\Pi_{m-k})$. We assume that that assumption $({\sf A}_1)$ holds. If
$\bfs x\in\Sigma$, then the Jacobian matrix
$$
J\bfs F_s({{\bfs x}})=J\bfs G_s (\bfs\Pi_{m-k}({{\bfs x}})) \cdot
J\,\bfs\Pi_{m-k}({{\bfs x}})
$$
is rank deficient, where $J\,\bfs\Pi_{m-k}$ is the Jacobian matrix
of $\bfs\Pi_{m-k}$ with respect to $\bfs X$. Arguing as in the proof
of Lemma \ref{lemma: determinant Jacobian Pi}, we have the
factorization
\begin{align*}
J\,\bfs\Pi_{m-k}=B_{m-k}^*\cdot A_{m-k}^* &:= \left(
\begin{array}{ccccc}
1 &  0 & 0 &  \dots & 0
\\
\Pi_1 &  -1 & 0 &  &
\\
\Pi_2 & -\Pi_1 & 1 & \ddots & \vdots
\\
\vdots &\vdots  & \vdots & \ddots & 0
\\
\Pi_{m-k-1} & -\Pi_{m-k-2} &\Pi_{m-k-3} & \cdots &\!\! (-1)^{m-k-1}
\end{array}
\!\!\right) \cdot A_{m-k}^*,
\\[1ex]
A_{m-k}^*&:=(X_j^{i-1})_{1\leq i\le m-k,1\le j\leq m}.
\end{align*}
As $J\bfs G_s (\bfs\Pi_{m-k}(\bfs x)) \cdot B_{m-k}^*$ is of full
rank, we conclude that $\mathrm{rank}(A_{m-k}^*({\bfs x}))<m-k$. We
have the following result.
\begin{lemma} \label{lemma: singular locus F_s not depending on Pi_(m-k+1),..,Pi_m}
Suppose that $\bfs F_s$ do not depend on $\Pi_{m-k+1},\ldots,\Pi_m$
and assumption $({\sf A}_1)$ holds. Then any point $\bfs x\in V$
such that $ J\bfs F_s({{\bfs x}})$ is not of full rank has at most
$m-k-1$ pairwise distinct coordinates. In particular, all the
elements of the singular locus $\Sigma$ of $V:=V(\bfs F_s)$ have at
most $m-k-1$ pairwise distinct coordinates. Therefore,
\begin{equation}\label{eq: singular locus F_s without Pi(m-1), Pi(m)}
\Sigma\subset\{\bfs x\in V:\mathrm{rank}\big(J\bfs F_s(\bfs
x)\big)<s\}\subset\bigcup_{\mathcal{L}_\mathcal{I}}
\mathcal{L}_\mathcal{I}\cap V,\end{equation}
where $\mathcal{I}:=\{I_1,\ldots,I_{m-k-1}\}$ runs over all the
partitions of $\{1,\ldots,m\}$ into $m-k-1$ nonempty subsets
$I_j\subset \{1,\ldots,m\}$ and $\mathcal{L_{I}}$ is the linear
variety
$$
\mathcal{L}_{\mathcal I}:=
\mathrm{span}(\bfs{v}^{(I_1)},\ldots,\bfs{v}^{(I_{m-k-1})})
  $$
spanned by the vectors $\bfs{v}^{(I_j)}:=
(v_1^{(I_j)},\ldots,v_{m}^{(I_j)})\in\{0,1\}^m$ defined by
$v_l^{(I_j)}:=1$ iff $l\in I_j$.
\end{lemma}
\begin{proof}
Let $\bfs x\in\Sigma$. Then $\mathrm{rank}(A_{m-k}^*({\bfs
x}))<m-k$. Now suppose that $\bfs x$ has $m-k$ pairwise distinct
coordinates. Without loss of generality we may suppose that $\bfs
x:=(x_1,\ldots,x_m)$ and $x_i\not=x_j$ for $1\le i<j\le m-k$.
Observe that the $(m-k)\times(m-k)$-minor of $A_{m-k}^*({\bfs x})$
consisting of the first $m-k$ columns of $A_{m-k}^*({\bfs x})$ is
the Vandermonde $(m-k)\times(m-k)$-matrix
${A}_{m-k}:=(x_j^{i-1})_{1\leq i,j\le m-k}$, whose determinant is
$\prod_{1\leq i<j\le m-k}(x_i-x_j)\not=0$. This implies that
$A_{m-k}^*({\bfs x})$ has rank $m-k$, which contradicts the
condition $\mathrm{rank}(A_{m-k}^*({\bfs x}))<m-k$. The statement of
the lemma readily follows.
\end{proof}

In what follows, we shall use the following weaker consequence of
Lemma \ref{lemma: singular locus F_s not depending on
Pi_(m-k+1),..,Pi_m}: all the elements in the singular locus $\Sigma$
of $V$ have at most $m-k$ pairwise distinct coordinates. Let
$\mathcal{I}:=\{I_1,\ldots,I_{m-k}\}$ be a partition of
$\{1,\ldots,m\}$. Assume without loss of generality that $i\in I_i$
for $1\le i\le m-k$. Consider the mapping
$$\begin{array}{rcl}
\bfs\Pi_{m-k,\mathcal{I}}:\A^{m-k}&\to&\A^{m-k},\\
(x_1,\ldots,x_{m-k})&\mapsto&
\bfs\Pi_{m-k}(x_1\bfs{v}^{(I_1)}+\cdots+x_{m-k}\bfs{v}^{(I_{m-k})}),
\end{array}$$
with $\bfs{v}^{(I_1)},\ldots,\bfs{v}^{(I_{m-k})}$ as in the
statement of the lemma. We have that any $\bfs x\in\Sigma$ belongs
to the intersection $\mathcal{L}_\mathcal{I}\cap V$ for a suitable
partition $\mathcal{I}$ of $\{1,\ldots,m\}$. The set
$\mathcal{L}_{\mathcal{I}}\cap V$ is isomorphic to the set
$$\{\bfs x\in\A^{m-k}:\bfs G_s(\bfs\Pi_{m-k,\mathcal{I}}(\bfs x))=\bfs 0\}.$$
Denote $L_{\mathcal{I}}:\A^{m-k}\to\A^m$,
$L_{\mathcal{I}}(x_1,\ldots,x_{m-k}):=x_1\bfs{v}^{(I_1)}+\cdots+x_{m-k}\bfs{v}^{(I_{m-k})}$.
Observe that $\bfs\Pi_{m-k,\mathcal{I}}=\bfs\Pi_{m-k}\circ
L_{\mathcal{I}}$.
\begin{proposition}\label{prop: Phi_i is a finite morphism}
For $k\ge 1$, if $\mathrm{char}(\fq)$ does not divide
$m(m-1)\cdots(m-k+1)$, then $\bfs\Pi_{m-k,\mathcal{I}}$ is a finite
morphism.
\end{proposition}
\begin{proof}
Denote $\Pi_{i,\mathcal{I}}
:=\Pi_i(X_1\bfs{v}^{(I_1)}+\cdots+X_{m-k}\bfs{v}^{(I_{m-k})})$ for
$1\le i\le m$ and $\bfs\Pi_{m,\mathcal{I}}:=\bfs\Pi_m\circ
L_{\mathcal{I}}$. We have to prove that the ring homomorphism
$\cfq[\bfs\Pi_{m-k,\mathcal{I}}]\to \cfq[X_1,\ldots,X_{m-k}]$ is
integral. Substituting
$X_1\bfs{v}^{(I_1)}+\cdots+X_{m-k}\bfs{v}^{(I_{m-k})}$ for
$(X_1,\ldots,X_m)$ in the identity
$$X_i^m+\sum_{j=1}^m(-1)^j\Pi_jX_i^{m-j}=0$$
for $1\le i\le m-k$ we conclude that the ring  homomorphism
$\cfq[\bfs\Pi_{m,\mathcal{I}}]\hookrightarrow
\cfq[X_1,\ldots,X_{m-k}]$ is integral. As a consequence, it suffices
to show that the ring  homomorphism
$$\cfq[\bfs\Pi_{m-k,\mathcal{I}}]\to
\cfq[\bfs\Pi_{m,\mathcal{I}}]=\cfq[\bfs\Pi_{m-k,\mathcal{I}}][\Pi_{m-k+1,\mathcal{I}},\ldots,\Pi_{m,\mathcal{I}}]$$
is integral.

For this purpose, we exhibit $k$ polynomials vanishing on the image
of the mapping $\bfs\Pi_{m,\mathcal{I}}$. Denote $\bfs
X:=(X_1,\ldots,X_m)$, $\bfs X_{m-k}=(X_1,\ldots,X_{m-k})$ and
$f_{\bfs X}:=(T-X_1)\cdots(T-X_m)$. To describe these $k$
polynomials, in what follows we identify the elements of
$\bfs\Pi_{m,\mathcal{I}}(\bfs x_{m-k})=\bfs\Pi_m\circ
L_{\mathcal{I}}(\bfs x_{m-k})$ with $\bfs x_{m-k}\in\A^{m-k}$ with
monic polynomials of $\cfq[T]$ of degree $m$ by mapping each
$\bfs\Pi_{m,\mathcal{I}}(\bfs x)$ to the polynomial
\begin{equation}\label{eq: definition f_x}
f_{L_\mathcal{I}(\bfs x_{m-k})}:=T^m+ \sum_{j=1}^m(-1)^j \Pi_j\big(
L_\mathcal{I}(\bfs x_{m-k})\big)T^{m-j}.
\end{equation}
For any $f_{L_\mathcal{I}(\bfs x_{m-k})}$ and $0\le j<k$, let
$\mathrm{sDisc}_j(f_{L_\mathcal{I}(\bfs
x_{m-k})}):=\mathrm{sRes}_j(f_{L_\mathcal{I}(\bfs
x_{m-k})},f'_{L_\mathcal{I}(\bfs x_{m-k})})$ denote the $j$th
subdiscriminant of $f_{L_\mathcal{I}(\bfs x_{m-k})}$, namely the
$j$th subresultant of $f_{L_\mathcal{I}(\bfs x_{m-k})}$ and its
derivative $f'_{L_\mathcal{I}(\bfs x_{m-k})}$ (see Appendix
\ref{section: appendix} for details). Observe that
$$\mathrm{sDisc}_j(f_{L_\mathcal{I}(\bfs
x_{m-k})})=\mathrm{sDisc}_j(f_{\bfs X})\circ L_\mathcal{I}(\bfs
x_{m-k})\textrm{ for }0\le j<k.$$

By definition, for any $\bfs
x_{m-k}:=(x_1,\ldots,x_{m-k})\in\A^{m-k}$ the roots of $f_{\bfs
x_{m-k}}$ in $\cfq$ are $x_1,\ldots,x_{m-k}$, with multiplicities at
least $|I_1|,\ldots,|I_{m-k}|$. This implies that the degree of
$\gcd(f_{L_\mathcal{I}(\bfs x_{m-k})},f'_{L_\mathcal{I}(\bfs
x_{m-k})})$ is at least $|I_1|+\cdots+|I_{m-k}|-(m-k)=k$. By, e.g.,
\cite[Proposition 4.24]{BaPoRo06}, we conclude that
\begin{equation}\label{eq: identity disc-subdisc f_x=0}
\mathrm{sDisc}_j(f_{L_\mathcal{I}(\bfs x_{m-k})})=0\textrm{ for
}0\le j<k.
\end{equation}
Considering the generic subdiscriminants
$\mathrm{sDisc}_j(f_{L_\mathcal{I}(\bfs X_{m-k})})$ for $0\le j<k$
as elements of $\cfq[\bfs\Pi_{m,\mathcal{I}}]$, according to
\eqref{eq: identity disc-subdisc f_x=0} we have
$$\mathrm{sDisc}_j(f_{L_\mathcal{I}(\bfs X_{m-k})})=\mathrm{sDisc}_j(f_{\bfs X})\circ L_\mathcal{I}(\bfs
X_{m-k})=0\textrm{ for }0\le j<k$$
in $\cfq[\bfs\Pi_{m,\mathcal{I}}]$. From Theorem \ref{th: quotient
by subdiscs is integral} we deduce that
\begin{align*}
\cfq[\bfs\Pi_{m-k,\mathcal{I}}]&\to
\cfq[\bfs\Pi_{m,\mathcal{I}}]/(\mathrm{sDisc}_j(f_{L_\mathcal{I}(\bfs
X_{m-k})}):0\le j<k)=\cfq[\bfs\Pi_{m,\mathcal{I}}]
\end{align*}
is integral.

Now we show that $\bfs\Pi_{m-k,\mathcal{I}}$ is dominant. Let $\bfs
x\in\A^{m-k}$. The fact that the ring homomorphism
$\cfq[\bfs\Pi_{m-k,\mathcal{I}}]\to \cfq[\bfs X_{m-k}]$ is integral
implies that the fiber
$(\bfs\Pi_{m-k,\mathcal{I}})^{-1}(\bfs\Pi_{m-k,\mathcal{I}}(\bfs
x))$ is finite. By the Theorem of the dimension of fibers (see,
e.g., \cite[\S 6.3, Theorem 7]{Shafarevich94}) we conclude that
$\dim\A^{m-k}=m-k\le
\dim\overline{\bfs\Pi_{m-k,\mathcal{I}}(\A^{m-k})}$. It follows that
$\dim\overline{\bfs\Pi_{m-k,\mathcal{I}}(\A^{m-k})}=m-k$, which
implies that
$\overline{\bfs\Pi_{m-k,\mathcal{I}}(\A^{m-k})}=\A^{m-k}$.
\end{proof}

Now we are able to establish the main result of this section.
\begin{theorem}\label{th: singular locus for F_s without Pi(m-1), Pi(m)}
Let $k\ge 2$ and suppose that $\mathrm{char}(\fq)$ does not divide
$m(m-1)\cdots(m-k+1)$. Let $\bfs F_s$ be polynomials as in
\eqref{def: f} not depending on $\Pi_{m-k+1},\ldots,$ $\Pi_m$ such
that assumption $({\sf A}_1)$ holds. Then the set of points $\bfs
x\in V:=V(\bfs F_s)\subset\A^m$ such that $ J\bfs F_s({{\bfs x}})$
is rank deficient has dimension at most $m-s-k$. In particular, the
singular locus $\Sigma$ of $V$ has dimension at most $m-s-k$.
\end{theorem}
\begin{proof}
According to \eqref{eq: singular locus F_s without Pi(m-1), Pi(m)},
it suffices to show that $\mathcal{L}_\mathcal{I}\cap V$ has
dimension at most $m-s-k$ for any partition $\mathcal{I}$ of
$\{1,\ldots,m\}$ into $m-k$ subsets. Therefore, fix a partition
$\mathcal{I}$ as above. As $\bfs G_s$ does not depend on
$Y_{m-k+1},\ldots,Y_m$, we may consider $G_1,\ldots,G_s$ as elements
of $\fq[Y_1,\ldots,Y_{m-k}]$. In this sense, the set
$\mathcal{L}_\mathcal{I}\cap V$ is isomorphic to the set
$$\{\bfs x\in\A^{m-k}:\bfs G_s(\bfs\Pi_{m-k,\mathcal{I}}(\bfs x))=\bfs 0\}.$$
Assumption $({\sf A}_1)$ implies that $\{\bfs G_s=\bfs
0\}\subset\A^{m-k}$ is of pure dimension $m-s-k$. According to
Proposition \ref{prop: Phi_i is a finite morphism},
$\bfs\Pi_{m-k,\mathcal{I}}$ is a finite morphism. It follows that
$$\bfs\Pi_{m-k,\mathcal{I}}^{-1}(\{\bfs G_s=\bfs 0\})=\{\bfs x\in\A^{m-k}:\bfs G_s(\bfs\Pi_{m-k,\mathcal{I}}(\bfs x))=\bfs 0\}$$
is of pure dimension $m-s-k$.
\end{proof}
%
%
\subsection{The geometry of the projective closure}
Consider the embedding of $\A^m$ into the projective space $\Pp^m$
defined by the mapping $(x_1,\ldots, x_m)\mapsto(1:x_1:\dots:x_m)$.
The closure $\mathrm{pcl}(V)\subset\Pp^m$ of the image of $V$ under
this embedding in the Zariski topology of $\Pp^m$ is called the {\em
projective closure} of $V$. The points of $\mathrm{pcl}(V)$ lying in
the hyperplane $\{X_0=0\}$ are called the {\em points of
$\mathrm{pcl}(V)$ at infinity}.

We also need information concerning the behavior of $V$ at
``infinity''. For this purpose, we consider the {\em singular locus
$\Sigma^{\infty}\subset\mathbb{P}^{m}$ of $\mathrm{pcl}(V)$ at
infinity}, namely the set of singular points of $\mathrm{pcl}(V)$
lying in the hyperplane $\{X_0=0\}$. We have the following result.
\begin{lemma}\label{lemma: singular locus at infinity without Pi_(m-1),Pi_m}
Suppose that $\mathrm{char}(\fq)$ does not divide
$m(m-1)\cdots(m-k+1)$, and let $\bfs F_s$ be as in \eqref{def: f}
satisfying assumption $({\sf A_2})$, not depending on
$\Pi_{m-k+1},\ldots,$ $\Pi_m$. Denote by $\bfs F_s^{(\bfs d)}$ the
vector of homogeneous components of highest degree of $\bfs F_s$.
Then the set $\{\bfs x\in V(\bfs F^{(\bfs d)}): \textrm{rank
}\big(J(\bfs F_s^{(\bfs d)})(\bfs x)\big)<s\}\subset\A^m$ has
dimension at most $m-s-k$. In particular, the singular locus
$\Sigma^{\infty}$ of $\mathrm{pcl}(V)$ at infinity has dimension at
most $m-s-k-1$.
\end{lemma}
\begin{proof}
Recall that each $F_i$ is expressed in the form
$F_i=F_i^{(d_i)}+F_i^*$, where $F_i^{(d_i)}$ denotes the homogeneous
component of $F_i$ of degree $d_i$ and $F_i^*$ is the sum of the
homogeneous components of $F_i$ of degree less than $d_i$. Then the
homogenization $F_i^h$ of $F_i$ can be expressed in the following
way:
$$F_i^h=F_i^{(d_i)}+X_0^{e_i}(F_i^*)^h,\textrm{ with }e_i\ge 1.$$
As a consequence, any point ${\bfs
x}:=(0:x_1:\cdots:x_{m})\in\Sigma^\infty$ satisfies the conditions
\begin{equation}\label{eq: singular locus at infinity}
\bfs F_s^h({\bfs x})=\bfs F_s^{(\bfs d)}(x_1,\ldots,x_m)=0,\quad
\textrm{rank }J(\bfs F_s^{(\bfs d)})(x_1,\ldots,x_m)<s,
\end{equation}
where $\bfs F_s^{(\bfs d)}:=(F_1^{(d_1)},\ldots,F_s^{(d_s)})$.
Assumption $({\sf A_2})$ asserts that $\bfs F_s^{(\bfs d)}$
satisfies assumption $({\sf A_1})$. Therefore, Theorem \ref{th:
singular locus for F_s without Pi(m-1), Pi(m)} proves that the
affine cone of $\A^m$ defined by \eqref{eq: singular locus at
infinity} has dimension at most $m-s-k$. It follows that the
projective variety $\Sigma^\infty$ has dimension at most $m-s-k-1$.
\end{proof}

Now we show the main result of this section.
\begin{theorem}\label{th: pcl(V) is normal abs irred}
Suppose that $\mathrm{char}(\fq)$ does not divide
$m(m-1)\cdots(m-k+1)$, and let $\bfs F_s$ be as in \eqref{def: f}
satisfying assumptions $({\sf A_1})$ and $({\sf A_2})$, not
depending on $\Pi_{m-k+1},\ldots,$ $\Pi_m$. Then $\mathrm{pcl}(V)$
is a complete intersection regular in codimension $k-1$ of degree
$d_1\cdots d_s$.
\end{theorem}
\begin{proof}
Denote $V^*:=V(\bfs F_s^h)\subset\Pp^m$. According to Lemma
\ref{lemma: consequences hypothesis A_1}, the affine varieties
$V(\bfs F_s)$ and $V(\bfs F_s^{(\bfs d)})$ are of (pure) dimension
$m-s$, and therefore $V^*$ has dimension $m-s$. As it is defined by
$s$ homogeneous polynomial, by Lemma \ref{lemma: reg seq and
complete int} we conclude that $\bfs F_s^h$ form a regular sequence.
Further, Theorem \ref{th: singular locus for F_s without Pi(m-1),
Pi(m)} and Lemma \ref{lemma: singular locus at infinity without
Pi_(m-1),Pi_m} show that the set of points of $V^*$ for which
$J(\bfs F_s^h)$ is not of maximal rank, has codimension at least 2
in $V^*$. By, e.g., \cite[Theorem 18.15]{Eisenbud95}, we deduce that
$\bfs F_s^h$ generate a radical ideal. It follows that $V^*$ is an
ideal-theoretic complete intersection.

From Theorem \ref{th: singular locus for F_s without Pi(m-1), Pi(m)}
it follows that the singular locus of $V^*\cap\{X_0\not=0\}$ has
dimension at most $m-s-k$. Further, Lemma \ref{lemma: singular locus
at infinity without Pi_(m-1),Pi_m} shows that the singular locus of
$V^*$ at infinity has dimension at most $m-s-k-1$. We conclude that
the singular locus of $V^*$ has dimension at most $m-s-k$, and thus
it is regular in codimension $k-1$. In particular, $V^*$ is a normal
variety, and Theorem \ref{th: normal complete int implies irred}
implies that it is absolutely irreducible.

Now, since $\mathrm{pcl}(V)$ is of pure dimension $m-s$ and it is
contained in the absolutely irreducible variety $V^*$, of dimension
$m-s$, we deduce that $\mathrm{pcl}(V)=V^*$. Finally, the B\'ezout
Theorem \eqref{eq: Bezout theorem} proves the assertion on the
degree.
\end{proof}

We conclude with the following result, concerning the behavior of
$\mathrm{pcl}(V)$ at infinity.
\begin{corollary}\label{coro: locus F_s at infinity without Pi_(m-1),Pi_m}
With hypotheses as in Theorem \ref{th: pcl(V) is normal abs irred},
the locus of $\mathrm{pcl}(V)$ at infinity is a complete
intersection regular in codimension $k-1$ of degree $d_1\cdots d_s$.
\end{corollary}
\begin{proof}
The fact that $\bfs F_s$ satisfies assumptions $({\sf A_1})$ and
$({\sf A_2})$ implies that $\bfs F_s^{(\bfs d)}$ also satisfies both
assumptions. Therefore, the assertion of the corollary follows
immediately from Theorem \ref{th: pcl(V) is normal abs irred}.
\end{proof}
%
%
\subsection{Estimates for $|V(\fq)|$ and consequences}
In what follows, we shall use estimates on the number of
$\fq$--rational points of singular projective complete intersections
of \cite{CaMaPr15} and \cite{GhLa02a} (see also \cite{CaMa07} or
\cite{MaPePr16} for other estimates).

More precisely, let $V\subset \Pp^n$ be a complete intersection
defined over $\fq$ of dimension $l\geq 2$, which is regular in
codimension $k-1$ for certain $2\le k\le l$. Let $\bfs
d:=(d_1,\ldots,d_{n-l})$ be its multidegree, and let
\begin{align*}
p_l&:=q^l+q^{l-1}+\cdots+q+1=|\Pp^l(\fq)|,\\
\delta&:=\prod_{i=1}^{n-l}d_i,\quad D:=\sum_{i=1}^{n-l}(d_i-1),
\quad d:=\max\{d_1,\ldots,d_{n-l}\}.
\end{align*}
Then the following estimates hold (see \cite[Theorem 1.3]{CaMaPr15}
for the first two estimates and \cite[Theorem 6.1 and Proposition
4.2]{GhLa02a} for the last one):
\begin{align}\label{eq: estimate normal var CaMaPr}
\big||V(\fq)|-p_l\big| &\leq
(\delta(D-2)+2)q^{l-\frac{1}{2}}+14D^2\delta^2 q^{l-1}\textrm{ for
}k=2,\\[1ex]
\label{eq: estimate var regular codim 2 CaMaPr}
\big||V(\fq)|-p_l\big| &\leq 14D^3\delta^2 q^{l-1}\textrm{ for
}k=3,\\
\label{eq: estimate var regular codim k-1 GhLa}
\big||V(\fq)|-p_l\big| &\leq {n+1\choose
l}(d+1)^nq^{\frac{2l-k+1}{2}}+9\cdot 2^{n-l}((n-l)d+3)^{n+1}
q^{\frac{2l-k}{2}}\\\notag&\qquad\qquad\qquad\qquad\qquad
\qquad\qquad\qquad\qquad\qquad\textrm{ for general }k.
\end{align}

Now we are able to estimate the number of points of complete
intersections defined by smooth symmetric systems.
\begin{theorem}\label{th: estimate |V| without Pi_(m-1),Pi_m}
Suppose that $\mathrm{char}(\fq)$ does not divide
$m(m-1)\cdots(m-k+1)$. Let $\bfs F_s$ be as in \eqref{def: f}
satisfying assumptions $({\sf A_1})$ and $({\sf A_2})$, not
depending on $\Pi_{m-k+1},\ldots,\Pi_m$, and let $V:=V(\bfs
F_s)\subset\A^m$. Let $d_i:=\deg F_i$ for $1\le i\le s$,
$\delta:=\prod_{i=1}^sd_i$, $D:=\sum_{i=1}^s(d_i-1)$ and
$d:=\max\{d_1,\ldots,d_s\}$. We have
\begin{align*}
\big||V(\fq)|-q^{m-s}\big|\le
\left\{\begin{array}{l}q^{m-s-\frac{1}{2}}(1+q^{-1})\big((\delta(D-2)+2)+14D^2\delta^2q^{-\frac{1}{2}}\big)\
\textrm{
for }k=2,\\[1ex]
q^{m-s-1}(1+q^{-1})14D^3\delta^2\ \textrm{ for
}k=3.\end{array}\right.
\end{align*}
Further, for $2\le k< m-s$ we have
\begin{align*}
\big||V(\fq)|-q^{m-s}\big|&\le
q^{m-s-\frac{k-1}{2}}(1+q^{-1})\bigg({m+1\choose s+1}(d+1)^m+9\cdot
2^s(sd+3)^{m+1} q^{-\frac{1}{2}}\bigg).
\end{align*}
\end{theorem}
\begin{proof}
Observe that
\begin{align*}
\big||V(\fq)|-q^{m-s}\big|&=
\big||\mathrm{pcl}(V)(\fq)|-|\mathrm{pcl}(V(\fq))^{\infty}|-p_{m-s}+p_{m-s-1}\big|\\
&\le
\big||\mathrm{pcl}(V)(\fq)|-p_{m-s}\big|+\big||\mathrm{pcl}(V(\fq))^{\infty}|-p_{m-s-1}\big|.\end{align*}
Theorem \ref{th: pcl(V) is normal abs irred} asserts that
$\mathrm{pcl}(V)\subset\Pp^{m}$ is a complete intersection defined
over $\fq$ of multidegree $\bfs d:=(d_1,\ldots,d_s)$ which is
regular in codimension $k-1$. On the other hand, by Corollary
\ref{coro: locus F_s at infinity without Pi_(m-1),Pi_m} the locus
$\mathrm{pcl}(V)^{\infty}\subset\Pp^{m-1}$ of $\mathrm{pcl}(V)$ at
infinity is also a complete intersection of multidegree $\bfs
d:=(d_1,\ldots,d_s)$ which is regular in codimension $k-1$.
Therefore, applying \eqref{eq: estimate normal var CaMaPr} and
\eqref{eq: estimate var regular codim 2 CaMaPr} we obtain
\begin{align*}
\big||V(\fq)|-q^{m-s}\big|\le
\left\{\begin{array}{l}q^{m-s-\frac{1}{2}}(1+q^{-1})\big((\delta(D-2)+2)+14D^2\delta^2q^{-\frac{1}{2}}\big)\
\textrm{
for }k=2,\\[1ex]
q^{m-s-1}(1+q^{-1})14D^3\delta^2\ \textrm{ for
}k=3.\end{array}\right.
\end{align*}
On the other hand, for general $k$ we apply \eqref{eq: estimate var
regular codim k-1 GhLa} to obtain
\begin{align*}
\big||\mathrm{pcl}(V)(\fq)|-p_{m-s}\big|&\le
q^{m-s-\frac{k-1}{2}}\bigg({m+1\choose s+1}(d+1)^m+9\cdot
2^{s}(sd+3)^{m+1}
q^{-\frac{1}{2}}\bigg),\\
\big||\mathrm{pcl}(V)^{\infty}(\fq)|-p_{m-s-1}\big|&\le
q^{m-s-\frac{k+1}{2}}\bigg({m\choose s+1}(d+1)^{m-1}+9\cdot
2^s(sd+3)^{m} q^{-\frac{1}{2}}\bigg)\\
&\le q^{m-s-\frac{k+1}{2}}\bigg({m+1\choose s+1}(d+1)^m+9\cdot
2^s(sd+3)^{m+1} q^{-\frac{1}{2}}\bigg).\end{align*}
We readily deduce the theorem.
\end{proof}

If $q\ge 36D^2\delta^2$, then we can slightly simplify the estimate
for $k=2$ to
$$\big||V(\fq)|-q^{m-s}\big|\leq
q^{m-s-\frac{1}{2}}\mbox{$\frac{3}{2}$}\big(D\delta+14D^2\delta^2q^{-\frac{1}{2}}\big)\leq
q^{m-s-\frac{1}{2}}\mbox{$\frac{3}{2}$}\big(1+\mbox{$\frac{14}{6}$}\big)D\delta.$$
Using this estimate we are able to establish a condition which
implies that $V(\fq)$ is nonempty. For $q\ge 36D^2\delta^2$, we have
that $|V(\fq)|>0$ if
$$q^{m-s}-\mbox{$\frac{3}{2}$}\big(1+\mbox{$\frac{14}{6}$}\big)D\delta\, q^{m-s-\frac{1}{2}}> 0,$$
which actually holds. In other words, we have the following result.
\begin{corollary}
Suppose that $\mathrm{char}(\fq)$ does not divide $m(m-1)$. Let
$\bfs F_s$ be as in \eqref{def: f} satisfying assumptions $({\sf
A_1})$ and $({\sf A_2})$, not depending on $\Pi_{m-1},\Pi_m$. If
$q\ge 36D^2\delta^2$, then $|V(\fq)|\ge 1$.
\end{corollary}

In the applications of the next sections not only estimates on the
number of $\fq$--rational points of a complete intersection as above
are required, but also on the number of $\fq$--rational points with
certain pairwise--distinct coordinates, which is the subject of the
next results.
\begin{proposition}
\label{theorem: nb point V_r distinct coordinates} With notations
and assumptions as in Theorem \ref{th: estimate |V| without
Pi_(m-1),Pi_m}, for any $i$ and $j$ with $1\le i<j\le m$ we have
that $V\cap\{X_i=X_j\}$ is of pure dimension $m-s-1$. In particular,
$$|V(\fq)\cap\{X_i=X_j\}|\le \delta\, q^{m-s-1}.$$
\end{proposition}
\begin{proof}
Theorem \ref{th: pcl(V) is normal abs irred} shows that
$\mathrm{pcl}(V)$ is a complete intersection which is regular in
codimension 2. Therefore, by Theorem \ref{th: normal complete int
implies irred} we conclude that it is absolutely irreducible. This
implies that $V$ is also absolutely irreducible.

Without loss of generality we may assume that $i=m-1$ and $j=m$. Let
$\mathcal{I}:=\{I_1,\ldots,I_{m-1}\}$ be the partition of
$\{1,\ldots,m\}$ such that $I_k:=\{k\}$ for $1\le k\le m-2$ and
$I_{m-1}:=\{m-1,m\}$. As $G_1,\ldots,G_s$ do not depend on
$Y_{m-k+1},\ldots,Y_m$, we may consider $G_1,\ldots,G_s$ as elements
of $\cfq[Y_1,\ldots,Y_{m-1}]$. In this sense, the set
$\{X_{m-1}=X_m\}\cap V$ is isomorphic to the set
$$\{\bfs x\in\A^{m-1}: \bfs G_s(\bfs\Pi_{m-1,\mathcal{I}}(\bfs x))=\bfs 0\},$$
where
$\Pi_{m-1,\mathcal{I}}:=\Pi_{m-1}(X_1,\ldots,X_{m-1},X_{m-1})$.
Assumption $({\sf A}_1)$ implies that $\{\bfs
G_s=0\}\subset\A^{m-1}$ is of pure dimension $m-1-s$. Proposition
\ref{prop: Phi_i is a finite morphism} asserts that
$\bfs\Pi_{m-1,\mathcal{I}}$ is a finite morphism. It follows that
$$\bfs\Pi_{m-1,\mathcal{I}}^{-1}(\{\bfs G_s=0\})
=\{\bfs x\in\A^{m-1}:G(\bfs\Pi_{m-1,\mathcal{I}}(\bfs x))=\bfs 0\}$$
is of pure dimension $m-s-1$.

Finally, by the B\'ezout inequality \eqref{eq: Bezout} it follows
that $ \deg V\cap\{X_i=X_j\} \le\deg V$. Then the second assertion
readily follows from \eqref{eq: upper bound -- affine gral}.
\end{proof}
Let $\mathcal{S}$ be a subset of the set $\{(i,j):1\le i<j\le m\}$
and $V^{=}\subset\A^m$ the variety
$$V^{=}:=\bigcup_{(i,j)\in\mathcal{S}}V\cap\{X_i=X_j\}.$$
%
Finally, denote $V^{\not=}:=V\setminus V^=$. We have the following
result.
\begin{corollary}
\label{coro: nb point V_r distinct coordinates} With notations and
assumptions as in Theorem \ref{th: estimate |V| without
Pi_(m-1),Pi_m}, we have
\begin{align*}
\big||V^{\not=}(\fq)|-q^{m-s}\big|\le
\left\{\begin{array}{l}\!\!q^{m-s-\frac{1}{2}}(1+q^{-1})\big((\delta(D-2)+2)+14D^2\delta^2q^{-\frac{1}{2}}\big)
+|\mathcal{S}|\,\delta\, q^{m-s-1}\\[1ex]\hskip8.5cm \textrm{
for }k=2,\\[1ex]
q^{m-s-1}(1+q^{-1})14D^3\delta^2+|\mathcal{S}|\,\delta\, q^{m-s-1}\
\textrm{ for }k=3.\end{array}\right.
\end{align*}
\end{corollary}
\begin{proof}
According to Proposition \ref{theorem: nb point V_r distinct
coordinates},
$$|V^=(\fq)|\le \sum_{(i,j)\in\mathcal{I}}\delta\, q^{m-s-1}
\le |\mathcal{S}|\delta\, q^{m-s-1}.$$
It follows that
\begin{align*}
\big||V^{\not=}(\fq)|-q^{m-s}\big|&\le \big||V(\fq)|-q^{m-s}\big|+
|V^{=}(\fq)|\le \big||V(\fq)|-q^{m-s}\big|+|\mathcal{S}|\,\delta
\,q^{m-s-1}.
\end{align*}
The corollary easily follows from Theorem \ref{th: estimate |V|
without Pi_(m-1),Pi_m}.
\end{proof}
%
%
\section{Sparse hypersurfaces}
\label{section: hypersurfaces not depending on Pi_m}
Now we shall strengthen the conclusions of the previous section for
hypersurfaces. More precisely, we shall consider the case of a
(smooth symmetric) hypersurface not depending on
$\Pi_{m-k+2},\ldots,\Pi_m$. In other words, in this section our
focus will be put on nonzero symmetric $F\in \fq[X_1,\ldots,X_m]$
satisfying assumptions $({\sf A_1})$ and $({\sf A_2})$ that can be
expressed in terms of the elementary symmetric polynomials $\Pi_1
\klk \Pi_{m-k+1}$ of $\fq[X_1 \klk X_m]$, namely
\begin{equation}\label{def: f_1}
F=G(\Pi_1,\ldots,\Pi_{m-k+1}),
\end{equation}
where $G\in\fq[Y_1,\ldots,Y_{m-k+1}]$. We denote by
$V_F\subset\A^{m}$ the hypersurface defined by $F$ and by
$\Sigma_F\subset\A^m$ the singular locus of $V_F$.

Theorems \ref{th: estimate |V| without Pi_(m-1),Pi_m} provides an
estimate for the cardinality of $V_F(\fq)$ when $F$ does not depend
on $\Pi_{m-k+1},\ldots,\Pi_m$. Therefore, we shall concentrate on
the remaining case when $F$ does depend on $\Pi_{m-k+1}$. By Lemma
\ref{lemma: singular locus F_s not depending on Pi_(m-k+1),..,Pi_m}
we conclude that all the elements in the singular locus $\Sigma_F$
have at most $m-k+1$ pairwise distinct coordinates. In particular,
for any $\bfs x:=(x_1,\ldots,x_m)\in\Sigma_F$ there exists a
partition $\mathcal{I}:=\{I_1,\ldots,I_{m-k+1}\}$ of
$\{1,\ldots,m\}$ into $m-k+1$ nonempty subsets $I_j\subset
\{1,\ldots,m\}$ such that
\begin{equation}\label{eq: singular locus F_s without Pi(m)}
\bfs x\in \mathcal{L}_\mathcal{I}\cap V_F,\end{equation}
where $\mathcal{L_{I}}$ is the linear variety
$$
\mathcal{L}_{\mathcal I}:=
\mathrm{span}(\bfs{v}^{(I_1)},\ldots,\bfs{v}^{(I_{m-k+1})})
  $$
spanned by the vectors $\bfs{v}^{(I_j)}:=
(v_1^{(I_j)},\ldots,v_{m}^{(I_j)})\in\{0,1\}^m$ defined by
$v_l^{(I_j)}:=1$ iff $l\in I_j$.

Fix a partition $\mathcal{I}:=\{I_1,\ldots,I_{m-k+1}\}$ as above.
Assume without loss of generality that $i\in I_i$ for $1\le i\le
m-k+1$, and consider the mapping
$$\begin{array}{rcl}
\bfs\Pi_{m-k+1,\mathcal{I}}:\A^{m-k+1}&\to&\A^{m-k+1},\\
(x_1,\ldots,x_{m-k+1})&\mapsto&
\bfs\Pi_{m-k+1}(x_1\bfs{v}^{(I_1)}+\cdots+x_{m-k+1}\bfs{v}^{(I_{m-k+1})}),
\end{array}$$
with $\bfs{v}^{(I_1)},\ldots,\bfs{v}^{(I_{m-k+1})}$ as above. We
have that any $\bfs x:=(x_1,\ldots,x_m)\in
\mathcal{L}_\mathcal{I}\cap V_F$ belongs to the image
$L_\mathcal{I}(V(G\circ\bfs\Pi_{m-k+1,\mathcal{I}}))$ of
$V(G\circ\bfs\Pi_{m-k+1,\mathcal{I}})$ under $L_\mathcal{I}$.

Now we are able to establish a characterization of the intersections
$\mathcal{L}_\mathcal{I}\cap V_F$.
\begin{theorem}\label{th: singular locus for F without Pi(m)}
Suppose that $\mathrm{char}(\fq)$ does not divide
$m(m-1)\cdots(m-k)$. Let $F$ be a polynomial as in \eqref{def: f_1}
not depending on $\Pi_{m-k+2},\ldots,\Pi_{m}$ such that assumption
$({\sf A}_1)$ holds. Then $\mathcal{L}_\mathcal{I}\cap V_F$ has pure
dimension $m-k$ for any partition
$\mathcal{I}:=\{I_1,\ldots,I_{m-k+1}\}$ of $\{1,\ldots,m\}$.
\end{theorem}
\begin{proof}
Fix a partition $\mathcal{I}:=\{I_1,\ldots,I_{m-k+1}\}$ of
$\{1,\ldots,m\}$. As $G$ does not depend on $Y_{m-k+2},\ldots,Y_m$,
we may consider $G$ as an element of $\cfq[Y_1,\ldots,Y_{m-k+1}]$.
In this sense, the set $\mathcal{L}_\mathcal{I}\cap V_F$ is
isomorphic to the set
$$\{\bfs x\in\A^{m-k+1}: G(\bfs\Pi_{m-k+1,\mathcal{I}}(\bfs x))=\bfs 0\}.$$
Assumption $({\sf A}_1)$ implies that $\{G=0\}\subset\A^{m-k+1}$ is
of pure dimension $m-k$. Proposition \ref{prop: Phi_i is a finite
morphism} asserts that $\bfs\Pi_{m-k+1,\mathcal{I}}$ is a finite
morphism. It follows that
$$\bfs\Pi_{m-k+1,\mathcal{I}}^{-1}(\{G=0\})=\{\bfs x\in\A^{m-k+1}:G(\bfs\Pi_{m-k+1,\mathcal{I}}(\bfs x))=\bfs 0\}$$
is of pure dimension $m-k$.
\end{proof}

\begin{corollary} \label{coro: singular locus V_F}
With hypotheses as in Theorem \ref{th: singular locus for F without
Pi(m)}, if $\dim\Sigma_F= m-k$ and $\mathcal{C}$ is an irreducible
$\fq$-component of $\Sigma_F$ of dimension $m-k$, then there exist a
partition $\mathcal{I}:=\{I_1,\ldots,I_{m-k}\}$ of $\{1,\ldots,m\}$
such that $\mathcal{C}=\mathcal{L}_\mathcal{I}$.
\end{corollary}
\begin{proof}
According to Lemma \ref{lemma: singular locus F_s not depending on
Pi_(m-k+1),..,Pi_m}, any irreducible component $\mathcal{C}$ as in
the statement of the corollary satisfies
$$\mathcal{C}\subset \bigcup_{\mathcal{I}} \mathcal{L}_\mathcal{I}\cap
V_F \subset \bigcup_{\mathcal{I}}\mathcal{L}_\mathcal{I},$$
where $\mathcal{I}$ runs over all the partitions
$\mathcal{I}:=\{I_1,\ldots,I_{m-k}\}$ of $\{1,\ldots,m\}$ as above.
As each linear variety $\mathcal{L}_\mathcal{I}$ as in the
right-hand side is irreducible of dimension $m-k$, the conclusion in
the corollary readily follows.
\end{proof}

Now assume that $\dim\Sigma_F= m-k$ and let $\mathcal{C}$ be an
irreducible $\fq$-component of $\Sigma_F$ as in Corollary \ref{coro:
singular locus V_F}. Without loss of generality we may assume that
$j\in I_j$ for $1\le j\le m-k$. If $\bfs x\in \mathcal{C}$, then
$\nabla F({{\bfs x}})=\nabla G (\bfs\Pi_{m-k+1}({{\bfs x}}))\cdot
J\,\bfs\Pi_{m-k+1}({{\bfs x}})=\bfs 0$, where $\nabla$ denotes the
gradient operator. As a consequence, we have
$$
\sum_{j=1}^{m-k}\dfrac{\partial G}{\partial
Y_j}(\bfs\Pi_{m-k+1}({{\bfs x}}))\,\dfrac{\partial \Pi_j}{\partial
    X_l}({\bfs x})
=-\dfrac{\partial G}{\partial Y_{m-k+1}}(\bfs\Pi_{m-1}({{\bfs
x}}))\,\dfrac{\partial \Pi_{m-k+1}}{\partial X_l}({\bfs x})$$
for $1\le l\le m$. This shows that the following matrix identity
holds:
$$
\left(
\begin{array}{cccccc}
\dfrac{\partial \Pi_{1}}{\partial X_1}({\bfs x}) & \cdots &
\dfrac{\partial \Pi_{m-k}}{\partial X_{1}}({\bfs x})
\\
\vdots & & \vdots
\\
\dfrac{\partial \Pi_1}{\partial X_{m}}({\bfs x}) & \cdots &
\dfrac{\partial \Pi_{m-k}}{\partial X_{m}}({\bfs x})
\end{array}
\right)
\begin{pmatrix} \dfrac{\partial G}{\partial Y_1}(\bfs\Pi_{m-k+1}({{\bfs
x}}))\, \\ \vdots \\ \dfrac{\partial G}{\partial
Y_{m-k}}(\bfs\Pi_{m-k+1}({{\bfs x}}))\end{pmatrix}$$
$$=\dfrac{\partial G}{\partial Y_{m-k+1}}(\bfs\Pi_{m-k+1}({{\bfs
x}}))\begin{pmatrix} -\dfrac{\partial \Pi_{m-k+1}}{\partial X_1}({\bfs x})\\
\vdots \\ -\dfrac{\partial \Pi_{m-k+1}}{\partial X_{m}}({\bfs
x})\end{pmatrix}.
$$
Considering the first $m-k$ equations of this system we obtain the
$(m-k)\times(m-k)$ system
\begin{equation}\label{eq: system a}
B({\bfs x})
\begin{pmatrix} \dfrac{\partial G}{\partial Y_1}(\bfs\Pi_{m-k+1}({{\bfs
x}}))\, \\ \vdots \\ \dfrac{\partial G}{\partial
Y_{m-k}}(\bfs\Pi_{m-k+1}({{\bfs x}}))\end{pmatrix}
=\dfrac{\partial G}{\partial Y_{m-k+1}}(\bfs\Pi_{m-k+1}({{\bfs x}}))
\begin{pmatrix} -\dfrac{\partial \Pi_{m-k+1}}{\partial X_1}({\bfs x})\\
\vdots \\ -\dfrac{\partial \Pi_{m-k+1}}{\partial X_{m-k}}({\bfs
x})\end{pmatrix},
\end{equation}
where $B\in \fq[\bfs X]^{(m-k) \times (m-k)}$ is the matrix
\begin{equation}\label{eq: definition B}
B:=\left(
\begin{array}{cccccc}
\dfrac{\partial \Pi_{1}}{\partial X_1} &  \cdots & \dfrac{\partial
\Pi_{m-k}}{\partial X_{1}}
\\
\vdots &  & \vdots
\\
\dfrac{\partial \Pi_1}{\partial X_{m-k}} & \cdots & \dfrac{\partial
\Pi_{m-k}}{\partial X_{m-k}}
\end{array}
\right).\end{equation}

Suppose that $\frac{\partial G}{\partial
Y_{m-k+1}}(\bfs\Pi_{m-k+1}({{\bfs x}}))=0$ for any $\bfs
x\in\mathcal{C}$. Since $\nabla G(\bfs y)\not=\bfs 0$ for any $\bfs
y\in V(G)$ and $G(\bfs\Pi_{m-k+1}(\bfs x))=0$ for any $\bfs
x\in\mathcal{C}$, the vector in the left-hand side of \eqref{eq:
system a} is a nonzero solution of \eqref{eq: system a} for any
$\bfs x\in\mathcal{C}$, which implies that the matrix $B(\bfs x)$ is
singular for any $\bfs x\in\mathcal{C}$. Now we obtain an explicit
expression for the determinant of $B$.
\begin{lemma}\label{lemma: det B} We have
$$\det(B)=(-1)^{\frac{(m-k-1)(m-k)}{2}}\prod_{1\le
j < k\le m-k}(X_j-X_k).$$
\end{lemma}
\begin{proof}
From \eqref{eq: derivative Pi_i} we deduce that $B^t$ can be
factored as
\begin{equation} \label{eq: factorization B} B^t=\left(
\begin{array}{ccccc}
    1 & \ 0 & 0 &  \dots & 0
    \\
    \ \ \Pi_1 &  -1 & 0 &  &
    \\
    \Pi_2 & \ \ -\Pi_1 & 1 & \ddots & \vdots
    \\
    \vdots &\vdots  & \vdots & \ddots & 0
    \\
    \Pi_{m-k-1} & -\Pi_{m-k-2} &\Pi_{m-k-3} & \cdots &\!\! (-1)^{m-k-1}
\end{array}
\!\!\right)\cdot A_{m-k},
\end{equation}
where
$A_{m-k}$ is the $(m-k)\times (m-k)$--Vandermonde matrix
$$
A_{m-k}:=(X_j^{i-1})_{1\leq i,j\leq m-k}.
$$
This readily implies the lemma. \end{proof}

It follows that
$$\mathcal{C}=\bigcup_{1\le
k < l\le m-k}\mathcal{C}\cap\{X_k=X_l\}.$$
As the union in the right-hand side has dimension at most $m-k-1$,
we conclude that $\dim\Sigma_F\le m-k-1$, contradicting thus our
hypotheses.

Therefore, neither $\frac{\partial G}{\partial
Y_{m-k+1}}(\bfs\Pi_{m-k+1})$ nor $\det(B)$ vanish identically on
$\mathcal{C}$. In particular, there exists a nonempty Zariski open
subset $\mathcal{U}$ of $\mathcal{C}$ such that, for any $\bfs
x\in\mathcal{U}$, we have that $\frac{\partial G}{\partial
Y_{m-k+1}}(\bfs\Pi_{m-k+1}({{\bfs x}}))\not=0$, the matrix $B(\bfs
x)$ is nonsingular, and the vector $\big(\frac{\partial G}{\partial
Y_1}(\bfs\Pi_{m-k+1}({{\bfs x}})),\ldots, \frac{\partial G}{\partial
Y_{m-k}}(\bfs\Pi_{m-k+1}({{\bfs x}}))\big)$ is the unique solution
of \eqref{eq: system a}. The Cramer rule implies
\begin{equation}\label{eq: identity a_j}
\frac{\frac{\partial G}{\partial Y_j}(\bfs\Pi_{m-k+1}({{\bfs
x}}))}{\frac{\partial G}{\partial Y_{m-k+1}}(\bfs\Pi_{m-k+1}({{\bfs
x}}))}=\frac{\det(B^j({\bfs x}))}{\det(B({\bfs x}))} \quad (1 \leq j
\leq m-k),\end{equation}
where $B^j\in \fq[\bfs X]^{(m-k) \times (m-k)}$ is the matrix
obtained by replacing the $j$th column of $B$ by the vector
$(-{\partial \Pi_{m-k+1}}/{\partial X_1} \klk -{\partial
\Pi_{m-k+1}}/{\partial X_{m-k}})^t$. Next we obtain an explicit
expression for $\det(B^j)$.
\begin{proposition}\label{prop: det Bj=0} For $1\le j\le m-k$, we have
$\det(B^j)=0$.
\end{proposition}
\begin{proof}
For $j=1$, combining \eqref{eq: factorization B} and \eqref{eq:
derivative Pi_i} for $i=m-k+1$ we see that
\begin{align}\label{eq: factorization B1}
(B^1)^t&=C_1 \cdot A_{m-k}^*\\ \notag &:=\left(
\begin{array}{ccccccc}
-\Pi_{m-k} & \Pi_{m-k-1} & \ldots &  \ldots & (-1)^{m-k}\Pi_1
&(-1)^{m-k+1}
\\
\Pi_1 &  -1 & 0 & \ldots & 0 &0
\\
\Pi_2 & -\Pi_1 & 1 & \ddots & \vdots & \vdots
\\
\vdots &\vdots  & \vdots & \ddots & 0 & 0
\\
\Pi_{m-k-1} & -\Pi_{m-k-2} &\Pi_{m-k-3} & \ldots &\!\! (-1)^{m-k-1}
& 0
\end{array}
\!\!\right) \cdot A_{m-k}^*,
\end{align}
where $A_{m-k}^*:=(X_j^{i-1})_{1\leq i\leq m-k+1, \, 1 \leq j \leq
m-k}.$ By the Cauchy--Binet formula, it follows that
\begin{equation}\label{eq: determinant B1t}
\det (B^1)=\sum_{i=1}^{m-k+1} \det (C_{1,i})\cdot \det
(A_{m-k,i}^*),
\end{equation}
where $C_{1,i}$ and $A_{m-k,i}^*$ are the $(m-k)\times
(m-k)$--matrices obtained by removing the $i$th column of $C_1$ and
the $i$th row of $A_{m-k}^*$ respectively. From \cite[Lemma
2.1]{Ernst00} (see also \cite[Exercise 281]{FaSo65}) we deduce that
\begin{equation}\label{eq: determinant Am-1i}
\det(A_{m-k,i}^*)= \Pi_{m-k+1-i}^*\cdot \prod_{1\le j < k\le
m-k}(X_j-X_k)=(-1)^{\frac{(m-k)(m-k-1)}{2}} \,
\Pi_{m-k+1-i}^*\,\det(B),
\end{equation}
where $\Pi_{m-k+1-i}^*$ is the $(m-k+1-i)$-th elementary symmetric
polynomial in $X_1 \klk X_{m-k}$.

By elementary calculations we deduce that, for $1 \leq i \leq
m-k+1$,
\begin{equation} \label{eq: determinant C1i}
\det (C_{1,i})=
    (-1)^{\frac{(m-k+1)(m-k+2)}{2}+i} \det(\mathcal{T}_{i-1}),
\end{equation}
where $\mathcal{T}_{0}:=1$ and, for $i\geq 2$, $\mathcal{T}_{i-1}$
is the $(i-1) \times (i-1)$ Toeplitz--Hessenberg matrix
$$\mathcal{T}_{i-1}:=\left(
\begin{array}{ccccc}
\Pi_1 & -1 & 0 &  \dots & 0
\\
-\Pi_2\ \ &  \ddots & \ddots&  & \vdots
\\
\vdots &\ddots  & \ddots & \ddots & 0
\\
\vdots &  & \ddots & \ddots &\!\! -1\,
\\
(-1)^{i}\Pi_{i-1} & \dots & \dots & -\Pi_2&\!\! \Pi_1
\end{array}
\!\right).$$
By the Trudi formula (see \cite[Ch. VII]{Muir60}; see also
\cite[Theorem 1]{Merca13}) we deduce that
$\det(\mathcal{T}_{i-1})=H_{i-1}$, where $H_{i-1}$ is the $(i-1)$th
complete homogeneous symmetric polynomial of $\fq[\bfs X]$ (see
\cite[Section 4]{Merca13}). Therefore, combining \eqref{eq:
determinant B1t}, \eqref{eq: determinant Am-1i} and \eqref{eq:
determinant C1i} we conclude that
\begin{align*}
\det(B^1)&=\det(B)\, \sum_{i=1}^{m-k+1}(-1)^{i-1}H_{i-1} \Pi_{m-k+1-i}^*\\
&=\det(B)\,
\sum_{i=1}^{m-k+1}\sum_{l=0}^{i-1}(-1)^{i-1}X_{m-k+1}^lH_{i-1-l}^*
\Pi_{m-k+1-i}^*\\
&=\det(B)\,
\sum_{l=0}^{m-k}\Bigg(\sum_{i=l+1}^{m-k+1}(-1)^{i-1}H_{i-1-l}^*
\Pi_{m-k+1-i}^*\Bigg)X_{m-k+1}^l\\
&=\det(B)\, \sum_{l=0}^{m-k}\Bigg(\sum_{i=0}^{m-k-l}(-1)^{i+l}H_i^*
\Pi_{m-k-l-i}^*\Bigg)X_{m-k-1}^l=0,
\end{align*}
where the last identity is due to the well-known relations between
complete and elementary symmetric polynomials (see, e.g., \cite[\S
I.2, (2.6')]{Macdonald98}).

As the proof for the case $j \geq 2$ is similar, we shall be brief.
As in the case $j=1$, we have that $(B^j)^t$ can be factored as
\begin{align}\label{eq: factorization Bj}
(B^j)^t&:=C_j \cdot A_{m-k}^*\\\notag&=\left(
\begin{array}{cccccc}
1 & 0 & &  \ldots & 0 & 0
\\
 \Pi_1 &  -1 & 0 &  &
\\
\vdots &\vdots  & & \ddots &  & \vdots
\\
\Pi_{j-2} & -\Pi_{j-3} & \ldots &  \ldots & 0 & 0
\\
-\Pi_{m-k} & \Pi_{m-k-1} & \ldots &  \ldots & (-1)^{m-k}\Pi_1&
(-1)^{m-k+1}
\\
\Pi_{j} & -\Pi_{j-1} & \Pi_{j-2} &  \ldots & 0 & 0
\\
\vdots &\vdots  & \vdots & \ddots & \vdots & \vdots
\\
\Pi_{m-k-1} & -\Pi_{m-k-2} &\Pi_{m-k-3} & \cdots & (-1)^{m-k-1} & 0
\end{array}
\!\!\right) \cdot A_{m-k}^*.
\end{align}

Applying the Cauchy--Binet formula to the factorization \eqref{eq:
factorization Bj}, we obtain
\begin{equation}\label{eq: determinant Bjt}
\det (B^j)=\sum_{i=1}^{m-k+1} \det (C_{j,i})\cdot \det
(A_{m-k,i}^*),
\end{equation}
where $C_{j,i}$ and $A_{m-k,i}^*$ are the $(m-k)\times
(m-k)$--matrices obtained by removing the $i$th column of $C_j$ and
the $i$th row of $A_{m-k}^*$ respectively. Elementary calculations
show that
\begin{equation} \label{eq: determinant Cji}
\det (C_{j,i})=\left\{\begin{array}{cl}
    0 & \textrm{for }1\le i \le j-1, \\
    (-1)^{\frac{(m-k+1)(m-k+2)}{2}+1+\alpha_{ij}} \det(\mathcal{T}_{i-j}) & \textrm{for
    } j \leq i \leq m-k+1,
\end{array}\right.
\end{equation}
where $\alpha_{ij}:=\frac{j(j-1)-i(i-1)}{2}+(j-1)(i-j)+\lfloor
\frac{i-j}{2} \rfloor=(j-i)\frac{i-j+1}{2}+\lfloor \frac{i-j}{2}
\rfloor\equiv i-j\mod 2$ and $\mathcal{T}_{i-j}$ is the $(i-j)
\times (i-j)$ Toeplitz--Hessenberg matrix

$$\mathcal{T}_{i-j}:=\left(
\begin{array}{ccccc}
\Pi_1 & -1 & 0 &  \dots & 0
\\
-\Pi_2\ \ &  \ddots & \ddots&  & \vdots
\\
\vdots &\ddots  & \ddots & \ddots & 0
\\
\vdots &  & \ddots & \ddots &\!\! -1\,
\\
(-1)^{i-j+1}\Pi_{i-j} & \dots & \dots & -\Pi_2&\!\! \Pi_1
\end{array}
\!\right)$$
for $i-j \geq 1$, while $\mathcal{T}_{0}:=1$. The Trudi formula
implies $\det(\mathcal{T}_{i-j})=H_{i-j}$, where $H_{i-j}$ is the
$(i-j)$th complete homogeneous symmetric polynomial of $\fq[\bfs
X]$. Therefore, combining \eqref{eq: determinant Bjt}, \eqref{eq:
determinant Am-1i} and \eqref{eq: determinant Cji} we conclude that
\begin{align*}
\det(B^j)&=\det(B)\, \sum_{i=j}^{m-k+1}(-1)^{i-j}H_{i-j} \Pi_{m-k+1-i}^*\\
&=\det(B)\,
\sum_{i=j}^{m-k+1}\sum_{l=0}^{i-j}(-1)^{i-j}X_{m-k+1}^kH_{i-j-l}^*
\Pi_{m-k+1-i}^*\\
&=\det(B)\,
\sum_{l=0}^{m-k+1-j}\Bigg(\sum_{i=l+j}^{m-k+1}(-1)^{i-j}H_{i-j-l}^*
\Pi_{m-k+1-i}^*\Bigg)X_{m-k+1}^l\\
&=\det(B)\,
\sum_{l=0}^{m-k+1-j}\Bigg(\sum_{i=0}^{m-k+1-l-j}(-1)^{i+l}H_i^*
\Pi_{m-k+1-l-j-i}^*\Bigg)X_{m-k+1}^l=0,
\end{align*}
where the last identity is due to, e.g., \cite[\S I.2,
(2.6')]{Macdonald98}.
This finishes the proof of the proposition.
\end{proof}
It follows that
$$\frac{\partial G}{\partial
Y_j}(\bfs\Pi_{m-k+1}({{\bfs x}}))=0\quad(1\le j\le m-k)$$
for any $\bfs x\in\mathcal{U}$, and thus for any $\bfs x\in
\mathcal{C}=\mathcal{L}_\mathcal{I}$. If there is $1\le j\le m-k$
such that $\partial G/Y_j\not=0$, as $G$ is square-free we have that
the resultant $\textrm{Res}(G,\partial G/Y_j)$ is nonzero. This
implies that $G$ and $\partial G/\partial Y_j$ intersect properly.
Since $\Pi_{m-k+1,\mathcal{I}}(\mathcal{C})\subset V(G)\cap
V(\partial G/\partial Y_j)$ and $V(G)$ $V(\partial G/\partial Y_j)$
intersect properly in $\A^{m-k+1}$, we conclude that $\dim V(G)\cap
V(\partial G/\partial Y_j)\le m-k-1$, and thus
$\dim(\Pi_{m-k+1,\mathcal{I}}(\mathcal{C}))\le m-k-1$. Therefore,
Proposition \ref{prop: Phi_i is a finite morphism} implies that
$\mathcal{C}$ has dimension at most $m-k-1$, contradicting our
hypothesis.

Finally, we analyze the dimension of $\Sigma_F$  when $F$ is of the
form
$$F=f(\Pi_{m-k+1})\quad(f\in\fq[T]\textrm{ square-free}).$$
%
According to \eqref{eq: singular locus F_s without Pi(m)},
\begin{equation}\label{eq: sigma_F caso particular}
\Sigma_F\subset \bigcup_{\mathcal{I}}
\{f(\Pi_{m-k+1})=0\}\cap\mathcal{L}_\mathcal{I},
\end{equation}
where $\mathcal{I}$ runs over the partitions of $\{1,\ldots,m\}$ as
above. Observe that $\Pi_{m-k+1}$ does not vanish on any
$\mathcal{L}_\mathcal{I}$ as above. Indeed, let
$\mathcal{I}:=\{I_1,\ldots,I_{m-k+1}\}$ be one such partition, and
let $\mathcal{L}_{\mathcal I}:=
\mathrm{span}(\bfs{v}^{(I_1)},\ldots,\bfs{v}^{(I_{m-k+1})})$ be the
corresponding linear variety, where each $\bfs{v}^{(I_j)}:=
(v_1^{(I_j)},\ldots,v_{m}^{(I_j)})\in\{0,1\}^m$ is defined by
$v_l^{(I_j)}:=1$ iff $l\in I_j$. The image of the map
$(x_1,\ldots,x_{m-k+1})\mapsto
x_1\bfs{v}^{(I_1)}+\cdots+x_{m-k+1}\bfs{v}^{(I_{m-k+1})}$
parametrizes $\mathcal{L}_{\mathcal I}$. In particular,
$\Pi_{m-k+1}$ does not vanish on $\mathcal{L}_\mathcal{I}$ if and
only if the polynomial
$\Pi_{m-k+1}(X_1\bfs{v}^{(I_1)}+\cdots+X_{m-k+1}\bfs{v}^{(I_{m-k+1})})$
is nonzero. By substituting 1 for $X_i$ for $1\le i\le m-k+1$ we
obtain
$\Pi_{m-k+1}(\bfs{v}^{(I_1)}+\cdots+\bfs{v}^{(I_{m-k+1})})={m\choose
m-k+1}$, which is nonzero if $\mathrm{char}(\fq)$ does not divide
$m(m-1)\cdots(m-k+2)$. This proves our claim.

It follows $f(\Pi_{m-k+1})$ does not vanish on any
$\mathcal{L}_\mathcal{I}$, and thus all the intersections in the
right-hand side of \eqref{eq: sigma_F caso particular} have
dimension at most $m-k-1$, which contradicts our hypothesis. As a
consequence, we have the following result.
\begin{theorem}\label{th: sing locus V(F) without Pi_m}
Suppose that $\mathrm{char}(\fq)$ does not divide
$m(m-1)\cdots(m-k)$. For $F$ as in \eqref{def: f_1}, satisfying
assumptions $({\sf A_1})$ and $({\sf A_2})$, the singular locus
$\Sigma_F$ has dimension at most $m-k-1$.
\end{theorem}

Let $d:=\deg F$ and write $F=F^{(d)}+F^*$, where $F^{(d)}$ denotes
the homogeneous component of $F$ of degree $d$. Observe that, if $F$
satisfies assumptions $({\sf A_1})$ and $({\sf A_2})$, then so does
$F^{(d)}$.

As before, we consider the embedding of $\A^m$ into the projective
space $\Pp^m$ defined by the mapping $(x_1,\ldots,
x_m)\mapsto(1:x_1:\dots:x_m)$ and the projective closure
$\mathrm{pcl}(V_F)\subset\Pp^m$ of the image of $V_F$ under this
embedding.

According to Theorem \ref{th: sing locus V(F) without Pi_m}, the
singular locus  $\Sigma_F$ has dimension at most $m-k-1$. Actually,
the proof of Theorem \ref{th: sing locus V(F) without Pi_m} shows a
stronger assertion: the set of points $\bfs x\in\A^m$, for which
simultaneously $F(\bfs x)=0$ and $\nabla F(\bfs x)=\bfs 0$ hold, has
dimension at most $m-k-1\le m-2$. As a consequence, by, e.g.,
\cite[Theorem 18.15]{Eisenbud95}, the ideal generated by $F$ is
radical, and thus $F$ is square-free. In particular, $F$ is a
minimal defining polynomial for $V_F\subset\A^m$. It follows that
the homogenization $F^h$ of $F$ is square-free, and thus is a
minimal defining polynomial for $\mathrm{pcl}(V_F)\subset\Pp^m$.

We denote by $\mathrm{pcl}(V_F)^\infty$ the points of
$\mathrm{pcl}(V_F)$ at infinity. As $F^h=F^{(d)}+X_0^{e}(F^*)^h$ for
a suitable $e\ge 1$, then $F^{(d)}$ defines
$\mathrm{pcl}(V_F)^\infty$. In particular, as $F^{(d)}$ satisfies
assumptions $({\sf A_1})$ and $({\sf A_2})$, arguing as above we
deduce that $F^{(d)}$ is a minimal defining polynomial for
$\mathrm{pcl}(V_F)^\infty\subset\Pp^{m-1}$.

Finally, denote by $\Sigma_F^{\infty}\subset\mathbb{P}^{m}$ the
singular locus of $\mathrm{pcl}(V_F)$ at infinity. Taking into
account that $F^{(d)}$ satisfies assumptions $({\sf A_1})$ and
$({\sf A_2})$, we have the following immediate consequence of
Theorem \ref{th: sing locus V(F) without Pi_m}
\begin{corollary}\label{coro: singular locus at infinity without Pi_m}
With hypothesis as in Theorem \ref{th: sing locus V(F) without
Pi_m}, the singular locus $\Sigma_F^{\infty}$ of $\mathrm{pcl}(V_F)$
at infinity has dimension at most $m-k-2$.
\end{corollary}
\begin{proof}
Let $W:=V(F^{(d)})\subset\A^m$ be the affine cone defined by
$F^{(d)}$. By Theorem \ref{th: sing locus V(F) without Pi_m} we have
that the singular locus $\mathrm{Sing}(W)$ of $W$  has dimension at
most $m-k-1$. The corresponding projective variety is
$V(F^{(d)})=\mathrm{pcl}(V_F)\cap\{X_0=0\}\subset\Pp^{m-1}$, and its
singular locus, namely $\Sigma_F^{\infty}\subset\Pp^{m-1}$, has
dimension at most $m-k-2$.
\end{proof}

Now we are able to establish our main geometric characterization of
$\mathrm{pcl}(V_F)$ and $\mathrm{pcl}(V_F)^\infty$.
\begin{theorem}\label{th: pcl(V_F) is a normal hypersurface}
Suppose that $\mathrm{char}(\fq)$ does not divide
$m(m-1)\cdots(m-k)$. Let $F$ be as in \eqref{def: f_1}, satisfying
assumptions $({\sf A_1})$ and $({\sf A_2})$. Then the hypersurfaces
$\mathrm{pcl}(V_F)$ and $\mathrm{pcl}(V_F)^\infty$ are regular in
codimension $k-1$.
\end{theorem}
\begin{proof}
Corollary \ref{coro: singular locus at infinity without Pi_m} shows
that the singular locus of $\mathrm{pcl}(V_F)$ at infinity has
dimension at most $m-k-2$. Therefore,
$\mathrm{pcl}(V_F)^\infty=\mathrm{pcl}(V_F)\cap\{X_0=0\}$ is a
hypersurface of $\{X_0=0\}\subset\Pp^m$ regular in codimension
$k-1$. On the other hand, according to Theorem \ref{th: sing locus
V(F) without Pi_m}, the singular locus of
$\mathrm{pcl}(V_F)\cap\{X_0\not=0\}$ has dimension at most $m-k-1$.
This and Corollary \ref{coro: singular locus at infinity without
Pi_m} imply that the singular locus of $\mathrm{pcl}(V_F)$ has
dimension at most $m-k-1$, and thus that $\mathrm{pcl}(V_F)$ is
regular in codimension $k-1$.
\end{proof}
%
%
\subsection{Estimates on $|V_F(\fq)|$}
Now we estimate the number of $\fq$-rational zeros of any $F$ as
above. Let $F$ be as in \eqref{def: f_1}, satisfying assumptions
$({\sf A_1})$ and $({\sf A_2})$. We have
\begin{align*}
\big||V_F(\fq)|-q^{m-1}\big|=&
\big||\mathrm{pcl}(V_F)(\fq)|-|\mathrm{pcl}(V_F(\fq))^{\infty}|-p_{m-1}+p_{m-2}\big|\\
&\le
\big||\mathrm{pcl}(V_F)(\fq)|-p_{m-1}\big|+\big||\mathrm{pcl}(V_F(\fq))^{\infty}|-p_{m-2}\big|.
\end{align*}

According to Theorem \ref{th: pcl(V_F) is a normal hypersurface},
the hypersurfaces $\mathrm{pcl}(V_F)\subset\Pp^{m}$ and
$\mathrm{pcl}(V_F)^\infty\subset\Pp^{m-1}$ are regular in
codimension $k-1$ and defined over $\fq$. As a consequence,
\eqref{eq: estimate normal var CaMaPr} and \eqref{eq: estimate var
regular codim 2 CaMaPr} imply
\begin{align*}
\big||\mathrm{pcl}(V_F)(\fq)|-p_{m-1}\big| &\leq
\left\{\begin{array}{c}(d-1)(d-2)q^{m-\frac{3}{2}}+14(d-1)^2d^2q^{m-2}
\textrm{ for }k=2,\\[1ex]
14(d-1)^3d^2q^{m-2} \textrm{ for }k=3,
\end{array}\right.\\
\big||\mathrm{pcl}(V_F)^\infty(\fq)|-p_{m-2}\big|
&\le \left\{\begin{array}{c} (d-1)(d-2)q^{m-\frac{5}{2}}+14(d-1)^2d^2q^{m-3}\textrm{ for }k=2,\\[1ex]
14(d-1)^3d^2q^{m-3}\textrm{ for }k=3.
\end{array}\right.
\end{align*}
It follows that
\begin{align*}
\big||V_F(\fq)|-q^{m-1}\big| &\le \left\{\begin{array}{c}
q^{m-\frac{3}{2}}(1+q^{-1})\big((d-1)(d-2)+14(d-1)^2d^2q^{-\frac{1}{2}}\big)
\textrm{ for }k=2,\\[1ex]
q^{m-2}(1+q^{-1})14(d-1)^3d^2\textrm{ for }k=3.
\end{array}\right.
\end{align*}

On the other hand, for general $k$ we show that the general bound
\eqref{eq: estimate var regular codim k-1 GhLa} can be improved. In
\cite[Theorem 6.1]{GhLa02a} the authors prove that, for an
hypersurface $V\subset\Pp^l$ defined over $\fq$, of degree $d\ge 2$
and regular in codimension $k-1$, with $2\le k< l$, we have
$$
\big||V(\fq)|-p_{l-1}\big|\leq
b_{k-1,d}\,q^{\frac{2l-k-1}{2}}+C_{l-1-k,l-1}(V)q^{\frac{2l-k-2}{2}},
$$
where $b_{k-1,d}$ is the $(k-1)$th primitive Betti number of any
nonsingular hypersurface in $\Pp^{l-s-1}$ of degree $d$, which is
upper bounded by
$$b_{k-1,d} \le
\frac{d-1}{d}\big((d-1)^k-(-1)^k\big)\le (d-1)^k,$$
while  $C_{l-k-1,l-1}(V)$ is the sum
  $$
    C_{l-k-1,l-1}(V):=\sum_{i=l-1}^{2l-k-2}b_{i,\ell}(V)+\varepsilon_i,
  $$
where $b_{i,\ell}(V)$ denotes the $i$th $\ell$--adic Betti number of
$V$ for a prime $\ell$ different from $p:=\mathrm{char}(\fq)$ and
$\varepsilon_i:=1$ for even $i$ and $\varepsilon_i:=0$ for odd $i$.
According to \cite[Lemma 5.1]{CaMaPr12}, we have the bound
$$C_{l-k-1,l-1}(V) \le  6(d+2)^{l+1}.$$
As a consequence, we have the bound
$$\big||V(\fq)|-p_{l-1}\big|\leq (d-1)^k\,q^{\frac{2l-k-1}{2}}+
6(d+2)^{l+1}q^{\frac{2l-k-2}{2}}.$$
Now we apply this bound to estimate $|\mathrm{pcl}(V_F)(\fq)|$ and
$|\mathrm{pcl}(V_F(\fq))^{\infty}|$. We have
\begin{align*}
\big||V_F(\fq)|-q^{m-1}\big| \le&\,
\big||\mathrm{pcl}(V_F)(\fq)|-p_{m-1}\big|+\big||\mathrm{pcl}(V_F(\fq))^{\infty}|-p_{m-2}\big|\\
\le& \,(d-1)^k\,q^{\frac{2m-k-1}{2}}+
6(d+2)^{m+1}q^{\frac{2m-k-2}{2}}\\& +(d-1)^k\,q^{\frac{2m-k-3}{2}}+
6(d+2)^{m}q^{\frac{2m-k-4}{2}}\\
\le&\, q^{\frac{2m-k-1}{2}}(1+q^{-1})\big((d-1)^k+
6(d+2)^{m+1}q^{-\frac{1}{2}}\big).
\end{align*}
Summarizing, we have the following result.
\begin{theorem}\label{th: estimate |V_F| without Pi_m}
Suppose that $\mathrm{char}(\fq)$ does not divide
$m(m-1)\cdots(m-k)$. Let $F$ be as in \eqref{def: f_1}, satisfying
assumptions $({\sf A_1})$ and $({\sf A_2})$, and let
$V_F:=V(F)\subset\A^m$. Let $d:=\deg F$. We have
$$
\big||V_F(\fq)|-q^{m-1}\big|\leq
\left\{\begin{array}{c}
q^{m-\frac{3}{2}}(1+q^{-1})\big((d-1)(d-2)+14(d-1)^2d^2q^{-\frac{1}{2}}\big)
\textrm{ for }k=2,\\[1ex]
q^{m-2}(1+q^{-1})14(d-1)^3d^2\textrm{ for }k=3,
\\[1ex]
q^{m-\frac{k+1}{2}}(1+q^{-1})\big((d-1)^k+
6(d+2)^{m+1}q^{-\frac{1}{2}}\big)\textrm{ for }2\le k< m-1.
\end{array}\right.
$$
\end{theorem}

Finally, we consider the existence of points in $V_F(\fq)$, with $F$
as in \eqref{def: f_1}, satisfying assumptions $({\sf A_1})$ and
$({\sf A_2})$ and not depending on $\Pi_m$. If $q\ge 36(d-1)^2d^2$,
then we can slightly simplify the estimate for $k=2$ to
$$\big||V(\fq)|-q^{m-1}\big|\leq
q^{m-\frac{3}{2}}\mbox{$\frac{3}{2}$}\big((d-1)d+14(d-1)^2d^2q^{-\frac{1}{2}}\big)\leq
q^{m-\frac{3}{2}}\mbox{$\frac{3}{2}$}\big(1+\mbox{$\frac{14}{6}$}\big)(d-1)d.$$
Now we can proceed to establish a condition which implies that
$V(\fq)$ is nonempty. For $q\ge 36(d-1)^2d^2$, we have that
$|V(\fq)|>0$ if
$$q^{m-1}-\mbox{$\frac{3}{2}$}\big(1+\mbox{$\frac{14}{6}$}\big)(d-1)d\, q^{m-\frac{3}{2}}> 0.$$
As this condition holds, we have the following result.
\begin{corollary}
Suppose that $\mathrm{char}(\fq)$ does not divide $m(m-1)(m-2)$. Let
$F$ be as in \eqref{def: f_1}, satisfying assumptions $({\sf A_1})$
and $({\sf A_2})$ and not depending on $\Pi_m$. If $q\ge
36(d-1)^2d^2$, then $|V(\fq)|\ge 1$.
\end{corollary}
%
%
\section{The distribution of factorization patterns}
\label{section: distribution of fact patterns}
This section is devoted to the first application of the framework of
the previous sections, namely we obtain an estimate on the number
$|\mathcal{A}_{\bfs\lambda}|$ of elements on a family $\mathcal{A}$
of monic polynomials of $\fq[T]$ of degree $n$ having factorization
pattern $\bfs\lambda:=1^{\lambda_1}2^{\lambda_2}\cdots
n^{\lambda_n}$. Our estimate asserts that
$|\mathcal{A}_{\bfs\lambda}|=
\mathcal{T}(\bfs\lambda)\,q^{n-m}+\mathcal{O}(q^{n-m-1})$, where
$\mathcal{T}(\bfs\lambda)$ is the proportion of elements of the
symmetric group of $n$ elements with cycle pattern $\bfs\lambda$ and
$m$ is the codimension of $\mathcal{A}$. To this end, we reduce the
question to estimate the number of $\fq$--rational points with
pairwise--distinct coordinates of a certain family of complete
intersections defined over $\fq$. These complete intersections are
defined by symmetric polynomials which satisfy the hypotheses of
Theorem \ref{th: estimate |V| without Pi_(m-1),Pi_m} and Corollary
\ref{coro: nb point V_r distinct coordinates}.

Let $T$ be an indeterminate over $\cfq$. For a positive integer $n$,
let $\mathcal{P}:={\mathcal P}_n$ be the set of all monic
polynomials in $\fq[T]$ of degree $n$. Let
$\lambda_1,\dots,\lambda_n$ be nonnegative integers such that
$$\lambda_1+2\lambda_2+\cdots+n\lambda_n=n.$$
We denote by ${\mathcal P}_{\bfs \lambda}$ the set of $f\in
\mathcal{P}$ with factorization pattern $\bfs
\lambda:=1^{\lambda_1}2^{\lambda_2}\cdots n^{\lambda_n}$, namely the
elements $f\in \mathcal{P}$ having exactly $\lambda_i$ monic
irreducible factors over $\fq$ of degree $i$ (counted with
multiplicity) for $1\le i\le n$. Further, for any subset
$\mathcal{S}\subset\mathcal{P}$ we shall denote
$\mathcal{S}_{\bfs{\lambda}}:=\mathcal{S}\cap\mathcal{P}_{\bfs
\lambda}$.

In \cite{Cohen70},  S. Cohen showed that the proportion of elements
of $\mathcal{P}_{\bfs\lambda}$ in $\mathcal{P}$ is roughly the
proportion $\mathcal{T}({\bfs{\lambda}})$ of permutations with cycle
pattern $\bfs \lambda$ in the $n$th symmetric group $\mathbb{S}_n$,
where a permutation of $\mathbb{S}_n$ is said to have cycle pattern
$\bfs\lambda$ if it has exactly $\lambda_i$ cycles of length $i$ for
$1\le i\le n$. More precisely, Cohen proved that
$$
|\mathcal{P}_{\bfs\lambda}|=\mathcal{T}(\bfs\lambda)\,q^n+
\mathcal{O}(q^{n-1}),
$$
where the constant underlying the $\mathcal{O}$--notation depends
only on $\bfs \lambda$. Observe that the number of permutations in
$\mathbb{S}_n$ with cycle pattern $\bfs\lambda$ is $n!/w({\bfs
\lambda})$, where
$$w({\bfs \lambda}):=1^{\lambda_1}2^{\lambda_2}\dots n^{\lambda_n}
\lambda_1!\lambda_2!\dots\lambda_n!.$$
In particular, $\mathcal{T}({\bfs{\lambda}})={1}/{w({\bfs
\lambda})}$.

Further, in \cite{Cohen72} Cohen called
$\mathcal{S}\subset\mathcal{P}$ {\em uniformly distributed} if the
proportion $|\mathcal{S}_{\bfs\lambda}|/|\mathcal{S}|$ is roughly
$\mathcal{T}(\bfs\lambda)$ for every factorization pattern
$\bfs\lambda$. The main result of this paper (\cite[Theorem
3]{Cohen72}) provides a criterion for a {\em linear} family
$\mathcal{S}$ of polynomials of $\mathcal{P}$ to be uniformly
distributed in the sense above. For any such linear family
$\mathcal{S}$ of codimension $m \le n-2$, assuming that the
characteristic $p$ of $\fq$ is greater than $n$, it is shown that
\begin{equation}\label{eq: intro: Cohen72}
|\mathcal{S}_{\bfs\lambda}|=\mathcal{T}(\bfs\lambda)\,q^{n-m}+
\mathcal{O}(q^{n-m-\frac{1}{2}}).
\end{equation}
A difficulty with \cite[Theorem 3]{Cohen72} is that the hypotheses
for a linear family of $\mathcal{P}$ to be uniformly distributed
seem complicated and not easy to verify. In this direction,
\cite{BaBaRo15}, \cite{Ha16} and \cite{CeMaPe17} provide explicit
estimates on the number of elements with factorization pattern
$\bfs\lambda$ on certain linear families of $\fq[T]$, such as the
set of polynomials with some prescribed coefficients.

In \cite[Problem 2.2]{GaHoPa99} the authors ask for estimates on the
number of polynomials of a given degree with a given factorization
pattern lying in {\em nonlinear} families of polynomials with
coefficients parameterized by an affine variety defined over $\fq$.
\cite{MaPePr20} addresses this question, providing a general
criterion for a nonlinear family $\mathcal{A}\subset\fq[T]$ to be
uniform distributed in the sense of Cohen and explicit estimates on
the number of elements of $\mathcal{A}$ with a given factorization
pattern.

More precisely, let $m$ and $n$ be positive integers with $m<n$ and
$A_{n-1}\klk A_{0}$ indeterminates over $\cfq$. Denote
$\fq[\bfs{A}_1]:= \fq[A_{n-1},\ldots,A_1]$. Let $G_1\klk G_m\in
\fq[\bfs{A}_1]$ and let $W:=\{G_1=0,\ldots,G_m=0\}$ be the set of
common zeros in $\A^n:=\cfq{\!}^n$ of $G_1,\ldots,G_m$. Denoting by
$\fq[T]_n$ the set of monic polynomials of degree $n$ with
coefficients in $\fq$, we consider the following family of
polynomials:
\begin{equation}\label{eq: def non-linear family A}
\mathcal{A}:=\{T^n+a_{n-1}T^{n-1}+\cdots+a_0\in \fq[T]_n:
G_i(a_{n-1},\ldots,a_1)=0\,\,(1\leq i \leq m)\}.
\end{equation}

Consider the weight $\wt:\fq[\bfs{A}_1]\to\N_0$ defined by setting
$\wt(A_j):=n-j$ for $1\leq j \leq n-1$ and denote by $G_1^{\wt}\klk
G_m^{\wt}$ the components of highest weight of $G_1\klk G_m$. Let
$(\partial\bfs{G}/\partial \bfs{A}_1)$ be the Jacobian matrix of
$G_1\klk G_m$ with respect to $\bfs{A}_1$. In \cite{MaPePr20} it is
assumed that $G_1 \klk G_m$ satisfy the following conditions:
\begin{itemize}
\item[$({\sf H}_1)$] $G_1,\ldots,G_m$ form a regular sequence of $\fq[\bfs{A}_1]$.
\item[$({\sf H}_2)$]  $(\partial\bfs{G}/
\partial \bfs{A}_1)$ has full rank on every point of $W$.
\item[$({\sf H}_3)$]  $G_1^{\wt}\klk G_m^{\wt}$ satisfy
$({\sf H}_1)$ and $({\sf H}_2)$.
\end{itemize}

In what follows we identify the set $\cfq[T]_n$ of monic polynomials
of $\cfq[T]$ of degree $n$ with $\A^n$, by mapping each $f_{\bfs
a_0}:= T^n + a_{n-1}T^{n-1}+\cdots+a_0 \in \cfq[T]_n$ to $\bfs
a_0:=(a_{n-1},\ldots,a_0)\in\A^n$. For $\mathcal{B}\subset
\cfq[T]_n$, the set of elements of $\mathcal{B}$ which are not
square--free is called the discriminant locus
$\mathcal{D}(\mathcal{B})$ of $\mathcal{B}$ (see \cite{FrSm84} and
\cite{MaPePr14} for the study of discriminant loci). For $f_{\bfs
a_0}\in \mathcal{B}$, let $\mathrm{Disc}(f_{\bfs
a_0}):=\mathrm{Res}(f_{\bfs a_0},f'_{\bfs a_0})$ denote the
discriminant of $f_{\bfs a_0}$. Since $f_{\bfs a_0}$ has degree $n$,
by basic properties of resultants we have
$$\mathrm{Disc}(f_{\bfs a_0})= \mathrm{Disc}(F(\bfs A_0, T))|_{\bfs A_0=\bfs a_0} :=
\mathrm{Res}(F(\bfs A_0, T), F'(\bfs A_0, T), T)|_{\bfs A_0=\bfs
a_0},
$$
where the expression $\mathrm{Res}$ in the right--hand side denotes
resultant with respect to $T$. It follows that
$\mathcal{D}(\mathcal{B}):=\{\bfs a_0 \in \mathcal{B} :
\mathrm{Disc}(F(\bfs A_0, T))|_{\bfs A_0=\bfs a_0}= 0\}$. Further,
the first subdiscriminant locus $\mathcal{S}_1(\mathcal{B})$ of
$\mathcal{B}\subset\cfq[T]_n$ is the set of $\bfs a_0\in
\mathcal{D}(\mathcal{B})$ for which the first subdiscriminant
$\mathrm{sDisc}_1(f_{\bfs a_0}):=\mathrm{sRes}_1(f_{\bfs
a_0},f'_{\bfs a_0})$ vanishes, where $\mathrm{sRes}_1(f_{\bfs
a_0},f'_{\bfs a_0})$ denotes the first subresultant of $f_{\bfs
a_0}$ and $f'_{\bfs a_0}$. Since $f_{\bfs a_0}$ has degree $n$,
basic properties of subresultants imply
$$\mathrm{sDisc}_1(f_{\bfs a_0}
)= \mathrm{sDisc}_1(F(\bfs A_0, T))|_{\bfs A_0=\bfs a_0} :=
\mathrm{sRes}_1(F(\bfs A_0, T),F'(\bfs A_0, T), T))|_{\bfs A_0=\bfs
a_0},$$
where $\mathrm{sRes}$ in the right--hand side denotes first
subresultant with respect to $T$. We have
$\mathcal{S}_1(\mathcal{B}):=\{\bfs a_0\in\mathcal{D}(\mathcal{B}):
\mathrm{sDisc}_1(F(\bfs A_0, T))|_{\bfs A_0=\bfs a_0}=0\}$. The next
three conditions of \cite{MaPePr20} require that the discriminant
and the first subdiscriminant locus intersect well $W$:
\begin{enumerate}
\item[$({\sf H}_4)$] $\mathcal{D}(W)$ has codimension at least one in $W$.
\item [$({\sf H}_5)$] $ (A_0\cdot \mathcal{S}_1)(W):=\{\bfs a_0\in W: a_0=0\}\cup \mathcal{S}_1(W)$  has codimension at least one in
$\mathcal{D}(W)$.
\item[$({\sf H}_6)$] $\mathcal{D}(V(G_1^{\wt}\klk G_m^{\wt}))$ has codimension at least
one in $V(G_1^{\wt}\klk G_m^{\wt})\subset \cfq{\!}^r$.
\end{enumerate}

The main result of \cite{MaPePr20} asserts that any family
$\mathcal{A}$ satisfying hypotheses $({\sf H}_1)$--$({\sf H}_6)$ is
uniformly distributed in the sense of Cohen, and provides the
following explicit estimates on the number
$|\mathcal{A}_{\bfs\lambda}|$ of elements of $\mathcal{A}$ with
factorization pattern $\bfs\lambda$ for $m<n$:
\begin{equation}\label{eq: estimate fact patterns intro}
\big||\mathcal{A}_{\bfs\lambda}|
-\mathcal{T}(\bfs\lambda)\,q^{n-m}\big|\le
q^{n-m-1}\big(\mathcal{T}(\bfs\lambda) (D\delta\, q^{\frac{1}{2}}+14
D^2 \delta^2+n^2\delta)+n^2\delta\big),
\end{equation}
where $\delta:=\prod_{i=1}^m \wt(G_i)$ and
$D:=\sum_{i=1}^m(\wt(G_i)-1)$.

Applying Theorem \ref{th: estimate |V| without Pi_(m-1),Pi_m} we
shall be able to obtain a significantly simplified criterion for a
nonlinear family as in \eqref{eq: def non-linear family A} to
satisfy \eqref{eq: estimate fact patterns intro}.
%
%
\subsection{Factorization patterns and roots}
Let $\mathcal{A}\subset\mathcal{P}:=\mathcal{P}_n$ be the family of
\eqref{eq: def non-linear family A} and
$\bfs\lambda:=1^{\lambda_1}\cdots n^{\lambda_n}$ a factorization
pattern. Following the approach of \cite{CeMaPe17}, we shall show
that the condition that an element of $\mathcal{A}$ has
factorization pattern $\boldsymbol{\lambda}$ can be expressed in
terms of certain elementary symmetric polynomials.

Let $f$ be an element of $\mathcal{P}$ and $g\in \fq[T]$ a monic
irreducible factor of $f$ of degree $i$. Then $g$ is the minimal
polynomial of a root $\alpha$ of $f$ with $\fq(\alpha)=\fqi$. Denote
by $\mathbb G_i$ the Galois group $\mbox{Gal}(\fqi,\fq)$ of $\fqi$
over $\fq$. We may express $g$ in the following way:
$$g=\prod_{\sigma\in\mathbb G_i}(T-\sigma(\alpha)).$$
Hence, each irreducible factor $g$ of $f$ is uniquely determined by
a root $\alpha$ of $f$ (and its orbit under the action of the Galois
group of $\cfq$ over $\fq$), and this root belongs to a field
extension of $\fq$ of degree $\deg g$. Now, for
$f\in\mathcal{P}_{\bfs \lambda}$, there are $\lambda_1$ roots of $f$
in $\fq$, say $\alpha_1,\dots,\alpha_{\lambda_1}$ (counted with
multiplicity), which are associated with the irreducible factors of
$f$ in $\fq[T]$ of degree 1; we may choose $\lambda_2$ roots of $f$
in $\fqtwo\setminus\fq$ (counted with multiplicity), say
$\alpha_{\lambda_1+1},\dots, \alpha_{\lambda_1+\lambda_2}$, which
are associated with the $\lambda_2$ irreducible factors of $f$ of
degree 2, and so on. From now on we shall assume that a choice of
$\lambda_1\plp\lambda_n$ roots $\alpha_1\klk\alpha_{\lambda_1
\plp\lambda_n}$ of $f$ in $\cfq$ is made in such a way that each
monic irreducible factor of $f$ in $\fq[T]$ is associated with one
and only one of these roots.

Our aim is to express the factorization of $f$ into irreducible
factors in $\fq[T]$ in terms of the coordinates of the chosen
$\lambda_1\plp \lambda_n$ roots of $f$ with respect to certain bases
of the corresponding extensions $\fq\hookrightarrow\fqi$ as
$\fq$--vector spaces. To this end, we express the root associated
with each irreducible factor of $f$ of degree $i$ in a normal basis
$\Theta_i$ of the field extension $\fq\hookrightarrow \fqi$.

Let $\theta_i\in \fqi$ be a normal element and $\Theta_i$ the normal
basis of the extension $\fq\hookrightarrow\fqi$ generated by
$\theta_i$, i.e.,
$$\Theta_i=\left \{\theta_i,\cdots, \theta_i^{q^{i-1}}\right\}.$$
Observe that the Galois group $\mathbb G_i$ is cyclic and the
Frobenius map $\sigma_i:\fqi\to\fqi$, $\sigma_i(x):=x^q$ is a
generator of $\mathbb{G}_i$. Thus, the coordinates in the basis
$\Theta_i$ of all the elements in the orbit of a root
$\alpha_k\in\fqi$ of an irreducible factor of $f$ of degree $i$ are
the cyclic permutations of the coordinates of $\alpha_k$ in the
basis $\Theta_i$.

The vector that gathers the coordinates of all the roots
$\alpha_1\klk\alpha_{\lambda_1+\dots+\lambda_n}$ we chose to
represent the irreducible factors of $f$ in the normal bases
$\Theta_1\klk \Theta_n$ is an element of $\fq^n$, which is denoted
by ${\bfs x}:=(x_1,\dots,x_n)$. Set
\begin{equation}\label{eq: fact patterns: ell_ij}
\ell_{i,j}:=\sum_{k=1}^{i-1}k\lambda_k+(j-1)\,i
\end{equation}
for $1\le j \le \lambda_i$ and $1\le i \le n$. Observe that the
vector of coordinates of a root
$\alpha_{\lambda_1\plp\lambda_{i-1}+j}\in\fqi$ is the sub-array
$(x_{\ell_{i,j}+1},\dots,x_{\ell_{i,j}+i})$ of $\bfs x$. With these
notations, the $\lambda_i$ irreducible factors of $f$ of degree $i$
are the polynomials
\begin{equation}\label{eq: fact patterns: gij}g_{i,j}=\prod_{\sigma\in\mathbb G_i}
\Big(T-\big(x_{\ell_{i,j}+1}\sigma(\theta_i)+\dots+
x_{\ell_{i,j}+i}\sigma(\theta_i^{q^{i-1}})\big)\Big)
\end{equation}
for $1\le j \le \lambda_i$. In particular,
\begin{equation}\label{eq: fact patterns: f factored with g_ij}
f=\prod_{i=1}^n\prod_{j=1}^{\lambda_i}g_{i,j}.
\end{equation}

Let $X_1\klk X_n$ be indeterminates over $\cfq$, set $\bfs
X:=(X_1,\dots,X_n)$ and consider the polynomial $G\in
\fq[\bfs{X},T]$ defined as
\begin{equation}\label{eq: fact patterns: pol G}
G:=\prod_{i=1}^n\prod_{j=1}^{\lambda_i}G_{i,j},\quad
G_{i,j}:=\prod_{\sigma\in\mathbb G_i}
\Big(T-\big(X_{\ell_{i,j}+1}\sigma(\theta_i)+
\dots+X_{\ell_{i,j}+i}\sigma(\theta_i^{q^{i-1}})\big)\Big),
\end{equation}
where the $\ell_{i,j}$ are defined as in \eqref{eq: fact patterns:
    ell_ij}. Our previous arguments show that $f\in\mathcal{P}$ has
factorization pattern ${\bfs \lambda}$ if and only if there exists
$\bfs x\in\fq^n$ with $f=G({\bfs x},T)$.

Next we discuss how many elements $\bfs x\in\fq^n$ yield an
arbitrary polynomial $f=G(\bfs x,T)\in\mathcal{P}_{\bfs\lambda}$.
For $\alpha\in\fqi$, we have $\fq(\alpha)=\fqi$ if and only if its
orbit under the action of the Galois group $\G_i$ has exactly $i$
elements. In particular, if $\alpha$ is expressed by its coordinate
vector $\bfs x\in\fq^i$ in the normal basis $\Theta_i$, then the
coordinate vectors of the elements of the orbit of $\alpha$ form a
cycle of length $i$, because the Frobenius map $\sigma_i\in{\mathbb
    G}_i$ permutes cyclically the coordinates. As a consequence, there
is a bijection between cycles of length $i$ in $\fq^i$ and elements
$\alpha\in\fqi$ with $\fq(\alpha)=\fqi$.

To make this relation more precise, we introduce the notion of an
array of type $\bfs \lambda$. Let $\ell_{i,j}$ $(1\le i\le n,\ 1\le
j\le\lambda_i)$ be defined as in \eqref{eq: fact patterns: ell_ij}.
We say that ${\bfs x}=(x_1,\dots, x_n)\in\fq^n$ is of {\em type
$\bfs \lambda$} if and only if each sub-array $\bfs
x_{i,j}:=(x_{\ell_{i,j}+1},\dots,x_{\ell_{i,j}+i})$ is a cycle of
length $i$. The following result relates the quantity
$\mathcal{P}_{\boldsymbol{\lambda}}$ with the set of elements of
$\fq^n$ of type $\bfs \lambda$. The proof of this result is only
sketched here (see \cite[Lemma 2.2]{CeMaPe17} for a full proof).
\begin{lemma}
    \label{lemma: fact patterns: G(x,T) with fact pat lambda}

 For any
    ${\bfs x}=(x_1,\dots, x_n)\in \fq^n$, the polynomial $f:=G({\bfs x},T)$
    has factorization pattern $\bfs \lambda$ if and only if ${\bfs x}$
    is of type $\bfs \lambda$. Furthermore, for each square--free
    polynomial $f\in \mathcal{P}_{\bfs \lambda}$ there are $w({\bfs
        \lambda}):=\prod_{i=1}^n i^{\lambda_i}\lambda_i!$ different ${\bfs
        x}\in \fq^n $ with $f=G({\bfs x},T)$.
\end{lemma}

\begin{proof}[Sketch of proof]
    Let $\Theta_1,\dots,\Theta_n$ be the normal bases introduced before.
    Each $\bfs x\in \fq^n$ is associated with a unique finite sequence
    of elements $\alpha_k$ $(1\le k\le \lambda_1+\dots+\lambda_n)$ as
    follows: each $\alpha_{\lambda_1\plp\lambda_{i-1}+j}$ with $1\le
    j\le\lambda_i$ is the element of $\fqi$ whose coordinate vector in
    the basis $\Theta_i$ is the sub-array $(x_{\ell_{i,j}+1},
    \dots,x_{\ell_{i,j}+i})$ of $\bfs x$.
    Suppose that $G({\bfs x},T)$ has factorization pattern $\bfs
    \lambda$ for a given $\bfs x\in\fq^n$. Fix $(i,j)$ with $1\le i\le
    n$ and $1\le j\le\lambda_i$. Then $G({\bfs x},T)$ is factored as in
    \eqref{eq: fact patterns: gij}--\eqref{eq: fact patterns: f factored
        with g_ij}, where each $g_{i,j}\in\fq[T]$ is irreducible, and hence
    $\fq(\alpha_{\lambda_1\plp\lambda_{i-1}+j})=\fqi$. We conclude that
    the sub-array $(x_{\ell_{i,j}+1}, \dots,x_{\ell_{i,j}+i})$ defining
    $\alpha_{\lambda_1\plp\lambda_{i-1}+j}$  is a cycle of length $i$.
    This proves that $\bfs x$ is of type $\bfs\lambda$.
%

    Furthermore, for $\bfs x\in\fq^n$ of type $\bfs\lambda$, the
    polynomial $f:=G(\bfs x,T)\in\mathcal{P}_{\bfs\lambda}$ is
    square--free if and only if all the roots
    $\alpha_{\lambda_1\plp\lambda_{i-1}+j}$ with $1\le j\le \lambda_i$
    are pairwise--distinct, non--conjugated elements of $\fqi$. This
    implies that no cyclic permutation of a sub-array
    $(x_{\ell_{i,j}+1}, \dots,x_{\ell_{i,j}+i})$ with $1\le
    j\le\lambda_i$ agrees with another cyclic permutation of another
    sub-array $(x_{\ell_{i,j'}+1}, \dots,x_{\ell_{i,j'}+i})$. As cyclic
    permutations of any of these sub-arrays and permutations of these
    sub-arrays yield elements of $\fq^n$ associated with the same
    polynomial $f$, there are $w({\bfs
        \lambda}):=\prod_{i=1}^n i^{\lambda_i}\lambda_i!$ different elements
    $\bfs x\in\fq^n$ with $f=G(\bfs x,T)$.
\end{proof}

Consider the polynomial $G$ defined in \eqref{eq: fact patterns: pol
G} as an element of $\fq[\bfs X][T]$. We shall express the
coefficients of $G$ by means of the vector of linear forms $\bfs
Y:=(Y_1\klk Y_n)$, with $Y_i\in\cfq[\bfs X]$ for $1\le i\le n$,
defined in the following way:
\begin{equation}\label{eq: fact patterns: def linear forms Y}
(Y_{\ell_{i,j}+1},\dots,Y_{\ell_{i,j}+i})^{t}:=A_{i}\cdot
(X_{\ell_{i,j}+1},\dots, X_{\ell_{i,j}+i})^{t} \quad(1\le j\le
\lambda_i,\ 1\le i\le n),
\end{equation}
where $A_i\in\fqi^{i\times i}$ is the matrix
$$A_i:=\left(\sigma(\theta_i^{q^{h}})\right)_{\sigma\in {\mathbb G}_i,\, 0\le h\le i-1}.$$
According to (\ref{eq: fact patterns: pol G}), we may express the
polynomial $G$ as
$$G=\prod_{i=1}^n\prod_{j=1}^{\lambda_i}\prod_{k=1}^i(T-Y_{\ell_{i,j}+k})=
\prod_{k=1}^n(T-Y_k)=T^n+\sum_{k=1}^n(-1)^k\,(\Pi_k(\bfs Y))\,
T^{n-k},$$
where $\Pi_1(\bfs Y)\klk \Pi_n(\bfs Y)$ are the elementary symmetric
polynomials of $\fq[\bfs Y]$. By the expression of $G$ in \eqref{eq:
    fact patterns: pol G} we deduce that $G$ belongs to $\fq[{\bfs
    X},T]$, which in particular implies that $\Pi_k(\bfs Y)$ belongs to
$\fq[{\bfs X}]$ for $1\le k\le n$. Combining these arguments with
Lemma \ref{lemma: fact patterns: G(x,T) with fact pat lambda} we
obtain the following result.
\begin{lemma}
    \label{lemma: fact patterns: sym pols and pattern lambda} A
    polynomial $f:=T^n+a_{n-1}T^{n-1}\plp a_0\in\mathcal{P}$ has
    factorization pattern $\bfs \lambda$ if and only if there exists
    $\bfs{x}\in\fq^n$ of type $\bfs \lambda$ such that
    \begin{equation}\label{eq: fact patterns: sym pols and pattern lambda}
    a_j= (-1)^{n-j}\,\Pi_{n-j}(\bfs Y(\bfs x)) \quad(0\le j\le n-1).
    \end{equation}
    In particular, for $f$ square--free, there are $w(\bfs \lambda)$
    elements $\bfs x$ for which (\ref{eq: fact patterns: sym pols and
        pattern lambda}) holds.
\end{lemma}

As a consequence, we may express the condition that an element of
the family $\mathcal{A}$ of \eqref{eq: def non-linear family A} has
factorization pattern $\bfs \lambda$ in terms of the elementary
symmetric polynomials $\Pi_1\klk\Pi_{n-1}$ of $\fq[\bfs Y]$.
\begin{corollary}\label{coro: fact patterns: systems pattern lambda}
A polynomial $f:=T^n+a_{n-1}T^{n-1}\plp a_0\in\mathcal{A}$ has
factorization pattern $\bfs \lambda$ if and only if there exists
$\bfs{x}\in\fq^n$ of type $\bfs \lambda$ such that \eqref{eq: fact
patterns: sym pols and pattern lambda} and
\begin{equation}\label{eq: fact patterns: systems pattern lambda}
G_j\big(-\Pi_1(\bfs Y(\bfs x))\klk (-1)^{n-1}\, \Pi_{n-1}(\bfs
Y(\bfs x)) \big)=0 \quad(1\le j\le m)
\end{equation}
hold. In particular, if $f:=G(\bfs x,T)\in\mathcal{A}_{\bfs\lambda}$
is square--free, then there are $w(\bfs \lambda)$ elements $\bfs x$
for which \eqref{eq: fact patterns: systems pattern lambda} holds.
\end{corollary}
%
%
\subsection{The number of polynomials in $\mathcal{A}_{\bfs\lambda}$}
%
Given a factorization pattern $\bfs{\lambda}:=1^{\lambda_1} \cdots
n^{\lambda_n}$, consider the set $\mathcal{A}_{\bfs\lambda}$ of
elements of the family $\mathcal{A}\subset \mathcal{P}$ of
\eqref{eq: def non-linear family A} having factorization pattern
$\bfs{\lambda}$. In this section we estimate the number of elements
of
$\mathcal{A}_{\bfs\lambda}$. For this purpose, 
%
in Corollary \ref{coro: fact patterns: systems pattern lambda} we
associate to $\mathcal{A}_{\bfs \lambda}$ the following polynomials
of $\fq[\bfs X]$:
\begin{equation}\label{eq: geometry: def R_j}
F_j:=F_j^{\bfs\lambda}:= G_j\big(-\Pi_1(\bfs Y(\bfs x))\klk
(-1)^{n-1}\, \Pi_{n-1}(\bfs Y(\bfs x))\big)\quad (1\le j\le m).
\end{equation}
Up to the linear change of coordinates defined by $\bfs Y:=(-Y_1\klk
(-1)^nY_n)$, where $\bfs Y$ is the vector of linear forms of
$\cfq[\bfs X]$ defined  in \eqref{eq: fact patterns: def linear
forms Y}, we may express each $F_j$ as a polynomial in the first
$n-1$ elementary symmetric polynomials $\Pi_1,\ldots,\Pi_{n-1}$ of
$\fq[\bfs Y]$. More precisely, let $Z_1\klk Z_{n-1}$ be new
indeterminates over $\cfq$. Then we may write
\begin{equation}\label{eq: polynomials F_j patterns}
F_j=G_j(\Pi_1,\ldots,\Pi_{n-1})\quad (1\le j\le m).
\end{equation}
Applying Corollary \ref{coro: nb point V_r distinct coordinates} we
obtain the following result.
\begin{theorem}\label{th: inequality of V(fq) I}
Suppose that $\mathrm{char}(\fq)$ does not divide $n(n-1)(n-2)$.
Suppose further that $G_1,\ldots,G_m$ satisfy assumptions $({\sf
A}_1)$--$({\sf A}_2)$ of Section \ref{subsec: complete
intersections} and do not depend on $Z_{n-2},Z_{n-1},Z_n$, namely
$G_1,\ldots,G_m\in\fq[Z_1,\ldots,Z_{n-3}]$. Let
$V:=V(F_1,\ldots,F_m)\subset \A^n$,
\begin{equation}\label{eq: fact patterns: definition V=}
V^{=}:=\mathop{\bigcup_{1\le i\le n}}_{ 1\leq j_1 <j_2\leq
    \lambda _i,\, \,  1\leq k_1 <k_2 \leq i}
V\cap\{Y_{\ell_{i,j_1}+k_1}=Y_{\ell_{i,j_2}+k_2}\},
\end{equation}
and $V^{\neq}:=V\setminus V^{=}$, where $Y_{\ell_{i,j}+k}$ are the
linear forms of \eqref{eq: fact patterns: def linear forms Y}. Then
\begin{align*}
\big||V^{\neq}(\fq)|-q^{n-m}\big|&\leq 14 D^3
\delta^2(1+q^{-1})q^{n-m-1}+n^2\delta q^{n-m-1},
\end{align*}
where $\delta:=\prod_{i=1}^m \wt(G_i)$ and
$D:=\sum_{i=1}^m(\wt(G_i)-1)$.
\end{theorem}
\begin{proof}
The statement follows immediately applying Corollary \ref{coro: nb
point V_r distinct coordinates} with $k=3$, taking into account that
$\wt(G_i)=\deg F_i$ for $1\le i\le m$.
\end{proof}

Corollary \ref{coro: fact patterns: systems pattern lambda} relates
the number $|V(\fq)|$ to the quantity
$|\mathcal{A}_{\bfs{\lambda}}|$. More precisely, let $\bfs x:=(\bfs
x_{i,j}:1\le i\le n,1\le j\le \lambda_i)\in\fq^n$ be an
$\fq$--rational zero of $F_1\klk F_m$ of type $\bfs\lambda$. Then
$\bfs x$ is associated with an element
$f\in\mathcal{A}_{\bfs\lambda}$ having $Y_{\ell_{i,j}+k}(\bfs
x_{i,j})$ as an $\fqi$--root for $1\le i\le n$, $1\le j\le\lambda_i$
and $1\le k\le i$, where $Y_{\ell_{i,j}+k}$ is the linear form
defined as in \eqref{eq: fact patterns: def linear forms Y}.

Let $\mathcal{A}_{\bfs\lambda}^{sq}:=\{f\in
\mathcal{A}_{\bfs\lambda}: f \mbox{ is square--free}\}$ and
$\mathcal{A}_{\bfs\lambda}^{nsq}:=\mathcal{A}_{\bfs
\lambda}\setminus \mathcal{A}_{\bfs\lambda}^{sq}$. Corollary
\ref{coro: fact patterns: systems pattern lambda} further asserts
that any element $f\in \mathcal{A}_{\bfs\lambda}^{sq}$ is associated
with $w(\bfs\lambda):=\prod_{i=1}^n i^{\lambda_i}\lambda_i!$ common
$\fq$--rational zeros of $F_1\klk F_m$ of type $\bfs\lambda$.
Observe that $\bfs x\in\fq^n$ is of type $\bfs\lambda$ if and only
if $Y_{\ell_{i,j}+k_1}(\bfs x) \neq Y_{\ell_{i,j}+k_2}(\bfs x)$ for
$1\leq i\leq n$, $1\leq j\leq \lambda_i$ and $1\leq k_1 <k_2 \leq
i$. Furthermore, an $\bfs x\in\fq^n$ of type $\bfs\lambda$ is
associated with $f\in\mathcal{A}_{\bfs\lambda}^{sq}$ if and only if
$Y_{\ell_{i,j_1}+k_1}(\bfs x) \neq Y_{\ell_{i,j_2}+k_2}(\bfs x)$ for
$1\leq i\leq n$, $1\leq j_1<j_2\leq \lambda_i$ and $1\leq k_1 <k_2
\leq i$. As a consequence, we see that
$|\mathcal{A}_{\bfs\lambda}^{sq}| =\mathcal{T}(\bfs\lambda)
\big|V^{\neq}(\fq)\big|$, which implies
$$
\big||\mathcal{A}_{\bfs\lambda}^{sq}|
-\mathcal{T}(\bfs\lambda)\,q^{n-m}\big| =
\mathcal{T}(\bfs\lambda)\,\big||V^{\neq}(\fq)|-q^{n-m}\big|.
$$
From Theorem \ref{th: inequality of V(fq) I} we deduce that
\begin{align*}
\big||\mathcal{A}_{\bfs\lambda}^{sq}|
-\mathcal{T}(\bfs\lambda)\,q^{n-m}\big| &\le
\,\mathcal{T}(\bfs\lambda)\big(14\,q^{n-m-1}(1+q^{-1})
D^3\delta^2+ n^2\delta\, q^{n-m-1}\big)\\
&\le q^{n-m-1} \mathcal{T}(\bfs\lambda)\,\big(17\, D^3\delta^2+
n^2\delta\big).
\end{align*}

Now we are able to estimate $|\mathcal{A}_{\bfs \lambda}|$. We have
\begin{align}
\big||\mathcal{A}_{\bfs\lambda}|
-\mathcal{T}(\bfs\lambda)\,q^{n-m}\big|&=
\big||\mathcal{A}_{\bfs\lambda}^{sq}|+
|\mathcal{A}_{\bfs\lambda}^{nsq}|-\mathcal{T}(\bfs\lambda)q^{n-m}\big|\nonumber\\
&\le q^{n-m-1} \mathcal{T}(\bfs\lambda)\,\big(17\, D^3\delta^2+
n^2\delta\big)+ |\mathcal{A}_{\bfs\lambda}^{nsq}|. \label{eq:
estimates: estimate A_lambda aux}
\end{align}
It remains to obtain an upper bound for
$|\mathcal{A}_{\bfs\lambda}^{nsq}|$. Let
$f\in\mathcal{A}_{\bfs\lambda}^{nsq}$ and let $\bfs x\in\fq^n$ be of
type $\bfs\lambda$ such that \eqref{eq: fact patterns: systems
pattern lambda} holds. Then there exists $1\leq i\leq n$, $1\leq
j_1<j_2\leq \lambda_i$ and $1\leq k_1 <k_2 \leq i$ such that
$Y_{\ell_{i,j_1}+k_1}(\bfs x)= Y_{\ell_{i,j_2}+k_2}(\bfs x)$ holds.
In other words, $\bfs x$ belongs to the variety $V^=$ of \eqref{eq:
fact patterns: definition V=}. According to Proposition
\ref{theorem: nb point V_r distinct coordinates}, each intersection
$V\cap\{Y_{\ell_{i,j_1}+k_1}=Y_{\ell_{i,j_2}+k_2}\}$ is of pure
dimension $n-m-1$. By the B\'ezout inequality \eqref{eq: Bezout} we
conclude that $\deg V^=\le n^2\deg V=n(n-1)\delta$. Then \eqref{eq:
upper bound -- affine gral} implies $|V^=(\fq)|\le n(n-1)q^{n-m-1}$.
Finally, as each $f\in\mathcal{A}_{\bfs\lambda}$ is associated with
at most $w(\bfs\lambda)$ elements $\bfs x\in V^=(\fq)$, we conclude
that
\begin{equation}\label{eq: estimates: upper bound discr locus}
|\mathcal{A}_{\bfs\lambda}^{nsq}|\le |\mathcal{A}^{nsq}|\le
n(n-1)w(\bfs\lambda) \,q^{n-m-1}.
\end{equation}
Hence, combining \eqref{eq: estimates: estimate A_lambda aux} and
\eqref{eq: estimates: upper bound discr locus} we conclude that
\begin{align*}
\big||\mathcal{A}_{\bfs\lambda}|
-\mathcal{T}(\bfs\lambda)\,q^{n-m}\big|&\le
\mathcal{T}(\bfs\lambda)\,\big||V(\fq)|-q^{n-m}\big|+ n^2 q^{n-m-1}\\
&\le q^{n-m-1}\big(17\,\mathcal{T}(\bfs\lambda) D^3\delta^2+
2\,\mathcal{T}(\bfs\lambda)\,n^2\delta\big).
\end{align*}
In other words, we have the following result.
\begin{theorem}\label{theorem: estimate fact patterns}
Suppose that $\mathrm{char}(\fq)$ does not divide $n(n-1)(n-2)$.
Suppose further that $G_1,\ldots,G_m$ satisfy assumptions $({\sf
A}_1)$--$({\sf A}_2)$ of Section \ref{subsec: complete
intersections} and do not depend on $Z_{n-2},Z_{n-1},Z_n$, namely
$G_1,\ldots,G_m\in\fq[Z_1,\ldots,Z_{n-3}]$. We have
\begin{align*}
\big||\mathcal{A}_{\bfs\lambda}^{sq}|-\mathcal{T}(\bfs\lambda)\,q^{n-m}\big|&\le
q^{n-m-1}\mathcal{T}(\bfs\lambda)\,\big(17 D^3\delta^2+ n^2\delta\big),\\
\big||\mathcal{A}_{\bfs\lambda}|-\mathcal{T}(\bfs\lambda)\,q^{n-m}\big|&\le
q^{n-m-1}\mathcal{T}(\bfs\lambda)\big(17\,D^3\delta^2+
2\,n^2\delta\big),
\end{align*}
where $\delta:=\prod_{i=1}^m \wt(G_i)$ and
$D:=\sum_{i=1}^m(\wt(G_i)-1)$.
\end{theorem}

This result strengthens \eqref{eq: intro: Cohen72} and \eqref{eq:
estimate fact patterns intro} in several aspects. A first one is
that the hypotheses on the nonlinear family $\mathcal{A}$ in the
statement of Theorem \ref{theorem: estimate fact patterns} are much
easier to verify than those of \eqref{eq: intro: Cohen72} and
\eqref{eq: estimate fact patterns intro}. A second aspect is that
our error term is of order $\mathcal{O}(q^{n-m-1})$, thus providing
an asymptotic estimate much more precise than those of \eqref{eq:
intro: Cohen72} and \eqref{eq: estimate fact patterns intro}, for
which we provide explicit expressions for the constants underlying
the $\mathcal{O}$--notation with a good behavior. Finally, our
result is valid on an assumption on the characteristic $p$ of $\fq$
which is much wider than that of \eqref{eq: intro: Cohen72}.
%
%
\subsubsection{Factorization patterns in polynomials with prescribed coefficients}
A classical case which has received particular attention is that of
the elements of $\mathcal{P}$ having certain coefficients
prescribed. Therefore, we briefly state what we obtain in this case.
For this purpose, given $0< i_1<i_2<\cdots< i_m\le n$ and $\bfs
\alpha:=(\alpha_{i_1}\klk \alpha_{i_m})\in\fq^m$, set
$\mathcal{I}:=\{i_1\klk i_m\}$ and consider the set
$\mathcal{A}^\mathcal{I}:= \mathcal{A}(\mathcal{I},\bfs \alpha)$
defined in the following way:
\begin{equation}\label{eq: estimates: family A^m}
\mathcal{A}^\mathcal{I}:=\left\{T^n+a_1T^{n-1}\plp a_n\in\fq[T]:
a_{i_j}=\alpha_{i_j}\ (1\le j\le m)\right\}.
\end{equation}
Denote by $\mathcal{A}^{\mathcal{I},sq}$ the set of
$f\in\mathcal{A}^\mathcal{I}$ which are square--free.

For a given factorization pattern $\bfs\lambda$, let $G\in\fq[\bfs
X,T]$ be the polynomial of \eqref{eq: fact patterns: pol G}.
According to Lemma \ref{lemma: fact patterns: sym pols and pattern
lambda}, an element $f\in \mathcal{A}^\mathcal{I}$ has factorization
pattern $\bfs\lambda$ if and only if there exists $\bfs x$ of type
$\bfs\lambda$ such that
\begin{equation}\label{eq: fact pattern for prescribed coeff}
(-1)^{i_j}\Pi_{i_j}(\bfs Y(\bfs x))=\alpha_{i_j}\quad(1\le j\le m).
\end{equation}
From Theorem \ref{theorem: estimate fact patterns} we deduce the
following result.
\begin{corollary}\label{coro: estimates: estimates for presc coeff}
Suppose that $\mathrm{char}(\fq)$ does not divide $n(n-1)(n-2)$. For
$q > n$ and $i_m\le n-3$, we have
    \begin{align*}
    \big||\mathcal{A}^{\mathcal{I},sq}_{\bfs
        \lambda}|-\mathcal{T}(\bfs\lambda)\,q^{n-m}\big|&\le
    q^{n-m-1}\mathcal{T}(\bfs\lambda)\,\big(17D_{\mathcal{I}}^3\,
    \delta_{\mathcal{I}}^2 + n^2\delta_{\mathcal{I}}\big),\\
    \big||\mathcal{A}^\mathcal{I}_{\bfs
        \lambda}|-\mathcal{T}(\bfs\lambda)\,q^{n-m}\big|&\le
    q^{n-m-1}\mathcal{T}(\bfs\lambda)\big(17\,D_{\mathcal{I}}^3\,
    \delta_{\mathcal{I}}^2 +2\,n^2\delta_{\mathcal{I}}\big),
    \end{align*}
where $\delta_{\mathcal{I}}:=i_1\cdots i_m$ and
$D_{\mathcal{I}}:=\sum_{j=1}^m(i_j-1)$.
\end{corollary}

\begin{proof}
Let $G_j:=Z_{i_j}-\alpha_{i_j}$ for $1\le j\le m$. Then the
solutions of the set of equations
$$G_j\big(-\Pi_1(\bfs
Y(\bfs X)) \klk (-1)^{n-1}\, \Pi_{n-1}(\bfs Y(\bfs X))\big)=0\quad
(1\le j\le m)$$
defines precisely \eqref{eq: fact pattern for prescribed coeff}. The
condition $i_m\le n-3$ implies that $G_1,\ldots,G_m$ do not depend
on $Z_{n-2},Z_{n-1},Z_n$. Further, as $G_1,\ldots,G_m$ are elements
of degree $1$ whose homogeneous components of degree 1 are linearly
independent, the Jacobian matrix $(\partial \bfs{G}/\partial
\bfs{Z})(\boldsymbol{z})$ of $\boldsymbol{G}:=(G_1\klk G_m)$ with
respect to $\bfs{Z} :=(Z_1,\ldots,Z_{n-3})$ has full rank $m$ for
every $\boldsymbol{z}\in V(G_1,\ldots,G_m)$, namely $G_1,\ldots,G_m$
satisfy assumption $({\sf A}_1)$ of Section \ref{subsec: singular
locus of V}. Finally, the components
$G_1^{\mathrm{wt}}=Z_{i_1},\ldots,G_m^{\mathrm{wt}}=Z_{i_m}$ of
highest weight of $G_1,\ldots,G_m$ are linearly independent
homogeneous polynomials of degree $1$. Hence, $G_1,\ldots,G_m$
satisfy assumption ($\sf A_2$). As a consequence, taking into
account that $\wt(G_j)=i_j$ for $1 \leq j\leq m$, applying Theorem
\ref{theorem: estimate fact patterns} we easily deduce the
corollary.
\end{proof}
%
%
\section{Deep holes in Reed-Solomon codes}
\label{section: deep holes}
Our second application refers to a problem in Reed-Solomon codes
concerning the existence of deep holes. Given  a subset  ${\sf
D}:=\{x_1,\ldots,x_{n}\}\subset \fq$ and a positive integer  $k\leq
n$, the {\em Reed--Solomon code  of length $n$ and dimension $k$}
over $\fq$ is the following subset of $\fq^n$:
$$
{C}({\sf D},k):  = \{(f(x_{1}),\ldots,f(x_{n})) : f\in \fq[T],\,
\deg f \leq  k-1 \}.$$
The set ${\sf D}$ is called the {\em evaluation set} and the
elements of ${C}({\sf D},k)$ are called {\em codewords} of the code.
When ${\sf D}= \fq^*$, we say that $C({\sf D}, k)$ is the {\em
standard Reed--Solomon code of dimension $k$}.

Let $C: = {C}({\sf D},k)$. For $\bfs{w}\in \fq^{n}$,  we define the
{\em distance of ${\bfs w}$ to the code} $C$ as
$${\sf d}({\bfs w},C):=\displaystyle\min_{\mathbf{c}\in C}{\sf d}
({\bfs w},\bfs{c}),$$
where ${\sf d}$ is the Hamming distance of $\fq^n$. The {\em minimum
distance} ${\sf d}(C)$ of $C$ is the shortest distance between any
two distinct codewords.  The {\em covering radius} of $C$ is defined
as
$$
\rho:=\max_{\bfs{y}\in\mathbb{F}_{\hskip-0.7mmq}^{n}}{\sf
d}(\bfs{y},C).
$$
It is well--known that ${\sf d}(C)= n-k+1$ and $\rho= n-k$ hold.
Finally, we say that a ``word'' ${\bfs w}\in \fq^{n}$ is a {\em deep
hole} if ${\sf d}({{\bfs w}},C)=\rho$ holds.

A decoding algorithm for the code $C$ receives a word ${\bfs
w}\in\fq^{n}$ and outputs the message, namely the codeword that is
most likely to be received as ${\bfs w}$ after transmission, roughly
speaking. One of the most important algorithmic problems in this
setting is that of the \emph{maximum--likelihood decoding}, which
consists in computing the closest codeword to any given word ${\bfs
w}\in\fq^{n}$. Suppose that we receive a word ${\bfs
w}:=(w_1,\ldots,w_n)\in \fq^{n}$. Solving the maximum--likelihood
decoding for ${\bfs w}$ amounts at finding a polynomial $f \in
\fq[T]$ of degree at most $k-1$ satisfying the largest number of
conditions $f(x_i)=w_i$ with $1\leq i\leq n$. By interpolation,
there exists a unique polynomial $f_{\bfs w}$ of degree at most
$n-1$ such that $f_{\bfs w}(x_i)=w_i$ holds for $1\leq i\leq n$. In
this case, we say that the word ${\bfs w}$ was generated by the
polynomial $f_{\bfs w}$. If $\deg f_{\bfs w}\leq k-1$, then ${\bfs
w}$ is a codeword.

In this section our main concern will be the existence of deep holes
of the given standard Reed--Solomon code $C$. According to our
previous remarks, a deep hole can only arise as the word generated
by a polynomial $f \in \fq[T]$ with $k\le \deg f\le n-1$. In this
sense, we have the following result.
\begin{lemma}[{\cite[Corollary 1]{ChMu07}}]\label{prop:poli de grado k genera deep holes}
Polynomials of degree $k$ generate deep holes.
\end{lemma}

Next we reduce further the set of polynomials $f\in\fq[T]$ which are
candidates for generating deep holes. Suppose that we receive a word
${\bfs w} \in \fq^n$, which is generated by a polynomial $f_{\bfs
w}\in\fq[T]$ of degree greater than $k$. We want to know whether
${\bfs w}$ is a deep hole. We can decompose $f_{\bfs w}$ as a sum
$f_{\bfs w}=g+h$, where $g$ consists of the sum of the monomials of
$f_{\bfs w}$ of degree at least $k$ and $h$ consists of those of
degree at most $k-1$.
\begin{remark}[{\cite[Remark 1.2]{CaMaPr12}}] \label{remark:observacion sobre las distancias}
Let ${\bfs w}_g$ and ${\bfs w}_h$ be the words generated by $g$ and
$h$ respectively. Then ${\bfs w}$ is a deep hole if and only if
${\bfs w}_g$ is a deep hole.
\end{remark}

From Remark \ref{remark:observacion sobre las distancias} it follows
that any deep hole of the standard Reed--Solomon code $C$ is
obtained as the word ${\bfs w}_f$ generated by a polynomial $f \in
\fq[T]$ of the form
\begin{equation}\label{eq:polinomio f}
f:=T^{k+d}+f_{d-1}T^{k+d-1}+\cdots+f_{0}T^{k},
\end{equation}
where $d$ is a nonnegative integer with $k+d<q-1$. In view of Lemma
\ref{prop:poli de grado k genera deep holes}, we shall only discuss
the case $d\ge 1$.

From now on we shall consider the standard Reed--Solomon code
$C:=C(\fq^*,k)$. In \cite{ChMu07} it is conjectured that the
reciprocal of Lemma \ref{prop:poli de grado k genera deep holes}
also holds, namely a word ${\bfs w}$ is a deep hole of $C$ if and
only if it is generated by a polynomial $f\in\fq[T]$ of degree $k$.
Furthermore, the existence of deep holes of $C$ is related to the
non-existence of $\fq$--rational points, namely points whose
coordinates belong to $\fq$, of a certain family of hypersurfaces,
in the way that we now explain. Fix $f\in\fq[T]$ as in
\eqref{eq:polinomio f}  and let  ${\bfs w}_f$ be the generated word.
Let $X_1,\ldots, X_{k+1}$ be indeterminates over $\cfq$ and let
$Q\in \fq[X_1,\ldots,X_{k+1}][T]$ be the polynomial
$$Q=(T-X_1)\cdots(T-X_{k+1}).$$
There exists $R_f\in\fq[X_1,\ldots,X_{k+1}][T]$ with $\deg R_f\leq
k$ such that
\begin{equation}\label{eq:congruencia}
f\equiv R_f\mod{Q}.
\end{equation}
Assume that $R_f$ has degree $k$ and denote by $H_f\in
\fq[X_1,\ldots, X_{k+1}]$ its leading coefficient. Suppose that
there exists a vector $\bfs{x}\in (\fq^*)^{k+1}$  with
pairwise--distinct coordinates such that $H_f(\bfs{x})=0$ holds.
This implies that $r:=R_f(\bfs{x},T)$ has degree at most $k-1$ and
hence generates a codeword ${\bfs w}_r$. By \eqref{eq:congruencia}
we deduce that
$$d({\bfs w}_f,C) \leq d({\bfs w}_f,{\bfs w}_r)\leq q-k-2$$
holds, and thus ${\bfs w}_f$ is  not a deep hole.

As a consequence, the given polynomial $f$ does not generate a deep
hole of $C$ if and only if there exists a zero
$\bfs{x}:=(x_1,\ldots, x_{k+1})\in\fq^{k+1}$ of $H_f$ with nonzero,
pairwise--distinct coordinates, namely a solution
$\bfs{x}\in\fq^{k+1}$ of the following system of equalities and
non-equalities:
\begin{equation}\label{eq:sistema de las soluciones no deseadas}
H_f (X_1,\ldots, X_{k+1})=0,\ \prod_{1\le i<j\le
k+1}(X_i-X_j)\not=0,\ \prod_{1\le i\le k+1}X_i\not=0.
\end{equation}
%
%
\subsection{$H_f$ in terms of the elementary symmetric polynomials}
\label{subsec: planteo del problema}
Fix positive integers $d$ and $k$ such that $d < k$, and consider
the first $d$ elementary symmetric polynomials $\Pi_1 ,\ldots,
\Pi_d$ of $\fq[X_1,\ldots, X_{k+1}]$. For convenience of notation,
we shall denote  $\Pi_0:= 1$. As asserted above, the word
$\bfs{w}_f$ generated by a given polynomial $f$ is not a deep hole
of the standard Reed--Solomon code of dimension $k$ over $\fq$ if
the polynomial $H_f$ of \eqref{eq:polinomio f} has an
$\fq$--rational zero with nonzero, pairwise--distinct coordinates.

Next we show how the polynomials $H_f$ can be expressed in terms of
the elementary symmetric polynomials $\Pi_1 ,\ldots, \Pi_d$. To do
this, we first express the polynomial $H_d$ associated to the
monomial $T^{k+d}$ in terms of $\Pi_1 ,\ldots, \Pi_d$. From this
expression we readily obtain an expression for the polynomial $H_f$
associated to an arbitrary polynomial $f$ as in (\ref{eq:polinomio
f}) of degree $k+d$.
\begin{proposition}[{\cite[Proposition 2.2]{CaMaPr12}}]\label{prop: formula hipersuperficie H_0d}
Let $H_d\in\fq[X_1,\ldots, X_{k+1}]$ be the polynomial associated to
the monomial $T^{k+d}$. Then the following identity holds:
\begin{equation}\label{eq:formula explicita para los HOd}
H_d = \sum_{i_1+2\,i_2+\cdots+ d\,i_d=d} (-1)^{\Delta(i_1,\ldots,
i_d)} \dfrac{(i_1+\cdots+ i_d)!}{i_1!\cdots i_d!}\Pi_1^{i_1}\cdots
          \Pi_d^{i_d},
\end{equation}
where $0\le i_j\le d$ holds for $1\le j\le d$ and
$\Delta(i_1,\ldots, i_d):=i_2+i_4+\cdots+ i_{2\lfloor d/2\rfloor}$
denotes the sum of indices $i_j$ for which $j$ is an even number.
\end{proposition}

Finally we have the following expression of the polynomial $H_f\in
\fq[X_1,\ldots,X_{k+1}]$ associated to an arbitrary $f\in\fq[T]$ of
degree $k+d$ in terms of the polynomials $H_d$.
\begin{proposition}[{\cite[Proposition 2.3]{CaMaPr12}}]
\label{prop: formula para Hf}
Let $f:=T^{k+d}+f_{d-1}T^{k+d-1}+\cdots+f_0T^k$ be a polynomial of
$\fq[T]$ and let $H_f\in\fq[X_1,\ldots,X_{k+1}]$ be the polynomial
associated to $f$. Then the following identity holds:
\begin{equation} \label{eq:formula recursiva para H_d}
H_f=H_d+f_{d-1}H_{d-1}+\cdots+f_1H_1+f_0.
\end{equation}
\end{proposition}

\begin{remark}\label{rem: degree of Hd and H_f}
From Proposition \ref{prop: formula hipersuperficie H_0d} we easily
conclude that $H_d$ is a homogeneous polynomial of
$\fq[X_1,\ldots,X_{k+1}]$ of degree $d$ that can be expressed as a
polynomial in the elementary symmetric polynomials
$\Pi_1,\ldots,\Pi_d$. Further, $H_d$ is a monic element of
$\fq[\Pi_1,\ldots,\Pi_{d-1}][\Pi_d]$ of degree one in $\Pi_d$.
Combining these remarks and Proposition \ref{prop: formula para Hf}
we see that, for an arbitrary
$f:=T^{k+d}+f_{d-1}T^{k+d-1}+\cdots+f_0T^k\in\fq[T]$, the
corresponding polynomial $H_f\in \fq[X_1,\ldots,X_{k+1}]$ has degree
$d$ and is also a monic element of
$\fq[\Pi_1,\ldots,\Pi_{d-1}][\Pi_d]$ of degree one in $\Pi_d$.
 \end{remark}
%
%
\subsection{Estimates on the number of $\fq$-rational points and existence of deep holes}
Fix positive integers $d$ and $k$ such that $d < k$ and $q-1>k+d$,
and a polynomial $f:=T^{k+d}+f_{d-1}T^{k+d-1}+\cdots+ f_0T^k
\in\fq[T]$. Consider the hypersurface $V_f\subset\A^{k+1}$ defined
by the polynomial $H_f\in \fq[X_1,\ldots X_{k+1}]$ associated to
$f$. We shall establish a lower bound for the number of
$\fq$--rational points of $V_f$ without zero coordinates, and having
all the coordinates pairwise distinct. This will allow us to obtain
conditions on $q$, $d$ and $k$ which imply the nonexistence of deep
holes of the standard Reed--Solomon code $C$. For this purpose, we
show that the polynomial $H_f$ satisfies assumptions $({\sf A}_1)$
and $({\sf A}_2)$ of Section \ref{subsec: singular locus of V}.
\begin{lemma}\label{lemma: H_f satisfies A_1 and A_2}
The polynomial $H_f$ of \eqref{eq:polinomio f} satisfies assumptions
$({\sf A}_1)$--$({\sf A}_2)$ of Section \ref{subsec: singular locus
of V}.
\end{lemma}
\begin{proof}
Combining Proposition \ref{prop: formula para Hf} and Remark
\ref{rem: degree of Hd and H_f} we see that the homogeneous
component of highest degree of $H_f$ is $H_d$. According to Remark
\ref{rem: degree of Hd and H_f}, we may express $H_d$ as
$H_d=G_d(\Pi_1,\ldots,\Pi_d)$, where $G_d\in\fq[Z_1,\ldots,Z_d]$ is
of degree one in $Z_d$, and monic in $Z_d$. This implies that the
gradient $\nabla G_d$ of $G_d$ is of the form $\nabla G_d=(\partial
G_d/\partial Z_1,\ldots,\partial G_d/\partial Z_{d-1},1)$. In
particular, $\nabla G_d(\bfs z)\not=\bfs 0$ for any $\bfs z\in\A^d$,
which shows that $H_d$ satisfies assumption $({\sf A}_1)$, and thus
$H_f$ satisfies assumption $({\sf A}_2)$.

Further, let $G_f\in\fq[Z_1,\ldots,Z_d]$ be such that
$H_f=G_f(\Pi_1,\ldots,\Pi_d)$. From Remark \ref{rem: degree of Hd
and H_f} we conclude that  the gradient $\nabla G_f$ is also of the
form $\nabla G_f=(\partial G_f/\partial Z_1,\ldots,\partial
G_f/\partial Z_{d-1},1)$. It follows that $\nabla G_f(\bfs
z)\not=\bfs 0$ for any $\bfs z\in\A^d$, which implies that $H_f$
satisfies assumption $({\sf A}_1)$.
\end{proof}

The first step towards our main result is a lower bound on the
number of $\fq$--rational points of the hypersurface $V_f$. Lemma
\ref{lemma: H_f satisfies A_1 and A_2} shows that $H_f$ satisfies
the hypotheses of Theorem \ref{th: estimate |V_F| without Pi_m}. As
a consequence, we obtain the following result.
\begin{lemma}\label{lemma: nmb points V_f}
Suppose that $\mathrm{char}(\fq)$ does not divide $(k+1)k\cdots
(d-1)$. Then we have
$$
\big||V_f(\fq)|-q^k\big|\leq
q^{\frac{k+d-1}{2}}(1+q^{-1})\big((d-1)^{k-d+2}+
6(d+2)^{k+2}q^{-\frac{1}{2}}\big).$$
\end{lemma}
\begin{proof}
Observe that $H_f$ does not depend on $\Pi_{d+1},\ldots,\Pi_{k+1}$.
Therefore, the statement follows readily from Theorem \ref{th:
estimate |V_F| without Pi_m}.
\end{proof}

Next we obtain an upper bound on the number of $\fq$--rational
points of the hypersurface $V_f$ which are not useful in connection
with the existence of deep holes, namely those with a zero
coordinate or at least two equal coordinates. We start with the
first of these cases.
\begin{lemma} \label{lemma: zeros with a zero coordinate}
With hypotheses as in Lemma \ref{lemma: nmb points V_f}, let $N_1$
be the number of $\fq$--rational points of $V_f$ with a zero
coordinate, namely
$$N_1:=\Bigg|\bigcup_{i=1}^{k+1}V_f(\fq)\cap\{X_i=0\}\Bigg|.$$
Then
$$
N_1 \le
(k+1)\left(q^{k-1}+q^{\frac{k+d-2}{2}}(1+q^{-1})\big((d-1)^{k-d+1}
+6(d+2)^{k+1}q^{-\frac{1}{2}}\big)\right).
$$
\end{lemma}

\begin{proof}
Let $\bfs{x}:=(x_1,\ldots, x_{k+1})$ be a point of $V_f$ with a zero
coordinate. Assume without loss of generality that $x_{k+1}=0$. Then
$\bfs{x}$ is an $\fq$--rational point of the intersection
$W_{k+1}:=V_f\cap \{X_{k+1}=0\}$. Observe that $W_{k+1}$ is the
$\fq$--hypersurface of the linear space $\{X_{k+1}=0\}$ defined by
the polynomial $G_f(\Pi_{1}^{k},\ldots,\Pi_{d}^{k})$, where
$\Pi_{1}^{k},\ldots,\Pi_{d}^{k}$ are the first $d$ elementary
symmetric polynomials of $\fq[X_1,\ldots, X_{k}]$. As a consequence,
by Theorem \ref{th: estimate |V_F| without Pi_m} we obtain
$$
| |W_{k+1}(\fq)|-q^{k-1}| \le
q^{\frac{k+d-2}{2}}(1+q^{-1})\big((d-1)^{k-d+1} +
6(d+2)^{k+1}q^{-\frac{1}{2}}\big).
$$
Adding the upper bounds of the number of $\fq$--rational points of
the varieties $W_i:=V_f\cap\{X_i=0\}$ for $1\le i\le k+1$, the lemma
follows.
\end{proof}

Next we consider the number of $\fq$--rational points of $V_f$ with
two equal coordinates.
\begin{proposition} \label{prop:zeros with two equal coordinates}
With hypotheses as in Lemma \ref{lemma: nmb points V_f}, let $N_2$
be the number of $\fq$--rational points of $V_f$ with at least two
equal coordinates, namely
$$N_2:=\bigg|\bigcup_{1\le i<j\le k+1}V_f(\fq)\cap\{X_i=X_j\}\bigg|.$$
We have
$$
N_2 \leq
\frac{(k+1)k}{2}\Big(q^{k-1}+q^{\frac{k+d-2}{2}}(1+q^{-1})\big((d-1)^{k-d+1}
+6(d+2)^{k+1}q^{-\frac{1}{2}}\big)\Big).
$$
\end{proposition}
\begin{proof}
Let $\bfs{x}:=(x_1,\ldots, x_{k+1})\in V_f(\fq)$ be a point having
two distinct coordinates with the same value. Without loss of
generality we may assume that $x_k = x_{k+1}$ holds. Then $\bfs{x}$
is an $\fq$--rational point of the hypersurface $W_{k,k+1}\subset
\{X_k=X_{k+1}\}$ defined by the polynomial
$G_f(\Pi_{1}^*,\ldots,\Pi_{d}^*)\in \fq[X_1,\ldots,X_k]$, where
$\Pi_i^*:=\Pi_{i}(X_1,\ldots, X_k,X_k)$ is the polynomial of
$\fq[X_1,\ldots, X_k]$ obtained by substituting $X_k$ for $X_{k+1}$
in the $i$th elementary symmetric polynomial of $\fq[X_1,\ldots,
X_{k+1}]$. Observe that
\begin{equation}\label{eq:pi_i en x_k igual x_(k+1)}
      \Pi_{i}^* = \Pi_i^{k-1} + 2X_{k}\cdot \Pi_{i-1}^{k-1} + X_k^2 \cdot \Pi_{i-2}^{k-1}
\end{equation}
where $\Pi_j^l$ denotes the $j$th elementary symmetric polynomial of
$\fq[X_1,\ldots, X_l]$ for $1\le j\le d$ and $1\le l\le k+1$.

We claim that the singular locus of $\mathrm{pcl}(W_{k,k+1})$ and
the singular locus of $W_{k,k+1}$ at infinity have dimension at most
$d-1$ and $d-2$ respectively. Using \eqref{eq:pi_i en x_k igual
x_(k+1)} it can be proved that all the maximal minors of the
Jacobian matrix $(\partial\Pi^*_i/\partial X_j)_{1\le i\le d,1\le
j\le k}$ are equal, up to multiplication by nonzero constants, to
the corresponding minors of the Jacobian matrix
$(\partial\Pi^k_i/\partial X_j)_{1\le i\le d,1\le j\le k}$. Then the
proofs of Theorems \ref{th: singular locus for F_s without Pi(m-1),
Pi(m)}, \ref{th: pcl(V) is normal abs irred} and \ref{th: pcl(V_F)
is a normal hypersurface} go through with minor corrections and show
our claim. As a consequence, Theorem \ref{th: estimate |V_F| without
Pi_m} shows that
$$
|W_{k,k+1}(\fq)|\le  q^{k-1}+
q^{\frac{k+d-2}{2}}(1+q^{-1})\big((d-1)^{k-d+1}+
6(d+2)^{k+1}q^{-\frac{1}{2}}\big).
$$
We readily deduce the statement of the proposition.
\end{proof}
%
%
\subsubsection{Results of nonexistence of deep holes}
Now we are ready to prove the main results of this section. Fix $q$,
$k$ and $d\ge 3$ with $q-1>k+d$ and consider the standard
Reed--Solomon code $C$ of dimension $k$ over $\fq$. According to
\eqref{eq:sistema de las soluciones no deseadas}, a polynomial
$f:=T^{k+d}+f_{d-1}T^{k+d-1}+\cdots+f_{0}T^k$ does not generate a
deep hole of the code $C$ if and only if the corresponding
hypersurface $V_f\subset\A^{k+1}$ has an $\fq$--rational point with
nonzero, pairwise--distinct coordinates. Combining Lemmas
\ref{lemma: nmb points V_f} and \ref{lemma: zeros with a zero
coordinate} and Proposition \ref{prop:zeros with two equal
coordinates} we conclude that the number $N$ of such points
satisfies the following inequality:
\begin{equation}\label{eq:lower bound useful points}
  \begin{split}
  N\ge  q^k-\dfrac{(k+1)(k+2)}{2}q^{k-1}-(1+q^{-1})(d-1)^{k-d+1}q^{\frac{k+d-1}{2}}\bigg(d-1+
    \dfrac{(k+1)(k+2)}{2q^{\frac{1}{2}}}\bigg) \\
    -6(1+q^{-1})(d+2)^{k+1}q^{\frac{k+d-2}{2}}\bigg(d+2+
    \dfrac{(k+1)(k+2)}{2q^{\frac{1}{2}}}\bigg).
  \end{split}
\end{equation}
Therefore, the polynomial $f$ does not generate a deep hole of the
code $C$ if the right--hand side of (\ref{eq:lower bound useful
points}) is a positive number.

Suppose that $q$, $k$ and $d\ge 3$ satisfy the following conditions:
 \begin{equation}\label{eq:first condition on q}
   q>(k+1)^2,\quad k>3d.
 \end{equation}
Since $k\geq 10$,  it follows that $\frac{3}{4}(k+1)(k+2)\leq
(k+1)^{2}<q$ holds. Therefore, we have
$q-\frac{1}{2}(k+1)(k+2)>{q}/{3}$, which implies
  $$
    q^k-\dfrac{(k+1)(k+2)}{2}q^{k-1}=q^{k-1}\bigg(q-\dfrac{(k+1)(k+2)}{2}\bigg)> \frac{q^k}{3}.
  $$
Hence, the right--hand side of (\ref{eq:lower bound useful points})
is positive if the following condition holds:
\begin{align}
 \label{eq: first inequality of the proof of lower bound useful points}
\dfrac{q^k}{3}  & \geq
(1+q^{-1})(d-1)^{k-d+1}q^{\frac{k+d-1}{2}}\bigg(d-1+
\dfrac{(k+1)(k+2)}{2q^{\frac{1}{2}}}\bigg) \\
&\quad +\, 6(1+q^{-1})(d+2)^{k+1}q^{\frac{k+d-2}{2}}\bigg(d+2+
\dfrac{(k+1)(k+2)}{2q^{\frac{1}{2}}}\bigg).\notag
\end{align}
Since  $k+1<q^{\frac{1}{2}}$, we conclude that \eqref{eq: first
inequality of the proof of lower bound useful points} can be
replaced by the condition
\begin{align*}
\dfrac{q^k}{3}\ge&(1+q^{-1})(d-1)^{k-d+1}q^{\frac{k+d-1}{2}}\left(d-1+\frac{k+2}{2}\right)\\
&+ 6(1+q^{-1})(d+2)^{k+1}q^{\frac{k+d-2}{2}}
\left(d+2+\frac{k+2}{2}\right).
\end{align*}
As $d\leq\frac{k-1}{3}$, we obtain $d+2+\frac{k+2}{2}\leq k+1$. We
conclude that the right--hand side of \eqref{eq:lower bound useful
points} is positive if
$$
\dfrac{q^k}{3}\ge
(1+q^{-1})(d-1)^{k-d+1}(k-2)q^{\frac{k+d-1}{2}}+6(1+q^{-1})(d+2)^{k+1}(k+1)q^{\frac{k+d-2}{2}},
$$
or equivalently, if
$$q^k\ge 3(1+q^{-1})(d-1)^{k-d+1}(k-2)q^{\frac{k+d-1}{2}}
+18(1+q^{-1})(d+2)^{k+1}(k+1)q^{\frac{k+d-2}{2}},
$$
holds. Furthermore, this condition is in turn implied by the
following conditions:
$$\dfrac{q^k}{8}\ge 3(1+q^{-1})(d-1)^{k-d+1}(k-2)q^{\frac{k+d-1}{2}},\quad
\dfrac{7 q^k}{8}\ge 18(1+q^{-1})(d+2)^{k+1}(k+1)
q^{\frac{k+d-2}{2}},$$
which can be rewritten as
\begin{equation}\label{eq:first condition on q bis}
q^k\ge 24(1+q^{-1})(d-1)^{k-d+1}(k-2)q^{\frac{k+d-1}{2}},\  q^k\ge
\frac{144}{7}(1+q^{-1})(d+2)^{k+1}(k+1)q^{\frac{k+d-2}{2}}.
\end{equation}
The first inequality is implied by the following one:
$$q\ge (25(k-2))^{\frac{2}{k-d+1}}(d-1)^2.$$
From \eqref{eq:first condition on q} one easily concludes that
$3(k-d+1)\ge 2k+4$ holds. Since the function $k\mapsto
\big(25(k-2)\big){}^{3/(k+2)}$ is decreasing, taking into account
that $k\ge 10$ holds we deduce that a sufficient condition for the
fulfillment of the inequality above is
\begin{equation}\label{eq:first part first condition}
     q>4d^2.
\end{equation}
Next we consider the second inequality of (\ref{eq:first condition
on q bis}). First, we observe that (\ref{eq:first condition on q
bis}) holds if
\begin{equation}\label{eq:second part first condition inexplicit}
q>(21(k+1))^{\frac{2}{k-d+2}}\Big(\frac{d+2}{d}\Big)^{2+\frac{2d-2}{k-d+2}}d^{\,2+\frac{2d-2}{k-d+2}}.
\end{equation}
From \eqref{eq:first condition on q} we deduce $3(k-d+2)\ge 2k+7$.
Taking into account that the function $k\mapsto
\big(21(k+1)\big){}^{6/(2k+7)}$ is decreasing, as $k \geq 10$  we
see that \eqref{eq:second part first condition inexplicit} is
satisfied if the following condition holds:
\begin{equation}\label{eq:second part first condition}
q>12\,d^{\,2+\frac{2d-2}{k-d+2}}.
\end{equation}
Combining (\ref{eq:first condition on q}), (\ref{eq:first part first
condition}) and (\ref{eq:second part first condition}) we deduce the
following sufficient condition for the nonexistence of deep holes.
\begin{theorem}\label{th: deep holes - main}
Suppose that $\mathrm{char}(\fq)$ does not divide $(k+1)k\cdots
(d-1)$. Let  $k,d$ be integers with $k>3d$, $d\ge 3$ and $q-1>k+d$,
and let ${C}$ be the standard Reed--Solomon code of dimension $k$
over $\fq$. Let $\bfs{w}$ be a word generated by a polynomial
$f\in\fq[T]$ of degree $k+d$. Given a real number $\epsilon>0$, if
the conditions
$$
q>\max\{(k+1)^2,12\,d^{\,2+\epsilon}\},\quad k\ge
(d-2)\Big(\frac{\,2}{\epsilon}+1\Big)
$$
hold, then $\bfs{w}$ is not a deep hole of ${C}$.
\end{theorem}

In \cite{LiWa08} the existence of deep holes is considered. Using
the Weil estimate for certain character sums as in \cite{Wan97}, the
authors prove that, for $k,d$ as in Theorem \ref{th: deep holes -
main}, if
\begin{equation}\label{eq: DH li-wan}
q>\max\{(k+1)^2,d^{\,2+\epsilon}\}\textrm{ and }
k>\Big(\frac{2}{\epsilon}+1\Big)d+ \frac{8}{\epsilon}+2
\end{equation}
holds for a constant $\epsilon>0$, then no word $\bfs{w}$ generated
by a polynomial $f\in \fq[T]$ of degree $k+d<q-1$ is a deep hole of
the standard Reed--Solomon code of dimension $k$ over $\fq$ (see
\cite[Theorem 1.4]{LiWa08}). On the other hand, in \cite[Theorem
1.5]{CaMaPr12} it is shown that the conditions
\begin{equation}\label{eq: DH ca-ma-pr}
q>\max\{(k+1)^2,14\,d^{\,2+\epsilon}\},\quad k\ge
d\Big(\frac{\,2}{\epsilon}+1\Big)
\end{equation}
suffice to guarantee that no word $\bfs{w}$ generated by a
polynomial $f\in\fq[T]$ of degree $k+d$ is a deep hole of the
standard Reed--Solomon code of dimension $k$ over $\fq$.

Our result constitutes an improvement of those of \cite{LiWa08} and
\cite{CaMaPr12}, as can be readily deduced by comparing Theorem
\ref{th: deep holes - main} with \eqref{eq: DH li-wan} and
\eqref{eq: DH ca-ma-pr}. Nevertheless, as the ``main'' exponents in
these results are similar, we would like to stress here the
methodological aspect. Our approach is based on the estimates on the
number of $\fq$-rational zeros of symmetric polynomials of Theorem
\ref{th: estimate |V_F| without Pi_m}, and in this sense is similar
to the methodology of \cite{CaMaPr12}. Nevertheless, as the
conclusions of Theorem \ref{th: estimate |V_F| without Pi_m} apply
to a wider family of hypersurfaces than the corresponding result of
\cite{CaMaPr12}, our conclusion is stronger than that of
\cite{CaMaPr12}.
%
%
\appendix
\section{The variety of successive generic subdiscriminants}
\label{section: appendix}
The purpose of this appendix is to show the assertion on the
behavior of the variety defined by the vanishing of the first
successive generic subdiscriminants of degree $m$.

Let $\K$ be a field and $\K[X_1,\ldots,X_m]$ the ring of
multivariate polynomials in indeterminates $X_1,\ldots,X_m$ and
coefficients in $\K$. For $\bfs X:=(X_1,\ldots,X_m)$ and
$$f_{\bfs X}:=(T-X_1)\cdots(T-X_m),$$
the (generic) $k$th subdiscriminant $\mathrm{sDisc}_k(f_{\bfs X})$
of $f_{\bfs X}$ for $0\le k\le m-1$ is defined in the following way:
$$\mathrm{sDisc}_k(f_{\bfs X}):=\sum_{\stackrel{\scriptstyle
I\subset\{1,\ldots,m\}}{\#I=m-k}}\prod_{\stackrel{\scriptstyle
i,j\in I}{i<j}}(X_i-X_j)^2.$$
In particular, for $k=0$ we obtain the classical discriminant
$\mathrm{sDisc}_0(f_{\bfs X})=\mathrm{Disc}(f_{\bfs X}):=\prod_{1\le
i<j\le m}(X_i-X_j)^2$.

Observe that $\mathrm{sDisc}_k(f_{\bfs X})$ is a homogeneous element
of $\K[\bfs X]$ of degree $(m-k)(m-k-1)$. It is clear that
$\mathrm{sDisc}_k(f_{\bfs X})$ is a symmetric polynomial of $\K[\bfs
X]$, and therefore it can be expressed as a polynomial in the
elementary symmetric polynomials $\Pi_1,\ldots,\Pi_m$ of $\K[\bfs
X]$. In what follows, we shall consider $\mathrm{sDisc}_0(f_{\bfs
X}),\ldots, \mathrm{sDisc}_{k-1}(f_{\bfs X})$ as elements of the
polynomial ring $\K[\bfs\Pi_{m-k}][\Pi_{m-k+1},\ldots,\Pi_{m}]$,
where $\bfs\Pi_l:=(\Pi_1,\ldots,\Pi_l)$ for $1\le l\le m$.

We shall consider weights associated to
$\K[\bfs\Pi_{m-k}][\Pi_{m-k+1},\ldots,\Pi_{m}]$ or $\K[\bfs\Pi_m]$.
For ${\sf R}:=\K[\bfs\Pi_{m-k}][\Pi_{m-k+1},\ldots,\Pi_{m}]$ or
${\sf R}:=\K[\bfs\Pi_m]$, a {\em weight} on ${\sf R}$ is a function
$\wt:{\sf R}\to\N_0$ defined by setting
$\wt(\Pi_{m-k+1}^{\alpha_{m-k+1}}\cdots\Pi_m^{\alpha_m}):=\beta_{m-k+1}
\alpha_{m-k+1}+\cdots+ \beta_m\alpha_m$ for any
$\alpha_{m-k+1},\ldots,\alpha_m\in\Z_{\ge 0}^m$ or
$\wt(\Pi_1^{\alpha_1}\cdots\Pi_m^{\alpha_m}):=\beta_1\alpha_1+\cdots+
\beta_m\alpha_m$ for any $\alpha_1,\ldots,\alpha_m\in\Z_{\ge 0}^m$.
The weight $\wt(F)$ of an arbitrary $F\in{\sf R}$ is the highest
weight of all the monomials arising with a nonzero coefficient in
the dense representation of $F$. An element $F\in{\sf R}$ is said to
be {\em weighted homogeneous} with respect to the weight $\wt$
defined above, or simply $\wt$-weighted homogeneous, if all its
terms have the same weight. Any polynomial $F\in{\sf R}$ can be
uniquely written as a sum of weighted homogeneous polynomials
$F=\sum_{i}F_i$, where each $F_i$ is weighted homogeneous with
$\wt(F_i)=i$. The polynomials $F_i$ are called the {\em weighted
homogeneous components} of $F$. In particular, we denote by
$F^{\wt}$ the component of highest weight of any $F\in {\sf R}$.

Now we are able to state and prove the main result of this appendix.
\begin{theorem}\label{th: quotient by subdiscs is integral}
With notations as above, if $\textrm{char}(\K)$ does not divide
$m(m-1)\cdots(m-k+1)$, then the ring extension
\begin{align*}
\K[\bfs\Pi_{m-k}]\hookrightarrow {\sf
R}_k:=\K[\bfs\Pi_m]/(\mathrm{sDisc}_0(f_{\bfs X}),\ldots,
\mathrm{sDisc}_{k-1}(f_{\bfs X}))
\end{align*}
is integral.
\end{theorem}
\begin{proof}
We claim that there exists a monomial order in $\K[\bfs\Pi_m]$ such
that the leading monomial of $\mathrm{sDisc}_j(f_{\bfs X})$ in such
an order is $m(m-j)^{m-j-1}\Pi_{m-j}^{m-j-1}$ for $0\le j\le k-1$.

Assuming that the claim holds, we prove the theorem. For such a
monomial order, the leading monomials of the subdiscriminants
$\mathrm{sDisc}_j(f_{\bfs X})$ are relatively prime. This implies
that $\mathrm{sDisc}_0(f_{\bfs X}),\ldots,$
$\mathrm{sDisc}_{k-1}(f_{\bfs X})$ form a Gr\"obner basis of the
ideal they defined. It follows that the set
$$\big\{\Pi_{m-k+1}^{\alpha_{m-k+1}}\cdots\Pi_m^{\alpha_m}:
0\le\alpha_{m-j}\le m-j-2\textrm{ for }0\le j\le k-1\big\}$$
forms a basis of ${\sf R}_k$ as a $\K[\bfs\Pi_{m-k}]$-module,
showing thus the theorem.

Now we show the claim. The order of the claim will be obtained by
considering a number of (partial) weight orders, which successively
refine the previous ones. Given $\bfs\gamma\in\R^m$, the weight
order $>_{\bfs\gamma}$ determined by $\bfs\gamma$ is defined by
setting $\bfs\alpha>_{\bfs\gamma}\bfs\beta$ if
$\bfs\alpha\cdot\bfs\gamma>\bfs\beta\cdot\bfs\gamma$ for any
$\bfs\alpha,\bfs\beta\in\Z_{\ge 0}^m$, where $\cdot$ denotes the
standard inner product of $\R^m$.

Consider the weight $\wt_e:\K[\bfs\Pi_m]\to\N_0$ defined by setting
$\wt_e(\Pi_1^{\alpha_1}\cdots\Pi_m^{\alpha_m}):=\sum_{k=1}^mk\cdot\alpha_k$
for any $\alpha_1,\ldots,\alpha_m\in\Z_{\ge 0}$. Up to a power of
$m$, $\mathrm{sDisc}_j(f_{\bfs X})$ can be expressed as the $j$th
subresultant of $f_{\bfs X}$ and $\partial f_{\bfs X}/\partial T$
(see, e.g., \cite[Proposition 4.27]{BaPoRo06}). With this
expression, it is easy to deduce that $\mathrm{sDisc}_j(f_{\bfs X})$
is  $\wt_e$-weighted homogeneous of weight $(m-j)(m-j-1)$ for $0\le
j<k$.

Now we introduce the first weight to obtain the monomial order of
the claim. This weight $\wt_k:{\sf R}\to\N_0$ is defined by setting
$\wt_k(\Pi_{m-j}):=m-j$ for $0\le j<k$ and $\wt_k(\Pi_{m-j}):=0$
otherwise. In other words,
$$\wt_k(\Pi_1^{\alpha_1}\cdots\Pi_m^{\alpha_m}):=\sum_{j=0}^{k-1}(m-j)
\alpha_{m-j}\textrm{ for any }\alpha_1,\ldots,\alpha_m\in\Z_{\ge
0}.$$
The fact that $\mathrm{sDisc}_j(f_{\bfs X})$ is  $\wt_e$-weighted
homogeneous of weight $(m-j)(m-j-1)$ for $0\le j<k$ implies that
$$\wt_k(\mathrm{sDisc}_j(f_{\bfs X}))\le (m-j)(m-j-1).$$
Now we show that this inequality is actually an equality. Let
$A_1,\ldots,A_m$ be new indeterminates over $\K$, and let
$$f_{gen}:=T^m+A_1T^{m-1}+\cdots+A_m,\quad
f_j:=T^m+A_{m-j}T^j+\cdots+A_m\quad(0\le j<k).$$
We define a weight $\wt_{\bfs A}$ by setting $\wt_{\bfs A}(A_i):=i$
for any $i$, and weights $\wt_{\bfs A_j}$ for $0\le j<k$ by setting
$\wt_{\bfs A_j}(A_{m-i}):=m-i$ for $0\le i\le j$ and $\wt_{\bfs
A_j}(A_{m-i}):=0$ otherwise. Then we may easily reexpress our
assertion in terms of weights $\wt_{\bfs A}$ and $\wt_{\bfs A_j}$.
It is clear that
$$
\wt_{\bfs A_j}(\mathrm{sDisc}_j(f_j))\le
\wt_s(\mathrm{sDisc}_j(f_{\bfs X}))\le\wt_{\bfs
A}(\mathrm{sDisc}_j(f_{gen}))=(m-j)(m-j-1).
$$
Therefore, it suffices to show that
\begin{align*}
\wt_{\bfs A_j}\mathrm{sDisc}_j(f_j)=(m-j)(m-j-1).
\end{align*}

Observe that $f_j=f_j'\cdot \frac{T}{m}+r_j$, where $f_j'$ is the
derivative of $f_j$ with respect to $T$ and
\begin{equation}\label{eq: definition r_j}
r_j:=\frac{m-j}{m}A_{m-j}T^j+\frac{m-j+1}{m}A_{m-j+1}T^{j-1}+\cdots+A_m.
\end{equation}
According to \cite[Lemma 7.1]{GeCzLa92},
\begin{align}
\mathrm{sDisc}_j(f_j)=\mathrm{sRes}_j(f_j,f_j')
&=(-1)^{(m-j)(m-j-1)}m^{m-j}\mathrm{sRes}_j(f_j',r_j)\notag\\
&=m^{m-j}\mathrm{sRes}_j(f_j',r_j)\notag\\
&=m^{m-j}(\mbox{$\frac{m-j}{m}A_{m-j}$})^{m-j-1}\notag\\
&=m(m-j)^{m-j-1}A_{m-j}^{m-j-1},
\label{eq: expression dominant term subdisc}
\end{align}
where $\mathrm{sRes}_j$ denotes the $j$th subresultant with respect
to $T$. This in particular shows the assertion on the weight $\wt_k$
of the subdiscriminants.

For $0\le j<k$, we consider the homogeneous $\wt_k$-weighted
component $(\mathrm{sDisc}_j(f_{\bfs X}))^{\wt_k}$ of highest
$\wt_k$-weight of $\mathrm{sDisc}_j(f_{\bfs X})$, namely the
homogeneous $\wt_k$-weighted component of $\wt_k$-weight
$(m-j)(m-j-1)$. According to \eqref{eq: expression dominant term
subdisc}, the term $m(m-j)^{m-j-1}\Pi_{m-j}^{m-j-1}$ occurs in this
sum for $0\le j<k$. From the fact that $\mathrm{sDisc}_j(f_{\bfs
X})$ is $\wt_e$-weighted homogeneous of $\wt_e$-weight
$(m-j)(m-j-1)$, and the definition of $\wt_k$, we conclude that all
the nonzero monomials arising in $(\mathrm{sDisc}_j(f_{\bfs
X}))^{\wt_k}$ are monomials in $\Pi_{m-k+1},\ldots,\Pi_m$. Next we
refine the weight order $\wt_k$ by the weight order $\wt_0$ defined
by the linear form $\bfs\gamma_0:=(0,\ldots,0,-1,\ldots,-1)$, where
there are $m-k$ zero entries followed by $k$ entries equal to $-1$.
In other words, the terms in  $(\mathrm{sDisc}_j(f_{\bfs
X}))^{\wt_k}$ of highest $\wt_0$-weight are the monomials of least
degree.

As $(\mathrm{sDisc}_0(f_{\bfs X}))^{\wt_k}$ consists of a sum of
monomials of the form
$c_{\bfs\alpha}\Pi_{m-k+1}^{\alpha_{m-k+1}}\cdots\Pi_m^{\alpha_m}$
with $(m-k+1)\alpha_{m-k+1}+\cdots+m{\alpha_m}=m(m-1)$, taking into
account that the monomial $m^m\Pi_m^{m-1}$ arises in such a sum we
conclude that
$$((\mathrm{sDisc}_0(f_{\bfs X}))^{\wt_k,\wt_0}:=((\mathrm{sDisc}_0
(f_{\bfs X}))^{\wt_k})^{\wt_0}=m^m\Pi_m^{m-1}.$$

Next we argue that the term $m(m-j)^{m-j-1}\Pi_{m-j}^{m-j-1}$ occurs
in $((\mathrm{sDisc}_j(f_{\bfs X}))^{\wt_k,\wt_0}$ for $1\le j<k$.
The obvious isomorphism from $\K[\Pi_{m-k+1},\ldots,\Pi_m]$ to
$\K[A_{m-k+1},\ldots,A_m]$ maps $(\mathrm{sDisc}_j(f_{\bfs
X}))^{\wt_k}$ to $\mathrm{sDisc}_j(f_{k-1})$ for $1\le j<k$.
Therefore, it suffices to show that the term
$m(m-j)^{m-j-1}A_{m-j}^{m-j-1}$ occurs in
$(\mathrm{sDisc}_j(f_{k-1}))^{\wt_0}$ for $1\le j<k$. The
homogenization of $(\mathrm{sDisc}_j(f_{k-1}))^{\wt_0}$ with
homogenizing variable $A_0$ is
\begin{equation}\label{eq: disc_j as a subres}
((\mathrm{sDisc}_j(f_{k-1}))^{\wt_0})^h=
\mathrm{sRes}_j((f_{k-1})^h,(f_{k-1}')^h)=
(mA_0)^{m-k+1}\mathrm{sRes}_j((f_{k-1}')^h,r_{k-1}),
\end{equation}
where
\begin{align*}
(f_{k-1})^h&:=A_0T^m+A_{m-k+1}T^{k-1}+\cdots+A_m,\\
(f_{k-1}')^h&:=mA_0T^{m-1}+(k-1)A_{m-k+1}T^{k-2}+\cdots+A_{m-1},
\end{align*}
and $r_{k-1}$ is defined in \eqref{eq: definition r_j}. Denote
$C_{m-j}:=\frac{m-j}{m}A_{m-j}$ for $0\le j<k$ and $B_j:=(m-j)A_j$
for $j=0, m-k+1,\ldots,m-1$, so that
$$
r_{k-1}=C_{m-k+1}T^{k-1}+\cdots+C_m,\quad
(f_{k-1}')^h:=B_0T^{m-1}+B_{m-k+1}T^{k-2}+\cdots+B_{m-1}.$$
Note that both $C_j$ and $B_j$ are nonzero multiples in $\K$ of
$A_j$ for any $j$. The corresponding Sylvester matrix
$\mathrm{Syl}_0:=\mathrm{Syl}((f_{k-1}')^h,r_{k-1})$ is
$$\mathrm{Syl}_0=\left(
    \begin{array}{ccccccccc}
             B_0   & 0           & \cdots & \!\!\!  B_{m-k+1} \!\!\!    & \cdots & B_{m-1} \\
                              & \ddots              & \ddots &  &\ddots & & \ddots\\
                              &                  & B_0 & 0 &\cdots& \!\!\!B_{m-k+1}  \!\!\!   & \cdots & B_{m-1} \\
     C_{m-k+1}\!\!\! &      \cdots            & \cdots  &  C_m     &  \\
                              &\!\!\! C_{m-k+1} \!\!\!       & \cdots & \cdots  & C_m \\
                              &               & \ddots &  & & \ddots\\
                              &                  &        &   \!\!\! C_{m-k+1}\!\!\! &      \cdots      & \cdots      &   C_m     \\
                              &                  &        &  & \!\!\! C_{m-k+1}\!\!\! &      \cdots      & \cdots      &   C_m     \\
    \end{array}
  \right),
$$
where there are $k-1$ rows with $B$'s and $m-1$ rows with $C$'s. The
sum $(\mathrm{sDisc}_j(f_{k-1}))^{\wt_0}$ of the least-degree terms
of $\mathrm{sDisc}_j(f_{k-1})$ is, up to a nonzero multiple in $\K$,
the coefficient of $B_0^{k-1}$ in the determinant of
$\mathrm{Syl}_0$, which equals $C_m^{m-1}=A_m^{m-1}$, as already
shown.

Now, the corresponding expression for the $j$th discriminant
$((\mathrm{sDisc}_j(f_{\bfs X}))^{\wt_k,\wt_0}$ is, up to a nonzero
multiple in $\K$, the determinant of the submatrix $\mathrm{Syl}_j$
of $\mathrm{Syl}_0$ obtained by deleting the $j$ last rows of $B$'s,
the $j$ last rows of $C$'s, and the last $2j$ columns of
$\mathrm{Syl}_0$. The term $B_0$ arises exactly in the first $k-1-j$
diagonal entries of $\mathrm{Syl}_j$. Therefore, the coefficient of
$B_0^{k-1-j}$ in the expansion of $\det(\mathrm{Syl}_j)$ as a
polynomial in $B_0,\ldots,B_{m-1}$ equals, up to a nonzero
$\K$-multiple, the sum $(\mathrm{sDisc}_j(f_{k-1}))^{\wt_k}$, and
equals the determinant of the $(m-j-1)\times(m-j-1)$-matrix
$$T_j:=\left(
             \begin{array}{cccc}
               C_{m-j} & \!\!\!C_{m-j+1}\!\!\! & \cdots &\!\!\! C_{2(m-j-1)}\\
               C_{m-j-1} & C_{m-j} & \ddots & \vdots \\
               \vdots &\ddots  & \ddots & C_{m-j+1} \\
               C_2 & \cdots & \!\!\! C_{m-j-1}\!\!\!& C_{m-j} \\
             \end{array}
           \right),
$$
with the convention that $C_l:=0$ for $l>m$ or $l<m-k+1$. In
particular, the term
$B_0^{k-1-j}C_{m-j}^{m-1-j}=(mA_0)^{k-1-j}\big(\frac{m-j}{m}A_{m-j}\big)^{m-1-j}$
arises in the monomial expansion of $\det(\mathrm{Syl}_j)$.
Multiplying the term
$m^{k-1-j}\big(\frac{m-j}{m}A_{m-j}\big)^{m-1-j}$ by $m^{m-k+1}$,
according to \eqref{eq: disc_j as a subres}, we conclude that the
term $m(m-j)^{m-j-1}A_{m-j}^{m-j-1}$ arises in the sum
$(\mathrm{sDisc}_j(f_{k-1}))^{\wt_0}$.

Now we successively refine this (partial) monomial order with weight
orders $\wt_1,\ldots,\wt_{k-1}$ so that
\begin{equation}\label{eq: sdisc_j in terms of weights}
(\mathrm{sDisc}_j(f_{k-1}))^{\wt_0,\wt_1,\ldots,\wt_j}=
m(m-j)^{m-j-1}A_{m-j}^{m-j-1}\end{equation}
for $1\le j<k$. The weight $\wt_i$ is defined by setting
$\wt_i(A_{m-l}):=1$ for $l\le i$ and $0$ otherwise. Observe that
$\wt_j$ orders monomials taking into account degree only in
$\Pi_{m-k+1},\ldots,\Pi_{m-j}$ and
$\wt_j\circ\cdots\circ\wt_1=\wt_j$ for any $j$. As $\det T_j$
essentially represents $(\mathrm{sDisc}_j(f_{k-1}))^{\wt_0}$, up to
a nonzero multiple in $\K$, for $i<j$ we have that composition with
$\wt_1,\ldots,\wt_i$ yields the matrix
$$T_j^i:=\left(
             \begin{array}{cccccc}
               C_{m-j} & \cdots & C_{m-i} &\\
               C_{m-j-1} & C_{m-j} &  & \ddots \\
               \vdots &\ddots  & \ddots & & C_{m-i}\\
               \vdots &\ddots  & \ddots & & \vdots\\
               C_2 & \cdots & &\!\!\! C_{m-j-1}\!\!\!&\! C_{m-j} \\
             \end{array}
           \right).$$
In particular, $T_j^j$ is lower triangular with $\det
T_j^j=C_{m-j}^{m-1-j}$, which readily implies \eqref{eq: sdisc_j in
terms of weights} and completes the proof.
\end{proof}

\bibliographystyle{amsalpha}
\bibliography{refs1,finite_fields,Newref}
\end{document}

\section{Polynomials $F_1,\ldots,F_s$ not depending on
$\Pi_{m-k+1},\ldots,\Pi_m$}

\begin{corollary}\label{coro: singular locus for F_s without Pi(m-k+1),...,Pi(m)}
Let $k\ge 1$ and $F_1=0,\ldots,F_s=0$ be polynomials as in
\eqref{def: f} not depending on $\Pi_{m-k+1},\ldots,\Pi_m$. Then the
singular locus $\Sigma$ of $V:=V(\bfs F_s)\subset\A^m$ has dimension
at most $m-k-s-1\le m-s-2$.
\end{corollary}
\begin{proof}
According to Lemma \ref{lemma: singular locus F_s not depending on
Pi_(m-k+1),..,Pi_m}, we have the inclusion
\begin{equation}\label{eq: singular locus in terms of partitions}
\Sigma\subset\bigcup_{\mathcal{L}_{\mathcal{I}_t}}
\mathcal{L}_{\mathcal{I}_t}\cap V,\end{equation}
where $\mathcal{I}_t:=\{I_1,\ldots,I_t\}$ runs over all the
partitions of $\{1,\ldots,m\}$ into $t\le m-k-1$ nonempty subsets
$I_j\subset \{1,\ldots,m\}$ and $\mathcal{L}_{\mathcal{I}_t}$ is the
linear variety
$$
\mathcal{L}_{\mathcal{I}_t}:=
\mathrm{span}(\bfs{v}^{(I_1)},\ldots,\bfs{v}^{(I_t)})
$$
spanned by the vectors $\bfs{v}^{(I_j)}:=
(v_1^{(I_j)},\ldots,v_{m}^{(I_j)})\in\{0,1\}^m$ defined by
$v_l^{(I_j)}:=1$ iff $l\in I_j$. 

Let $\mathcal{C}$ be an irreducible component of $\Sigma$. As a
consequence of \eqref{eq: singular locus in terms of partitions},
there exists $\mathcal{I}_t$ as above such that $\mathcal{C}\subset
\mathcal{L}_{\mathcal{I}_t}\cap V$. It suffices to show that
$\mathcal{L}_{\mathcal{I}_t}\cap V$ has dimension at most
***. Without loss of generality, we may assume
that $\mathcal{I}_t$ is maximal with this property, namely there is
no partition $\mathcal{I}_{t'}$ with $\mathcal{I}_{t'}\subset
\mathcal{I}_t$ for which $\mathcal{C}\subset
\mathcal{L}_{\mathcal{I}_{t'}}\cap V$.

Assume without loss of generality that $i\in I_i$ for $1\le i\le t$.
Consider the mapping
$$\begin{array}{rcl}
\Phi_{\mathcal{I}_t}:\A^t&\to&\A^t,\\
(x_1,\ldots,x_t)&\mapsto&
\bfs\Pi_{m-k}(x_1\bfs{v}^{(I_1)},\ldots,x_t\bfs{v}^{(I_t)}),
\end{array}$$
where $\bfs\Pi_{m-k}:=(\Pi_1,\ldots,\Pi_{m-k})$. The set
$\mathcal{L}_{\mathcal{I}_t}\cap V$ is isomorphic to the set
$$\{\bfs x\in\A^t:\bfs G_s(\Phi_{\mathcal{I}_t}(\bfs x))=\bfs 0\}.$$

Denote $L_{\mathcal{I}_t}:\A^t\to\A^m$,
$L_{\mathcal{I}_t}(x_1,\ldots,x_t):=(x_1\bfs{v}^{(I_1)},\ldots,x_t\bfs{v}^{(I_t)})$.
Observe that $\Phi_{\mathcal{I}_t}=\Pi_{m-k}\circ
L_{\mathcal{I}_t}$. We have
\begin{align*}
J(\bfs G_s&\circ \Phi_{\mathcal{I}_t})(\bfs x)= J(\bfs
G_s)(\Phi_{\mathcal{I}_t}(\bfs x))\cdot
J(\bfs\Pi_{m-k})(L_{\mathcal{I}_t}(\bfs x))\cdot
J(L_{\mathcal{I}_t})\\[1ex]
&=J(\bfs G_s)(\Phi_{\mathcal{I}_t}(\bfs x))\cdot B_{m-k}^*\cdot
A_{m-k}^*(L_{\mathcal{I}_t}(\bfs x))\cdot\left(
                            \begin{array}{ccccc}
                              \bfs{v}^{(I_1)} & \ldots  & \bfs{v}^{(I_t)}   \\
                            \end{array}
                          \right)\\
&=J(\bfs G_s)(\Phi_{\mathcal{I}_t}(\bfs x))\cdot
B_{m-k}^*(L_{\mathcal{I}_t}(\bfs x))\cdot \left(
                            \begin{array}{cccccccc}
                              n_1 & n_2 & \ldots & n_t   \\
                              n_1x_1 & n_2x_2 & \ldots & n_tx_{m-3}   \\
                              \vdots & \vdots &  &  \vdots     \\
                               n_1x_1^{m-k-1}  & n_2x_2^{m-k-1} & \ldots & n_tx_t^{m-k-1}
                            \end{array}
                          \right),
\end{align*}
where $n_i:=\#(I_i)$ for $1\le i\le t$. By the maximality of $t$,
there exists a nonempty Zariski open subset $\mathcal{U}$ of
$\mathcal{C}$ such that for any $\bfs x\in\mathcal{U}$ we have
$x_i\not= x_j$ for $1\le i<j\le t$. It follows that the polynomials
$(\bfs G_s\circ \Phi_{\mathcal{I}_t})(\bfs X)$ intersect
transversally on $\mathcal{U}$. This implies that $\mathcal{C}$ has
dimension at most $t-s\le m-k-s-1$.
We readily deduce the corollary
\end{proof}